\documentclass[final,onefignum,onetabnum]{siamart220329}

\usepackage{braket,amsfonts,bm,amssymb}

\usepackage{array}

\usepackage[caption=false]{subfig}
\usepackage{pgfplots}
\pgfplotsset{compat=newest}

\newsiamthm{claim}{Claim}
\newsiamremark{remark}{Remark}
\newsiamremark{hypothesis}{Hypothesis}
\crefname{hypothesis}{Hypothesis}{Hypotheses}
\newsiamthm{example}{Example}
\newsiamthm{assumption}{Assumption}

\usepackage{algorithmic}

\usepackage{graphicx,epstopdf}

\Crefname{ALC@unique}{Line}{Lines}

\usepackage{amsopn}

\usepackage{xspace}
\usepackage{bold-extra}
\usepackage[most]{tcolorbox}

\colorlet{texcscolor}{blue!50!black}
\colorlet{texemcolor}{red!70!black}
\colorlet{texpreamble}{red!70!black}
\colorlet{codebackground}{black!25!white!25}

\lstdefinestyle{siamlatex}{%
  style=tcblatex,
  texcsstyle=*\color{texcscolor},
  texcsstyle=[2]\color{texemcolor},
  keywordstyle=[2]\color{texemcolor},
  moretexcs={cref,Cref,maketitle,mathcal,text,headers,email,url},
}

\tcbset{%
  colframe=black!75!white!75,
  coltitle=white,
  colback=codebackground, % bottom/left side
  colbacklower=white, % top/right side
  fonttitle=\bfseries,
  arc=0pt,outer arc=0pt,
  top=1pt,bottom=1pt,left=1mm,right=1mm,middle=1mm,boxsep=1mm,
  leftrule=0.3mm,rightrule=0.3mm,toprule=0.3mm,bottomrule=0.3mm,
  listing options={style=siamlatex}
}

\DeclareTotalTCBox{\code}{ v O{} }
{ %fontupper=\ttfamily\color{texemcolor},
  fontupper=\ttfamily\color{black},
  nobeforeafter,
  tcbox raise base,
  colback=codebackground,colframe=white,
  top=0pt,bottom=0pt,left=0mm,right=0mm,
  leftrule=0pt,rightrule=0pt,toprule=0mm,bottomrule=0mm,
  boxsep=0.5mm,
  #2}{#1}

\patchcmd\newpage{\vfil}{}{}{}
\usepackage{amssymb,bm}
\usepackage{mathtools}
\usepackage{graphics,graphicx,xcolor}
\usepackage{cases}
\usepackage{anyfontsize}
\usepackage{cleveref}
\usepackage{booktabs}
\usepackage{multirow}
\usepackage{tikz,pgfplots,pgf}

\usetikzlibrary{matrix,shapes,arrows,positioning}
\usetikzlibrary{shapes}
\tikzset{box/.style ={
		rectangle,
		rounded corners =5pt,
		minimum width =50pt,
		minimum height =20pt,
		inner sep=5pt,
		draw=blue}
}

\tikzset{zbox/.style ={
	rectangle,
	minimum width =50pt,
	minimum height =20pt,
	inner sep=5pt,
	draw=black}
}

\tikzset{ball/.style ={
circle,
minimum width =20pt,
minimum height =20pt,
inner sep=0.1pt,
draw=blue}
}

\tikzset{global scale/.style={
scale=#1,
every node/.append style={scale=#1}
}
}

\allowdisplaybreaks[4]
\newcommand{\bR}{\mathbb{R}}
\newcommand{\bn}{\bm{n}}
\newcommand{\cI}{\mathcal{I}}
\newcommand{\cN}{\mathcal{N}}
\newcommand{\cT}{\mathcal{T}}
\newcommand{\cL}{\mathcal{L}}
\newcommand{\cP}{\mathcal{P}}

\colorlet{mark}{black}
\newcommand{\blue}[1]{{\color{mark}#1}}

\crefname{example}{example}{examples}
\Crefname{example}{Example}{Examples}
\creflabelformat{example}{#2\textup{#1}#3}
\newcounter{exmp}
\newcommand{\exmplabel}[1]{\refstepcounter{exmp}\label[example]{#1}}

\renewcommand{\algorithmicrequire}{\textbf{Input:}}

\begin{tcbverbatimwrite}{tmp_\jobname_header.tex}
\title{The hard-constraint PINNs for interface optimal control problems\thanks{Submitted to the editors DATE.
\funding{The work of M. L was supported by NSTC of Taiwan (Grant Number 110-2115-M-A49-011-MY3). The work of  Y. S was supported by the Humboldt Research Fellowship for postdoctoral researchers. The work of X.Y was supported by the RGC TRS project T32-707/22-N. The work of H. Y was supported by the Fundamental Research Funds for the Central Universities, Nankai University (Grant Number 63221035) and Natural Science Foundation of Tianjin (Grant Number 22JCQNJC01120).}}}

\author{Ming-Chih Lai\thanks{Department of Applied Mathematics, National Yang Ming Chiao Tung University, 1001, Ta Hsueh Road, Hsinchu 300, Taiwan (\email{mclai@math.nctu.edu.tw}).}
\and Yongcun Song\thanks{Chair for Dynamics, Control, Machine Learning and Numerics$-$Alexander von Humboldt-Professorship, Department of Mathematics,  Friedrich-Alexander-Universit\"at Erlangen-N\"urnberg, 91058 Erlangen, Germany 
  (\email{ysong307$@$gmail.com}).}
\and Xiaoming Yuan\thanks{Department of Mathematics, The University of Hong Kong, Pok Fu Lam Road, Hong Kong,
	China 
  (\email{xmyuan@hku.hk}).}
\and Hangrui Yue\thanks{School of Mathematical Sciences, Nankai University, Tianjin 300071,
	China
  (\email{yuehangrui@gmail.com}).}
\and Tianyou Zeng\thanks{Department of Mathematics, The University of Hong Kong, Pok Fu Lam Road, Hong Kong,
		China (\email{logic@connect.hku.hk}) }
}

\headers{Hard-constraint PINNs for interface optimal control problems}{M. C. Lai, Y. Song, X. Yuan, H. Yue, and T. Zeng}
\end{tcbverbatimwrite}
\input{tmp_\jobname_header.tex}

\ifpdf
\hypersetup{ pdftitle={The hard-constraint PINNs for interface optimal control problems} }
\fi

\begin{document}
\maketitle

%% ------------------------------------------------------------------
%% ABSTRACT
%% ------------------------------------------------------------------
\begin{tcbverbatimwrite}{tmp_\jobname_abstract.tex}
\begin{abstract}
We show that the physics-informed neural networks (PINNs), in combination with some recently developed discontinuity capturing neural networks, can be applied to solve optimal control problems subject to partial differential equations (PDEs) with interfaces and some control constraints.
The resulting algorithm is mesh-free and scalable to different PDEs, and it ensures the control constraints rigorously.
Since the boundary and interface conditions, as well as the PDEs, are all treated as soft constraints by lumping them into a weighted loss function, it is necessary to learn them simultaneously and there is no guarantee that the boundary and interface conditions can be satisfied exactly.
This immediately causes difficulties in tuning the weights in the corresponding loss function and training the neural networks.
To tackle these difficulties and guarantee the numerical accuracy, we propose to impose the boundary and interface conditions as hard constraints in PINNs by developing a novel neural network architecture.
The resulting hard-constraint PINNs approach guarantees that both the boundary and interface conditions can be satisfied exactly or with a high degree of accuracy, and they are decoupled from the learning of the PDEs.
Its efficiency is promisingly validated by some elliptic and parabolic interface optimal control problems.
\end{abstract}


\end{tcbverbatimwrite}
\input{tmp_\jobname_abstract.tex}
%% ------------------------------------------------------------------
%% END HEADER
%% ------------------------------------------------------------------

\section{Introduction}

Partial differential equations (PDEs) with interfaces capture important applications in science and engineering such as fluid mechanics \cite{layton2009using},  biological science \cite{geng2007treatment}, and material science \cite{hou1997hybrid}.
Typically, PDEs with interfaces are modeled as piecewise-defined PDEs in different regions coupled together with interface conditions, e.g., jumps in solution and flux across the interface, hence nonsmooth or even discontinuous solutions.
Numerical methods for solving PDEs with interfaces have been extensively studied in the literature, see e.g., \cite{gong2008immersed,he2022mesh,hu2022discontinuity,li1998fast,li2006immersed}.
In addition to numerical simulation of PDEs with interfaces, it is very often to consider how to control them with certain goals.
As a result, optimal control problems of PDEs with interfaces (or interface optimal control problems, for short) arise in various fields.
To mention a few, see applications in crystal growth \cite{meyer2006optimal} and composite materials \cite{zhang2020unfitted}.

In this paper, we consider interface optimal control problems that can be abstractly written as
\begin{equation}\label{eq:ocip-general-form}
	\begin{aligned}
		\min_{y \in Y, u \in U} \quad  J(y, u) \quad
		\mathrm{s.t.}  \quad  \cI(y, u) = 0, ~
		u \in U_{ad}.
	\end{aligned}
\end{equation}
Above, $Y$ and $U$ are Banach spaces,  $J: Y \times U \to \bR$ is the objective functional to minimize,
and $y \in Y$ and $u \in U$ are the state variable and the control variable, respectively.
The operator $\cI: Y \times U \to Z$ with $Z$ Banach space defines a PDE with interface. Throughout, we assume that $\cI(y,u)=0$ is well-posed. That is, for each $u\in U_{ad}$, there exists a unique $y$ that solves $\cI(y, u) = 0$ and varies continuously with respect to $u$.  The control constraint $u\in U_{ad}$ imposes point-wise boundedness constraints on $u$ with the admissible set $U_{ad}$ a nonempty closed subset of $U$.
Problem \eqref{eq:ocip-general-form} aims to find an optimal control $u^* \in U_{ad}$, which determines a state $y^*$ through $\cI(y^*, u^*) = 0$, such that $J(y, u)$ is minimized by the pair $(y^*, u^*)$.

Problem \eqref{eq:ocip-general-form} is challenging from both theoretical analysis and algorithmic design perspectives. First,  solving problem \eqref{eq:ocip-general-form} entails appropriate discretization schemes due to the presence of interfaces.  For instance, direct applications of standard finite element or finite difference methods fail to produce satisfactory solutions because of the difficulty in enforcing the interface conditions into numerical
discretization, see e.g.,~ \cite{babuska1970finite}.  Moreover, similar to the typical optimal control problems with PDE constraints studied in \cite{de2015numerical,hinze2008optimization,lions1971optimal,troltzsch2010optimal},
the resulting algebraic systems after discretization are high-dimensional and ill-conditioned and hence difficult to be solved.
Finally, the presence of the control constraint $u\in U_{ad}$ leads to problem \eqref{eq:ocip-general-form} a nonsmooth optimization problem. Consequently,  the well-known gradient-type methods like gradient descent methods, conjugate gradient methods, and quasi-Newton methods cannot be applied directly. All these obvious difficulties imply that meticulously designed algorithms are  required for solving problem \eqref{eq:ocip-general-form}.

\subsection{State-of-the-art}
Numerical methods for solving some optimal control problems modeled by (\ref{eq:ocip-general-form}) have been studied in the literature. These methods combine mesh-based numerical discretization schemes and optimization algorithms that can respectively enforce the interface conditions and tackle the nonsmoothness caused by the constraint $u\in U_{ad}$.
For the numerical discretization of elliptic interface optimal control problems, we refer to the  immersed finite element methods in \cite{su2023numerical,zhang2015immersed}, the interface-unfitted finite element method based on Nitsche's approach in \cite{yang2018interface}, and the interface concentrated finite element method in \cite{wachsmuth2016optimal}. Moreover, an immersed finite element method is proposed in  \cite{zhang2020immersed} for parabolic interface  optimal control problems.
Although these finite element methods have shown to be effective to some extent, their practical implementation is not easy, especially for interfaces with complex geometries in high-dimensional spaces. Meanwhile, when the shape of the domain is complicated, generating a suitable mesh is even a nontrivial task, which imposes additional difficulty in solving the problems.

Moreover, various optimization methods have been developed in the context of optimal control problems, such as the semismooth Newton methods \cite{hinze2008optimization,ulbrich2011semismooth}, the inexact Uzawa method \cite{song2019inexact}, the alternating direction method of multipliers (ADMM) \cite{glowinski2022application}, and the primal-dual methods \cite{biccari2023two,song2023accelerated}.  All these optimization methods can be applied to solve  (\ref{eq:ocip-general-form}). It is notable that, to implement the above methods, two PDEs with interfaces ($\cI(y,u)=0$ and its adjoint system) or a saddle point problem are usually required to be solved repeatedly. After some proper numerical discretization such as the aforementioned immersed and interface-unfitted finite element methods, the resulting systems are large-scale and ill-conditioned, and the computation cost for solving the PDEs with interfaces or the saddle point problem repeatedly could be extremely high in practice.

\subsection{Physics-informed neural networks}

In the past few years, thanks to the universal approximation property \cite{cybenko1989approximation,gripenberg2003approximation,hornik1991approximation} and the great expressivity \cite{raghu2017expressive} of deep neural networks (DNNs),  some deep learning methods have been proposed to solve various PDEs, such as the physics-informed neural networks (PINNs) \cite{raissi2019physics}, the deep Ritz method \cite{e2018deep},  the deep Galerkin method \cite{sirignano2018dgm}, and the neural Q-learning method  \cite{cohen2023neural}. Compared with the traditional numerical methods for PDEs, deep learning methods are usually mesh-free, easy to implement, scalable to different PDE settings, and can overcome the curse of dimensionality. Among them, PINNs methods have become one of the most prominent deep learning methods and have been extensively studied in e.g.,  \cite{karniadakis2021physics,kharazmi2019variational,lu2021deepxde,raissi2019physics}. However, in general, these PINNs methods require the smoothness of the solutions to the PDEs mainly because the activation functions used in a DNN are in general smooth (e.g., the sigmoid function) or at least continuous (e.g., the rectified linear unit (ReLU) function). Consequently, the above PINNs methods cannot be directly used to solve PDEs with interfaces whose solutions are only piecewise-smooth.

To overcome the aforementioned difficulty,  some PINNs methods tailored for PDEs with interfaces are proposed in e.g., \cite{he2022mesh,hu2022discontinuity,tseng2022cusp,wu2022inn} and these methods primarily focused on developing new ways of using DNNs to approximate the underlying nonsmooth or discontinuous solution.  In \cite{he2022mesh}, it is suggested to approximate the solution by two neural networks corresponding to the two distinct sub-domains determined by the interface, so that the solution remains smooth in each sub-domain. A similar idea can also be found in \cite{wu2022inn}.  In this way, the numerical results obtained by PINNs are satisfactory but one has to train two neural networks, which requires more computational effort.  To alleviate this issue, a discontinuity capturing shallow neural network (DCSNN) is proposed in \cite{hu2022discontinuity}. The DCSNN allows a single neural network to approximate piecewise-smooth functions by augmenting a coordinate variable, which labels different pieces of each subdomain, as a feature input of the neural network. Since the neural network can be shallow, the resulting number of trainable parameters is moderate and thus the neural network is relatively easier to train. Inspired by \cite{hu2022discontinuity}, a cusp-capturing neural network is proposed in \cite{tseng2022cusp} to solve elliptic PDEs with interfaces whose solutions are continuous but have discontinuous first-order derivatives on the interfaces. The cusp-capturing neural network contains the absolute value of the zero level set function of the interface as a feature input and can capture the solution cusps (where the derivatives are discontinuous) sharply. Finally, for completeness, we mention that other deep learning methods for solving PDEs with interfaces can be referred to \cite{guo2021deep,hu2023efficient,sun2023dirichlet,wang2020mesh} and the references therein.

In addition to solving PDEs,  various PINNs for solving optimal control problems of PDEs  have been proposed in the literature, see \cite{barry2022physics,lu2021physics,mowlavi2023optimal,song2023admm}.
In \cite{mowlavi2023optimal}, the vanilla PINN method \cite{raissi2019physics} is extended to optimal control problems by approximating the control variable with another neural network in addition to the one for the state variable. Then, these two neural networks are simultaneously trained by minimizing a loss function defined by a weighted sum of the objective functional and the residuals of the PDE constraint. Then, PINNs with hard constraints are proposed  in \cite{lu2021physics} for solving optimal design problems, where the PDE and additional inequality constraints are treated as hard constraints by an augmented Lagrangian method.  In \cite{barry2022physics}, it is suggested to solve an optimal control problem by deriving the first-order optimality system, and approximating the control variable, the state variable, and the corresponding adjoint variable by different DNNs, respectively. Then, a stationary point of the optimal control problem can be computed by minimizing a loss function that consists of the residuals of the first-order optimality system.  Recently,  the ADMM-PINNs algorithmic framework is proposed in \cite{song2023admm} that applies to a general class of optimal control problems with nonsmooth objective functional.   It is worth noting that all the above-mentioned PINNs methods are designed for only optimal control problems with smooth PDE constraints, and they cannot be directly applied to interface optimal control problems modeled by (\ref{eq:ocip-general-form}). To our best knowledge, there is still no literature for studying the application of PINNs on the interface optimal control problems modeled by \eqref{eq:ocip-general-form}.

\subsection{Main contributions}
Inspired by the great success of  PINNs in solving various PDEs and optimal control problems,  we develop some PINNs methods in this paper for solving problem \eqref{eq:ocip-general-form}.  We first show that following the ideas in \cite{barry2022physics},  the PINN method \cite{raissi2019physics} can be applied to solve the first-order optimality system of problem (\ref{eq:ocip-general-form}) with the variables approximated by DCSNNs. The resulting PINN method  is mesh-free and scalable to different PDEs with interfaces and ensures the control constraint $u\in U_{ad}$  rigorously.
However, as shown in Section \ref{subsec:elliptic-pinn-loss}, this method treats the underlying PDEs and the boundary and interface conditions as soft constraints by penalizing them in the loss function with constant penalty parameters. Hence, the boundary and interface conditions cannot be satisfied exactly, and the numerical errors are mainly accumulated on the boundary and the interface as validated by the numerical results in  \Cref{sec:ocip-elliptic-numexp}.  Moreover, such a soft-constraint PINN method treats the PDE and the boundary and interface conditions together during the training process and its effectiveness strongly depends on the choices of the weights in the loss function. Typically, there is no established rule or principle to systematically determine the weights, and setting them manually by trial and error is extremely challenging and time-consuming.

To tackle the above issues, we propose the hard-constraint PINNs, where
the boundary and interface conditions are imposed as hard constraints and can be treated separately from the PDEs in the training process.  
In this context, the term ``hard constraints" refers to that the boundary and interface conditions can be satisfied exactly or with a high degree of accuracy by the designed neural networks.
For this purpose, we develop a novel neural network architecture by generalizing the DCSNN to approximate the first-order optimality system of  \eqref{eq:ocip-general-form}.  To be concrete, we first follow the ideas in \cite{lu2021physics,sheng2021pfnn} to modify the output of the neural network to impose the boundary condition. Then, to impose the interface condition as hard constraints, we propose to construct an auxiliary function for the interface as an additional feature input of the neural network.  Such an auxiliary function depends on the geometrical property of the interface and its construction is nontrivial. To address this issue, we elaborate on the methods for constructing appropriate auxiliary functions for interfaces with different geometrical properties. This ensures the hard-constraint PINNs are highly implementable.
Numerical results for different types of interface optimal control problems are reported to validate the effectiveness and flexibility of the hard-constraint PINNs.  Finally, we mention that the proposed hard-constraint PINNs can be directly applied to solve $\cI(y,u)=0$ per se since it is involved as a part of problem (\ref{eq:ocip-general-form}) and its first-order optimality system.

\subsection{Organization}
The rest of the paper is organized as follows.
In \Cref{sec:preliminaries}, for the convenience of further discussion, we specify problem (\ref{eq:ocip-general-form}) as a distributed elliptic interface optimal control problem, where the control arises as a source term in the model. Then, we review some existing results on the DCSNN.
In \Cref{sec:elliptic-pinns-hc}, we first demonstrate the combination of the DCSNN and the PINN method for solving the distributed elliptic interface optimal control problem and then propose the hard-constraint PINN method to impose the boundary and interface conditions as hard constraints.
We test several elliptic optimal control problems in \Cref{sec:ocip-elliptic-numexp} to validate the efficiency and effectiveness of the proposed hard-constraint PINN method.
In \Cref{sec:furtuer-extension}, we showcase how to extend the hard-constraint PINN method by an elliptic interface optimal control problem where the control acts on the interface and a distributed parabolic interface optimal control problem. Some related numerical experiments are also presented to validate the effectiveness. Finally, we make some conclusions and comments for future work in \Cref{sec:conclusion}.

\section{Preliminaries}\label{sec:preliminaries}

In this section, we present some preliminaries that will be used throughout the following discussions.
First, to impose our ideas clearly, we specify the generic model (\ref{eq:ocip-general-form}) as a distributed elliptic interface optimal control problem and summarize some existing results.
We then review the DCSNN proposed in \cite{hu2022discontinuity} for elliptic PDEs with interfaces.

\subsection{A distributed elliptic interface optimal control problem}

Let $\Omega \subset \bR^d \ (d=2, 3)$ be a bounded domain with Lipschitz continuous boundary $\partial \Omega$, and $\Gamma \subset \Omega$ be an oriented embedded interface, which divides $\Omega$ into two non-overlapping subdomains $\Omega^-$ (inside) and $\Omega^+$ (outside) such that $\Omega=\Omega^-\cup \Omega^+\cup \Gamma$ and $\overline{\Omega^+} \cap \overline{\Omega^-} = \Gamma$, see  \Cref{fig:geometry} for an illustration.
We consider the following optimal control problem:
\begin{equation}\label{eq:ocip-elliptic}
	\begin{aligned}
		\underset{y\in L^2(\Omega),u\in L^2(\Omega)}{\min} \ & J(y, u) := \dfrac{1}{2} \int_\Omega (y - y_d)^2 dx + \dfrac{\alpha}{2} \int_\Omega u^2 dx,
	\end{aligned}
\end{equation}
subject to the state equation
\begin{equation}\label{eq:ip-elliptic}
		 -\nabla\cdot(\beta\nabla y)=u+f ~\text{in}~\Omega \backslash \Gamma, ~
		 [y]_{\Gamma}=g_0,~ [\beta\partial_{\bm{n}}y]_{\Gamma}= g_1 ~\text{on}~\Gamma,~
		 y= h_0 ~\text{on}~\partial\Omega,
\end{equation}
and the control constraint $u\in U_{ad}$ with
\begin{equation}\label{eq:control-constraint-elliptic}
	u \in U_{ad} := \{ u \in L^2(\Omega) : u_a(x) \leq u(x) \leq u_b(x) \text{~a.e.~in~} \Omega \} \subset L^2(\Omega),
\end{equation}
where $u_a, u_b \in L^2(\Omega)$.

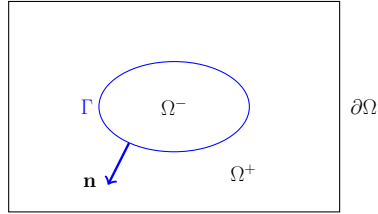
\begin{figure}[htbp]
	\centering
	\begin{tikzpicture}[global scale =0.4]
		\draw (2,3) rectangle (13,10);
		\draw (7.5,6.5) ellipse (2.5 and 1.5) [color=blue];
		\node at(4.6,6.5)  [font=\fontsize{20}{20}\selectfont] {\color{blue}$\Gamma$};
		\node at(7.5,6.5)  [font=\fontsize{20}{20}\selectfont] {$\Omega^-$};
		\node at(9.8,4.3)  [font=\fontsize{20}{20}\selectfont] {$\Omega^+$};
		\node at(4.7,4)  [font=\fontsize{20}{20}\selectfont] {$\mathbf{n}$};
		\node at(13.8,6.5)  [font=\fontsize{20}{20}\selectfont] {$\partial\Omega$};
		\draw[->][line width=0.9 pt][color=blue](6,5.29)--(5.3,3.9);
	\end{tikzpicture}
	\caption{The geometry of an interface problem: an illustration}
	\label{fig:geometry}
\end{figure}

Above,  the function $y_d\in L^2(\Omega)$ is the target and the constant $\alpha>0$ is a regularization parameter.  The functions $f\in L^2(\Omega), g_0\in H^{\frac{1}{2}}(\Gamma), g_1\in L^2(\Gamma)$, and $h_0\in H^{\frac{1}{2}}(\partial \Omega)$ are given, and $\beta$ is a positive piecewise-constant in $\Omega \backslash \Gamma$ such that $\beta = \beta^-$ in $\Omega^-$ and $\beta = \beta^+$ in $\Omega^+$.
The bracket $[\cdot]_\Gamma$ denotes the jump discontinuity across the interface $\Gamma$ and is defined by
$$
[y]_{\Gamma}(x):=\lim_{\tilde{x}\rightarrow x \text{~in~} \Omega^+}y(\tilde{x}) - \lim_{\tilde{x} \rightarrow x \text{~in~} \Omega^-} y(\tilde{x}), \forall x \in \Gamma.
$$
The operator $\partial_{\bn}$ stands for the normal  derivative on $\Gamma$, i.e. $\partial_{\bn}y (x) = \bn \cdot \nabla y(x)$ with $\bn\in \bR^d$ the outward unit normal vector of $\Gamma$.
In particular, we have
$$
[\beta\partial_{\bm{n}}y]_{\Gamma}(x):=\beta^+\lim_{\tilde{x} \rightarrow x \text{~in~} \Omega^+}\bm{n}\cdot\nabla y(\tilde{x}) - \beta^- \lim_{\tilde{x} \rightarrow x \text{~in~} \Omega^-} \bm{n}\cdot\nabla y(\tilde{x}), \forall x \in \Gamma.
$$
Moreover, $y = h_0 \text{~on~} \partial \Omega$ is called the \emph{boundary condition},  $[y]_\Gamma = g_0 \text{~on~} \Gamma$ is called the \emph{interface condition}, and $[\beta\partial_{\bm{n}}y]_{\Gamma}= g_1 \text{~on~} \Gamma$ is called the \emph{interface-gradient condition}.
For (\ref{eq:ocip-elliptic})-(\ref{eq:control-constraint-elliptic}), we have the following results.
\begin{theorem}[cf. \cite{zhang2015immersed}]\label{thm:optcond-elliptic}
	Problem (\ref{eq:ocip-elliptic})-(\ref{eq:control-constraint-elliptic}) admits a unique solution $(u^*, y^*)^\top \in U_{ad}\times L^2(\Omega)$, and the following first-order optimality system holds
	\begin{equation}\label{eq:oc}
		u^*=\mathcal{P}_{U_{ad}}\Big(-\frac{1}{\alpha}p^*\Big),
	\end{equation}
	where $\mathcal{P}_{U_{ad}}(\cdot)$ denotes the projection  onto $U_{ad}$, and  $p^*$ is the adjoint variable associated with $u^*$, which is obtained from the successive solution of the following two equations:
	\begin{equation}\label{eq:state}
 -\nabla\cdot(\beta\nabla y^*)=u^*+f ~ \text{in}~\Omega \backslash \Gamma,~ 			 [y^*]_{\Gamma}=g_0,~ [\beta\partial_{\bm{n}}y^*]_{\Gamma}= g_1 ~\text{on}~\Gamma,~
			y^*= h_0 ~\text{on}~\partial\Omega,
	\end{equation}
	\begin{equation}\label{eq:adjoint}
			 -\nabla\cdot(\beta\nabla p^*)=y^*-y_d ~\text{in}~\Omega \backslash \Gamma, ~
			[p^*] _{\Gamma}=0, ~ [\beta \partial_{\bm{n}}p^*]_{\Gamma}=0 ~\text{on}~\Gamma, ~
			 p^*=0 ~\text{on}~\partial\Omega.
	\end{equation}
\end{theorem}

Since problem (\ref{eq:ocip-elliptic})-(\ref{eq:control-constraint-elliptic}) is known to be convex \cite{zhang2015immersed}, the optimality conditions (\ref{eq:oc})-(\ref{eq:adjoint}) are also sufficient.  Furthermore, 
substituting (\ref{eq:oc}) into (\ref{eq:state}) yields
\begin{equation}\label{eq:state_reduced}
		 -\nabla\cdot(\beta\nabla y^*)=\mathcal{P}_{U_{ad}}\Big(-\frac{p^*}{\alpha}\Big)+f ~\text{in}~\Omega \backslash \Gamma, ~
		 [y^*]_{\Gamma}=g_0,~ [\beta\partial_{\bm{n}}y^*]_{\Gamma}= g_1 ~\text{on}~\Gamma,~
		 y^*= h_0 ~\text{on}~\partial\Omega.
\end{equation}
Therefore, solving \eqref{eq:ocip-elliptic}-(\ref{eq:control-constraint-elliptic}) is equivalent to solving equations (\ref{eq:adjoint}) and  (\ref{eq:state_reduced}) simultaneously.

\subsection{Discontinuity capturing shallow neural networks}\label{subsec:dscnn}

In this subsection, we briefly review and explain the idea of the DCSNN \cite{hu2022discontinuity} using  problem \eqref{eq:ip-elliptic}.

First, note that although the solution $y$ to \eqref{eq:ip-elliptic} is only a $d$-dimensional piecewise-smooth function,  it can be extended to a $(d+1)$-dimensional function $\tilde{y}(x, z)$, which is smooth  on the domain $\Omega\times \mathbb{R}$ and satisfies
\begin{equation}\label{eq:y-smooth-ext}
	y(x) = \begin{cases}
		\tilde{y}(x, 1), & \text{~if~} x \in \Omega^+, \\
		\tilde{y}(x, -1), & \text{~if~} x \in \Omega^-,
	\end{cases}
\end{equation}
where the additional input $z\in \mathbb{R}$ is the augmented coordinate variable that labels $\Omega^+$ and $\Omega^-$. Note that such a smooth extension $\tilde{y}$ always exists, since the function $y$ can be viewed as a smooth function defined on a closed subset of $\bR^{d+1}$, see \cite{lee2012introduction}. The extension (\ref{eq:y-smooth-ext}) is not unique since there are infinitely many choices of $\tilde{y}(x, z)$ for $z \neq \pm 1$.

Substituting (\ref{eq:y-smooth-ext}) to (\ref{eq:ip-elliptic}), it is easy to show that $\tilde{y}(x,z)$ satisfies the following equation
\begin{equation}\label{eq:y_smooth}
	\left\{
	\begin{aligned}
		& -\Delta_x \tilde{y}(x, z)=\left\{
		\begin{aligned}
			& \dfrac{1}{\beta^+} \left(  u(x)+f(x)  \right) \quad \text{~if~} x\in\Omega^+, z=1\\
			& \dfrac{1}{\beta^-}\left(  u(x)+f(x)  \right)  \quad \text{~if~}x\in\Omega^-, z=-1
		\end{aligned}
		\right.,\\
		&\tilde{y}(x, 1) - \tilde{y}(x, -1) = g_0(x), \quad \text{~if~} x \in \Gamma \\
		& \beta^+ \bn \cdot \nabla \tilde{y}(x, 1) - \beta^- \bn \cdot \nabla \tilde{y}(x, -1) = g_1(x), \quad \text{~if~} x \in \Gamma, \\
		&\tilde{y}(x, 1) = h_0(x), \quad \text{~if~} x \in \partial \Omega.
	\end{aligned}
	\right.
\end{equation}
Hence, the solution $y$ to problem \eqref{eq:ip-elliptic} can be obtained from (\ref{eq:y-smooth-ext}) with $\tilde{y}(x, z)$ computed by solving (\ref{eq:y_smooth}).
For solving problem (\ref{eq:y_smooth}), we note that the extended function  $\tilde{y}$ is smooth and one can construct a neural network $\hat{y}(x, z; \theta)$ with $d+1$ inputs, which is referred to as the DCSNN \cite{hu2022discontinuity}, to approximate $\tilde{y}$.
Since $\tilde{y}$ is continuous, it follows from the universal approximation theorem  \cite{cybenko1989approximation} that one can choose  $\hat{y}$ as a shallow neural network.
Then, the PINN method \cite{raissi2019physics} can be applied to solve  \eqref{eq:y_smooth} and
we refer to \cite{hu2022discontinuity} for the details.

\section{ The hard-constraint PINN method for (\ref{eq:ocip-elliptic})-(\ref{eq:control-constraint-elliptic})}\label{sec:elliptic-pinns-hc}

In this section, we first demonstrate that, combined with the DCSNN,  the PINN method \cite{raissi2019physics} can be applied to solve the reduced optimality conditions (\ref{eq:adjoint})-(\ref{eq:state_reduced}) and hence to solve (\ref{eq:ocip-elliptic})-(\ref{eq:control-constraint-elliptic}). Then, we impose the boundary and interface conditions in (\ref{eq:adjoint})-(\ref{eq:state_reduced}) as hard constraints by designing two novel neural networks to approximate $y$ and $p$, and propose the hard-constraint PINN method for solving (\ref{eq:ocip-elliptic})-(\ref{eq:control-constraint-elliptic}).

\subsection{A soft-constraint PINN method for (\ref{eq:ocip-elliptic})-(\ref{eq:control-constraint-elliptic})}\label{subsec:elliptic-pinn-loss}

We recall that,  for solving (\ref{eq:ocip-elliptic})-(\ref{eq:control-constraint-elliptic}), it is sufficient to solve the equations (\ref{eq:adjoint}) and  (\ref{eq:state_reduced}) simultaneously.
First, we apply two DCSNNs to approximate  $y$ and $p$. To this end, let $\tilde{y}: \Omega\times \mathbb{R}\rightarrow \mathbb{R}$ and $\tilde{p}: \Omega\times \mathbb{R}\rightarrow \mathbb{R}$ be two smooth extensions of  $y$ and $p$ respectively, which satisfy  \eqref{eq:y-smooth-ext} and
\begin{equation}\label{eq:p-smooth-ext}
	p(x) = \begin{cases}
		\tilde{p}(x, 1), & \text{~if~} x \in \Omega^+, \\
		\tilde{p}(x, -1), & \text{~if~} x \in \Omega^-.
	\end{cases}
\end{equation}
Then, substituting (\ref{eq:y-smooth-ext}) and (\ref{eq:p-smooth-ext}) into equations  \eqref{eq:state_reduced} and  \eqref{eq:adjoint}, we obtain that  $\tilde{y}$ and $\tilde{p}$ satisfy the following system:
\small
\begin{eqnarray}\label{eq:optcond-elliptic}
	\hspace{2em}\left\{~
	\begin{aligned}
		& -\Delta_x \tilde{y}(x, z)=\left\{
		\begin{aligned}
			& \dfrac{1}{\beta^+} \left(  f(x) + \mathcal{P}_{[u_a(x), u_b(x)]}\left( -\frac{1}{\alpha} \tilde{p}(x, z)\right) \right) \quad \text{~if~} x\in\Omega^+, z=1\\
			& \dfrac{1}{\beta^-} \left( f(x) + \mathcal{P}_{[u_a(x), u_b(x)]}\left(-\frac{1}{\alpha} \tilde{p}(x, z)\right) \right) \quad \text{~if~}x\in\Omega^-, z=-1
		\end{aligned}
		\right.,\\
		& \tilde{y}(x, 1) -\tilde{y}(x, -1)=g_0(x), \quad \text{~if~} x \in \Gamma, \\
		& \beta^+ \bn \cdot \nabla_x \tilde{y}(x, 1) - \beta^- \bn \cdot \nabla_x \tilde{y}(x, -1) = g_1(x), \quad \text{~if~} x \in \Gamma, \\
		& \tilde{y}(x, 1)=h_0(x), \quad \text{~if~} x \in \partial\Omega, \\ 
		& -\Delta_x \tilde{p}(x, z) = \left\{
		\begin{aligned}
			& \dfrac{1}{\beta^+} (\tilde{y}(x, z) - y_d(x)) \quad \text{~if~} x \in \Omega^+, z=1 \\
			& \dfrac{1}{\beta^-} (\tilde{y}(x, z) - y_d(x)) \quad \text{~if~} x \in \Omega^-, z=-1
		\end{aligned}
		\right.,\\
		& \tilde{p}(x, 1) -\tilde{p}(x, -1)=0, \quad \text{~if~} x \in \Gamma, \\
		& \beta^+ \bn \cdot\nabla_x \tilde{p}(x, 1) - \beta^- \bn \cdot\nabla_x \tilde{p}(x, -1) = 0, \quad \text{~if~} x \in \Gamma, \\
		& \tilde{p}(x, 1)=0, \quad \text{~if~} x \in \partial\Omega.
	\end{aligned}
	\right.
\end{eqnarray}
\normalsize

Once $\tilde{y}$ and $\tilde{p}$ are computed by solving \eqref{eq:optcond-elliptic}, the solutions $y$ and $p$ to (\ref{eq:adjoint}) and  (\ref{eq:state_reduced})  can be obtained using  \eqref{eq:y-smooth-ext} and (\ref{eq:p-smooth-ext}).  Next, we solve  \eqref{eq:optcond-elliptic} by the PINN method \cite{raissi2019physics}. For this purpose,  we first sample training sets $\cT := \{(x^i, z^i)\}_{i=1}^{M} \subset (\Omega^+ \times \{1\}) \cup (\Omega^- \times \{-1\}), \cT_B := \{x_B^i\}_{i=1}^{M_B} \subset \partial \Omega$, and $\cT_\Gamma := \{x_\Gamma^i\}_{i=1}^{M_\Gamma} \subset \Gamma$.
We then apply two DCSNNs $\hat{y}(x, z; \theta_y)$ and $\hat{p}(x, z; \theta_p)$ to approximate $\tilde{y}(x, z)$ and $\tilde{p}(x, z)$, respectively; and train
the neural networks by minimizing the following loss function:
\footnotesize
\begin{gather}\label{eq:loss-elliptic}
	\begin{aligned}
	\mathcal{L}(\theta_y, \theta_p) = & \frac{w_{y, r}}{M} \sum_{i=1}^{M} \left| -\Delta_x \hat{y}(x^i,z^i; \theta_{y})-\frac{\mathcal{P}_{[u_a(x^i), u_b(x^i)]}(-\frac{1}{\alpha} \hat{p}(x^i, z^i; \theta_p)) + f(x^i)}{\beta^\pm} \right| ^2 \\
	& + \frac{w_{y, b}}{M_b}\sum_{i=1}^{M_b}|\hat{y}(x_B^i,1; \theta_{y}) - h_0(x_B^i)|^2  + \frac{w_{y, \Gamma}}{M_{\Gamma}}\sum_{i=1}^{M_\Gamma} \left|\hat{y}(x_\Gamma^i, 1; \theta_{y}) - \hat{y}(x_\Gamma^i, -1; \theta_{y}) - g_0(x_\Gamma^i) \right|^2 \\
	& + \frac{w_{y, \Gamma_n}}{M_{\Gamma}} \sum_{i=1}^{M_\Gamma} \left| \beta^+ \bn \cdot \nabla_x \hat{y}(x_i^{\Gamma}, 1; \theta_{y}) - \beta^-\bn \cdot  \nabla_x \hat{y}(x_i^{\Gamma}, -1; \theta_{y}) - g_1(x_\Gamma^i) \right|^2 \\
	& + \frac{w_{p, r}}{M} \sum_{i=1}^{M} \left|-\Delta_x \hat{p}(x_i, z_i; \theta_{p})-\frac{\hat{y}(x_i, z_i; \theta_{y}) - y_d(x_i)}{\beta^\pm} \right|^2 + \frac{w_{p, b}}{M_b}\sum_{i=1}^{M_b}|\hat{p}(x^i_B,1; \theta_{p})|^2  \\
	& + \frac{w_{p, \Gamma}}{M_{\Gamma}}\sum_{i=1}^{M_\Gamma} \left|\hat{p}(x_\Gamma^i, 1; \theta_{p}) - \hat{p}(x_\Gamma^i, -1; \theta_{p})\right|^2 \\
	& + \frac{w_{p, \Gamma_n}}{M_{\Gamma}}\sum_{i=1}^{M_\Gamma} \left|\beta^+ \bn \cdot \nabla_x \hat{p}(x_i^{\Gamma}, 1; {\theta_{p}}) - \beta^- \bn \cdot \nabla_x \hat{p}(x_i^{\Gamma}, -1; {\theta_{p}}) \right|^2,
	\end{aligned}
\end{gather}
\normalsize
where $w_{y, *}$ and $w_{p, *}$ are the weights for each term.

Note that the loss function \eqref{eq:loss-elliptic} is nonnegative, and if $\mathcal{L}(\theta_y, \theta_p) $ goes to zero, then the resulting $(\hat{y}, \hat{p})$ gives an approximate solution to \eqref{eq:optcond-elliptic}. Moreover, for any function $v: \Omega\rightarrow \mathbb{R}$, we have
$$\mathcal{P}_{[u_a(x), u_b(x)]}(v(x)) = (L_1 \circ L_2 \circ v)(x),$$
where $L_1(v(x)) := \text{ReLU}(v(x)-u_a(x))+u_a(x)$ and $L_2(v(x)) :=  -\text{ReLU}(u_b(x)-v(x))+u_b(x)$ with $\text{ReLU}(v(x))=\max\{v(x),0\}$ the ReLU function.
This implies that the projection $\mathcal{P}_{[u_a(x^i), u_b(x^i)]}(-\frac{1}{\alpha} \hat{p}(x^i, z^i; \theta_p))$ can be viewed as a composition of  $-\frac{1}{\alpha}\hat{p}(x,z;\theta_{p})$ and a two-layer neural network with ReLU as the activation functions.
As a result, the loss function $ \mathcal{L}(\theta_y, \theta_p) $ in  \eqref{eq:loss-elliptic} can be minimized by a stochastic optimization method with all the derivatives $\Delta_x \hat{y}$, $\nabla_x\hat{y}$, $\Delta_x \hat{p}$, $\nabla_x\hat{p}$, and the gradients $\frac{\partial \mathcal{L}}{\partial \theta_{y}}$, $\frac{\partial \mathcal{L}}{\partial \theta_{p}}$ computed by automatic differentiation.

We summarize the above PINN method  in \Cref{alg:pinn}.
\begin{algorithm}[ht]
	\caption{A soft-constraint PINN method  for \eqref{eq:ocip-elliptic}-\eqref{eq:control-constraint-elliptic}}
	\begin{algorithmic}[1]
		\REQUIRE Weights $w_{y, *}, w_{p, *}$ in \eqref{eq:loss-elliptic}.
		\STATE Initialize the neural networks $\hat{y}(x, z; \theta_y)$ and $\hat{p}(x, z; \theta_p)$ with $\theta_y^0$ and $\theta_p^0$.
		\STATE Sample training sets $\cT = \{x^i, z^i\}_{i=1}^M \subset \Omega \times \{\pm 1\}$, $\cT_B = \{x_B^i\}_{i=1}^{M_B} \subset \partial \Omega$ and $\cT_\Gamma = \{x_\Gamma^i\}_{i=1}^{M_\Gamma} \subset \Gamma$.
		\STATE Calculate the values of $f, y_d$ over $\cT$, the value of $h$ over $\cT_B$ and the value of $g_0, g_1$ over $\cT_\Gamma$.
		\STATE Train the neural networks $\hat{y}(x, z; \theta_y)$ and $\hat{p}(x, z; \theta_p)$ to identify the optimal parameters $\theta_y^*$ and $\theta_p^*$ by minimizing \eqref{eq:loss-elliptic}.
		\STATE $y(x) \gets \begin{cases}
			\hat{y}(x, 1;\theta_y^*), & \text{~if~} x \in \Omega^+ \\
			\hat{y}(x, -1;\theta_y^*), & \text{~if~} x \in \Omega^-
		\end{cases}$,~  $ p(x) \gets \begin{cases}
			\hat{p}(x, 1;\theta_p^*), & \text{~if~} x \in \Omega^+ \\
			\hat{p}(x, -1;\theta_p^*), & \text{~if~} x \in \Omega^-
		\end{cases}$, \\ $u(x) \gets \mathcal{P}_{[u_a(x), u_b(x)]}\left(-\frac{1}{\alpha} p(x)\right)$.
		\ENSURE Approximate solutions $u(x)$ and $y(x)$ to \eqref{eq:ocip-elliptic}-\eqref{eq:control-constraint-elliptic}.
	\end{algorithmic}
	\label{alg:pinn}
\end{algorithm}

It is easy to see that the computed control $u$ satisfies the control constraint $u \in U_{ad}$ strictly.
Additionally,  \Cref{alg:pinn} is mesh-free and is very flexible in terms of the geometries of the domain and the interface.  However, note that in  \Cref{alg:pinn}, the boundary and interface conditions are penalized in the loss function  \eqref{eq:loss-elliptic} with constant penalty parameters. Hence, these conditions are treated as soft constraints and cannot be satisfied rigorously by the solutions $y$ and $p$ computed by  \Cref{alg:pinn}.  Moreover, such a soft-constraint approach treats the PDE and the boundary and interface conditions together during the training process and its effectiveness strongly depends on the choices of the weights in the loss function  \eqref{eq:loss-elliptic}. Manually determining these weights through trial and error is extremely challenging and time-demanding.
The numerical results in \Cref{sec:ocip-elliptic-numexp} also show that this soft-constraint approach generates solutions with numerical errors mainly accumulated on the boundaries and the interfaces. To tackle the above issues, we consider imposing the boundary and interface conditions as hard constraints and can be treated separately from the PDE in the training of the neural networks.

\subsection{Hard-constraint boundary and  interface conditions}\label{subsec:hc-bc}

In this subsection, we elaborate on the construction of new neural networks to approximate the state variable $y$ and the adjoint variable $p$ by modifying the DSCNNs $\hat{y}(x, z; \theta_y)$ and $\hat{p}(x, z; \theta_p)$ in   \Cref{alg:pinn} so that the boundary and interface conditions in (\ref{eq:adjoint}) and  (\ref{eq:state_reduced})  are imposed as hard constraints.  In the following discussions, for the sake of simplicity, we still denote $\theta_y$ and $\theta_p$ the parameters of the neural networks with hard-constraint boundary and interface conditions.

Let $y \in L^2(\Omega)$ be the solution of~(\ref{eq:state_reduced}), then it satisfies
\begin{equation*}%\label{eq:bcic-nonhomo}
	\begin{aligned}
		& y = h_0 \text{~on~} \partial \Omega, \quad
		[y]_\Gamma = g_0, \quad [\beta \partial_{\bn} y]_\Gamma = g_1 \text{~on~} \Gamma,
	\end{aligned}
\end{equation*}
for some functions $g_0, g_1: \Gamma \to \bR$ and $h_0: \partial \Omega \to \bR$.
We first introduce two functions $g, h: \overline{\Omega} \to \bR$ satisfying
\begin{equation}\label{eq:g-func-prop}
	g|_{\partial \Omega} = h_0, \quad [g]_\Gamma = g_0, \quad g|_{\Omega^+} \in C^2(\overline{\Omega^+}), \quad g|_{\Omega^-} \in C^2(\overline{\Omega^-}),
\end{equation}
\begin{equation}\label{eq:h-func-prop}
	h \in C^2(\overline{\Omega}), \quad h(x) = 0 \text{~if~and~only~if~} x \in \partial \Omega.
\end{equation}
If the functions $g_0$ and $h_0,$ the interface $\Gamma$, and the boundary $\partial \Omega$ admit analytic forms, it is usually easy to construct $g$ and $h$ with analytic expressions.
Some discussions can be found in \cite{lagari2020systematic,lagaris1998artificial,lu2021physics}.
Otherwise, we can either adopt the method in \cite{sheng2021pfnn} or construct $g$ and $h$ by training two neural networks.
For instance, we can train a DCSNN $\hat{g}(x, z; \theta_g)$ and a neural network $\hat{h}(x; \theta_h)$ with smooth activation functions (e.g. the sigmoid function or the hyperbolic tangent function) by minimizing the following loss functions:
\begin{equation}\label{eq:loss_g}
	\frac{w_{1g}}{M_b}\sum_{i=1}^{M_b}|\hat{g}(x_B^i,1; \theta_{y}) - h_0(x_B^i)|^2  +    \frac{w_{2g}}{M_{\Gamma}}\sum_{i=1}^{M_\Gamma} \left|\hat{g}(x_\Gamma^i, 1; \theta_{g}) - \hat{g}(x_\Gamma^i, -1; \theta_{g})-g_0(x_{\Gamma}^i)\right|^2,
\end{equation}
\begin{equation}\label{eq:loss_h}
	\frac{w_{1h}}{M_b}\sum_{i=1}^{M_b}|\hat{h}(x_B^i; \theta_{h})|^2+\frac{w_{2h}}{M}\sum_{i=1}^{M}|\hat{h}(x^i; \theta_{h})-\bar{h}(x^i)|^2,
\end{equation}
where $w_{1g}, w_{2g}, w_{1h}$, and $w_{2h}>0$ are the weights, $\{x^i\}_{i=1}^{M} \subset \Omega$, $\{x_B^i\}_{i=1}^{M_B} \subset \partial \Omega$, and $\{x_\Gamma^i\}_{i=1}^{M_\Gamma} \subset \Gamma$ are the training points, and $\bar{h}(x)\in C^2(\Omega)$ is a known function satisfying $\bar{h}(x)\neq 0$ in $\Omega$, e.g. $\bar{h}(x) = \min_{\hat{x}\in \partial\Omega} \{\| x- \hat{x} \|_2^4 \}$.

With the functions $g$ and $h$ satisfying  \eqref{eq:g-func-prop} and  \eqref{eq:h-func-prop}, we approximate $y$ by
\begin{equation}\label{eq:neural-form-hard-bcij}
	\hat{y}(x; \theta_y) = g(x) + h(x) \cN_y(x, \phi(x); \theta_y),
\end{equation}
where $\cN_y(x,\phi(x); \theta_y)$ is a neural network with smooth activation functions and parameterized by $\theta_y$, and $\phi: \overline{\Omega} \to \bR$ satisfying
\begin{equation}\label{eq:aux-func-prop}
	\begin{aligned}
		& \phi \in C(\overline{\Omega}), \quad \phi|_{\Omega^+} \in C^2(\overline{\Omega^+}), \quad \phi|_{\Omega^-} \in C^2(\overline{\Omega^-}),  \quad
		[\phi]_\Gamma = 0, \quad [\beta \partial_{\bm{n}} \phi]_\Gamma \neq 0 \text{~a.e. on}~\Gamma
	\end{aligned}
\end{equation}
is an auxiliary function for the interface $\Gamma$.
It follows from \eqref{eq:h-func-prop} and \eqref{eq:aux-func-prop} that $h(x)\cN_y(x, \phi(x); \theta_y)$ is a continuous function of $x$ over $\overline{\Omega}$.

For the neural network $ \hat{y}(x; \theta_y) $ given by  \eqref{eq:neural-form-hard-bcij}, it is easy to verify that
\begin{eqnarray*}
&&	[\hat{y}]_\Gamma (x) = [g]_\Gamma (x) + [h(\cdot) \cN_y(\cdot, \phi(\cdot))]_\Gamma (x) = g_0(x), \quad \forall x \in \Gamma,\\
&&	\hat{y}|_{\partial \Omega}(x) = g|_{\partial \Omega} (x) + h|_{\partial \Omega} (x) \left( \cN_y (\cdot, \phi(\cdot))|_{\partial \Omega}\right)(x)= h_0(x), \quad \forall x \in \partial \Omega.
\end{eqnarray*}
Hence, the interface condition $[{y}]_\Gamma = g_0$ and the boundary condition $y|_{\partial \Omega} = h_0$ are satisfied exactly by $\hat{y}(x; \theta_y)$ if functions $g$, $h$, and $\phi$ are given in analytic expressions, and can be satisfied with a high degree of accuracy if $g$, $h$, and $\phi$ are approximated by pretrained neural networks.

Furthermore, we have
\small
\begin{equation}\label{eq:ij-gradient-comp}
	\begin{aligned}
		[\beta \partial_{\bn} \hat{y}]_\Gamma (x) & = [\beta \partial_{\bn} g]_\Gamma(x) + [\beta \partial_{\bn} h \cN_y(\cdot, \phi(\cdot))]_\Gamma (x) \\
		& = [\beta \partial_{\bn} g]_\Gamma (x) + (\beta^+ - \beta^-)\big(\cN_y(x, \phi(x))(\bn \cdot \nabla h(x)) + h(x) (\bn \cdot \nabla_x \cN(x, \phi(x)))\big) \\
		& \quad \qquad + \frac{\partial \cN_y}{\partial \phi} \left(h(x) [\beta \partial_{\bn} \phi]_\Gamma (x) \right), \quad \forall x \in \Gamma,
	\end{aligned}
\end{equation}
\normalsize
which implies that the interface-gradient condition $ [\beta \partial_{\bm{n}} {y}]_\Gamma =g_1$ cannot be exactly satisfied by $\hat{y}(x; \theta_y)$ and has to be treated as a soft constraint,  see  \eqref{eq:loss-elliptic-hc} for the details.

\begin{remark}
	Note that the neural network $\hat{y}(x; \theta_y)$ given by  \eqref{eq:neural-form-hard-bcij} reduces to the DSCNN $\hat{y}(x,z; \theta_y)$  used in  \Cref{alg:pinn} by taking $g = 0, h = 1$ and replacing $\phi$ with the piecewise constant
	\[ z(x) = \begin{cases}
		1, \quad &\text{if~} x \in \Omega^+, \\
		-1, \quad &\text{if~} x \in \Omega^-.
	\end{cases} \]
	However, this auxiliary variable $z$ does not satisfy the assumptions in  \eqref{eq:aux-func-prop}.
	Hence, the auxiliary function $\phi$ is a nontrivial generalization of the augmented coordinate variable $z$ introduced in the DSCNN.
\end{remark}

Given $\cN_y$ a neural network with smooth activation functions, we have that
$\cN_y \in C^\infty(\overline{\Omega})$. Moreover, it follows from the smooth assumptions of $g$ and $\phi$ in  \eqref{eq:g-func-prop} and  \eqref{eq:aux-func-prop} that $\hat{y}|_{\Omega^+} \in C^2(\overline{\Omega^+})$ and $\hat{y}|_{\Omega^-} \in C^2(\overline{\Omega^-})$.
Hence, the second-order derivatives of $\hat{y}$ are well-defined and continuous on $\Omega^-$ and $\Omega^+$.
In particular, we have that
\small
\begin{equation}\label{eq:neural-form-derivatives}
	\left\{
	\begin{aligned}
		& \frac{\partial \hat{y}}{\partial x_i} = \frac{\partial g}{\partial x_i} + h \left(\frac{\partial \cN_y}{\partial x_i} + \frac{\partial \cN_y}{\partial \phi} \frac{\partial \phi}{\partial x_i}\right) + \cN_y \frac{\partial h}{\partial x_i}, \\
		& \frac{\partial^2 \hat{y}}{\partial x_i^2} = \frac{\partial^2 g}{\partial x_i^2} + h\left( \frac{\partial^2 \cN_y}{\partial x_i^2} + 2 \frac{\partial^2 \cN_y}{\partial x_i \partial \phi} \frac{\partial \phi}{\partial x_i} + \frac{\partial^2 \cN_y}{\partial \phi^2} \left(\frac{\partial \phi}{\partial x_i}\right)^2 + \frac{\partial \cN_y}{\partial \phi} \frac{\partial^2 \phi}{\partial x_i^2} \right)\\
		&\qquad\qquad+ 2\frac{\partial h}{\partial x_i} \left(\frac{\partial \cN_y}{\partial x_i} + \frac{\partial \cN_y}{\partial \phi} \frac{\partial \phi}{\partial x_i}\right) + \cN_y \frac{\partial^2 h}{\partial x_i^2}.
	\end{aligned}
	\right.
\end{equation}

\normalsize

Similar to what we have done for the state variable $y$, we can also approximate $p$ by a neural network with the boundary and interface conditions in  \eqref{eq:adjoint} as hard constraints. To be concrete,
since the boundary and interface conditions for $p$  are homogeneous, we approximate it by
\begin{equation}\label{eq:neural-form-p}
	\hat{p}(x; \theta_p) = h(x) \cN_p(x, \phi(x); \theta_p),
\end{equation}
where $\cN_p(x, \phi(x); \theta_p)$ is a neural netowrk with smooth activation functions and parameterized by $\theta_p$, the functions $h$ and $\phi$ satisfy  \eqref{eq:h-func-prop} and  \eqref{eq:aux-func-prop}, respectively. In particular, the functions $h$ and $\phi$ for $\hat{p}$ can be the same as the ones for $\hat{y}$.
The derivatives of $\hat{p}$ can be calculated in the same ways as those in \eqref{eq:ij-gradient-comp}~and \eqref{eq:neural-form-derivatives}, that is
{\small
	\begin{equation}\label{eq:neural-form-p-derivatives}
		\left\{
		\begin{aligned}
			& \frac{\partial \hat{p}}{\partial x_i} = h \left(\frac{\partial \cN_p}{\partial x_i} + \frac{\partial \cN_p}{\partial \phi} \frac{\partial \phi}{\partial x_i}\right) + \cN_p \frac{\partial h}{\partial x_i}, \\
			& \frac{\partial^2 \hat{p}}{\partial x_i^2} = h \left( \frac{\partial^2 \cN_p}{\partial x_i^2} + 2 \frac{\partial^2 \cN_p}{\partial x_i \partial \phi} \frac{\partial \phi}{\partial x_i} + \frac{\partial^2 \cN_p}{\partial \phi^2} \left(\frac{\partial \phi}{\partial x_i}\right)^2 + \frac{\partial \cN_p}{\partial \phi} \frac{\partial^2 \phi}{\partial x_i^2} \right)  \\
			&\qquad\qquad+ 2\frac{\partial h}{\partial x_i} \left(\frac{\partial \cN_p}{\partial x_i} + \frac{\partial \cN_p}{\partial \phi} \frac{\partial \phi}{\partial x_i}\right) + \cN_p \frac{\partial^2 h}{\partial x_i^2}, \\
			& [\beta \partial_{\bn} \hat{p}]_\Gamma  = (\beta^+ - \beta^-)\left(\cN_p\bm{n} \cdot \nabla h + h (\bm{n} \cdot \nabla_x \cN_p)\right) + \frac{\partial \cN_p}{\partial \phi} \left(h [\beta \partial_{\bn} \phi]_\Gamma \right).
		\end{aligned}
		\right.
	\end{equation}
	
}

\subsection{The choice of \texorpdfstring{$\phi$}{auxiliary function}}\label{subsec:construct-aux-func}

We note that the abstract and general neural networks $\hat{y}(x; \theta_y)$ and $\hat{p}(x; \theta_p)$ given in  \eqref{eq:neural-form-hard-bcij} and  \eqref{eq:neural-form-p} can be used in practice only when the auxiliary function $\phi(x)$ satisfying  \eqref{eq:aux-func-prop} is chosen appropriately. In this subsection, we illustrate how to choose $\phi(x)$ for interfaces with different geometrical properties. In particular,  we shall show that, if the shapes of $\Omega^+, \Omega^-$, and $\Gamma$ are regular enough and their analytic expressions are known, then we can construct an auxiliary function $\phi(x)$ analytically.

First, if $\Gamma$ is the regular zero level set of a function  $\psi \in C^2(\overline{\Omega})$\footnote{The zero level set of $\psi$ is regular means that it does not contain any point where $\nabla \psi$ vanishes.}, then we can define $\phi(x)$ as follows:
\begin{equation*}
	\phi(x) = \begin{cases}
		\psi(x), & \text{~if~} x \in \Omega^-, \\
		0, & \text{~if~} x \in \Omega^+ \cup \Gamma \cup \partial \Omega,
	\end{cases}
\end{equation*}
which is smooth and clearly satisfies \eqref{eq:aux-func-prop}.
We present an example below for further explanations.
\begin{example}\label{ex:aux-func-ex-1}
	(Circle-shaped interfaces)
	Consider a domain $\Omega\subset\mathbb{R}^{d}$ and the interface $\Gamma\subset\Omega$ is given by the circle $\Gamma = \{x \in \Omega: \lVert x \rVert_2 = r_0 \}$, with $r_0>0$.
	The domain $\Omega$ is divided into two parts $\Omega^- = \{x \in\Omega: \lVert x \rVert_2 < r_0 \}$ and $\Omega^+ = \{x \in \Omega: \lVert x \rVert_2 > r_0 \}$.
	In this case, the interface $\Gamma$ is the regular zero level set of $\psi(x)= r_0^2 - \lVert x \rVert_2^2$ and the auxiliary function $\phi$ can be defined as
	\begin{equation*}
		\phi(x) = \begin{cases}
			r_0^2 - \lVert x \rVert_2^2, & \text{~if~} x \in \Omega^- \\
			0, & \text{~if~} x \in \Omega^+ \cup \Gamma \cup \partial \Omega.
		\end{cases}
	\end{equation*}
\end{example}

The above idea can be easily extended to the case where $\Gamma$ is a finite union of the regular zero level sets of some functions $\psi_1, \ldots, \psi_n \in C^2(\overline{\Omega^-}) (n\geq 1)$.
See \Cref{ex:aux-func-ex-2} for a concrete explanation.

\begin{example}\label{ex:aux-func-ex-2}
	(Box-shaped interfaces)
	Consider a domain $\Omega\subset\mathbb{R}^d$ containing the box $B:= [a_1, b_1] \times \cdots \times [a_d, b_d] \in \bR^d$. The interface $\Gamma := \partial B$ divides $\Omega$ into $\Omega^- = (a_1, b_1) \times \cdots \times (a_d, b_d)$ and $\Omega^+ = \Omega \backslash B$.
	Here, the sub-domain $\Omega^-$ is the intersection of $2d$ half-spaces, whose corresponding hyperplanes are the zero level sets of $\psi_i(x) = x_i - a_i, ~ i = 1, \ldots, d$ and $\psi_i(x) = b_{i-d} - x_{i-d}, ~ i = d+1, \ldots, 2d$, respectively.
	When each $\psi_i$ is treated as a function defined on $\Omega^-$, we can see that $\psi_i \in C^2(\overline{\Omega^-})$ and $\Gamma$ is indeed characterized by the union of the regular zero level sets of $\psi_1, \ldots, \psi_{2d}$.
	In this case, we define
	\begin{equation}\label{eq:def_phi_1}
		\phi(x) = \begin{cases}
			\prod_{i=1}^d (x_i - a_i)(b_i - x_i), & \text{~if~} x \in \Omega^-, \\
			0, & \text{~if~} x \in \Omega^+ \cup \Gamma \cup \partial \Omega.
		\end{cases}
	\end{equation}
	This is a smooth function satisfying \eqref{eq:aux-func-prop}. In particular, the pairwise intersections of the zero level sets of $\{\psi_i(x)\}_{i=1}^{2d}$ have measure zero, and the regularity of these zero level sets ensures that $[\beta \partial_{\bm{n}} \phi]_\Gamma \neq 0$ almost everywhere.
\end{example}

Next, we consider a more general case, where $\Gamma$ is a finite union of the regular zero level sets of functions $\psi_1, \psi_2, \ldots, \psi_n$, which are of class $C^2$ only in an open neighborhood of $\Gamma$. Such a situation arises, for instance, when $\Gamma$ is defined by the zero level set of a function represented in polar coordinates since the angle parameter is not differentiable at the origin.
Due to the lack of global smoothness, we cannot simply define $\phi(x)$ as in \eqref{eq:def_phi_1}; otherwise, the resulting $\phi$ may not satisfy the assumption $\phi\in C^2(\Omega^-)$ in \eqref{eq:aux-func-prop}.  To tackle this issue,
we propose to set $\phi(x) = 0$ when $x \in \overline{\Omega^+}$, and  $\phi(x)$ to be a nonzero constant over the region inside $\Omega^-$ where $\psi_1, \ldots, \psi_n$ fail to be $C^2$ functions. Then, in the rest part of the domain, we define $\phi(x)$ by using $\psi_1, \ldots, \psi_n$ so that the resulting piecewise function $\phi(x)$ is well-defined and satisfies (\ref{eq:aux-func-prop}).   We shall elaborate on the above ideas in the remainder of this section and for this purpose, we make the following assumptions.

\begin{assumption}\label{ass:gamma-regularity-2}
	The sub-domain $\Omega^-$ is the intersection of the interior of finitely many oriented, smooth, and embedded manifolds $M_1, M_2, \ldots, M_n$, where $M_i \cap M_j$ is of measure zero whenever $i \neq j$ and $i, j \in \{1, \ldots, n\}$.
\end{assumption}
\begin{assumption}\label{ass:gamma-regularity-3}
	There exists an open neighborhood  $U \subset \bR^d$ of $\Gamma$, such that for each $i \in \{1, \ldots, n\}$ and manifold $M_i$, there exists functions $\psi_i : U \to \bR$ satisfying $\psi_i \in C^2(\overline{U})$ and
$$
		\psi_i(x) = 0 \text{~if~} x \in M_i \cap \Gamma, \quad \psi_i(x) > 0 \text{~if~} x \in U \cap \Omega^-, \quad \partial_{\bn} \psi_i \neq 0 \text{~on~} M_i \cap \Gamma.
$$
\end{assumption}
\begin{assumption}\label{ass:gamma-regularity-4}
	There exist positive constants $c_1, \ldots, c_n$ such that $\psi_i(x) > c_i$ for all $x \in \partial U \cap \overline{\Omega^-}$ and for all $i \in \{1, \ldots, n\}$.
\end{assumption}
Under  Assumptions \ref{ass:gamma-regularity-2}-\ref{ass:gamma-regularity-4}, we have the following result.

\begin{theorem}\label{thm:aux-func-construction}
	Suppose  Assumptions \ref{ass:gamma-regularity-2}-\ref{ass:gamma-regularity-4} hold and we define  $\psi: U \to \bR$ as
	$
	\psi(x) = \prod_{i=1}^n \psi_i(x).
	$
	For any constant $c$ such that $0 < c < \prod_{i=1}^n c_i$, let
	$L_c := \{ x \in U: \psi(x) \geq c \}. $
	Then the function $\phi: \bar{\Omega} \to \bR$ given by
	\begin{equation}\label{eq:def_phi}
		\phi(x) = \begin{cases}
			c^3, & \text{~if~} x \in (\overline{\Omega^-} \backslash U) \cup (\overline{\Omega^-} \cap L_c), \\
			{c^3 - }\left( c - \psi(x) \right)^3, & \text{~if~} x \in (U \cap \overline{\Omega^-}) \backslash L_c, \\
			0, & \text{~if~} x \in \overline{\Omega^+}
		\end{cases}
	\end{equation}
	is well-defined and satisfies  \eqref{eq:aux-func-prop}.
\end{theorem}

\begin{proof}
	Let
	$$D_1 := (\overline{\Omega^-} \backslash U) \cup (\overline{\Omega^-} \cap L_c), D_2 := (U \cap \overline{\Omega^-}) \backslash L_c, D_3 :=\overline{\Omega^+},$$
	then we shall show that $\phi$ is a well-defined piecewise function, i.e. $\bigcup_{i=1}^3D_i=\overline{\Omega}$ and $\phi$ has consistent function values over $\overline{D_1} \cap \overline{D_2}$, $\overline{D_2} \cap \overline{D_3}$, and $\overline{D_3} \cap \overline{D_1}$.
	
	First, observe that $D_1 \cup D_2 = \overline{\Omega^-}$, so $D_1 \cup D_2 \cup D_3 = \overline{\Omega}$.
	Moreover, $D_1$ and $D_3$ are closed in $\bR^d$.
	We next prove that (1) $\overline{D_1} \cap \overline{D_3} = \varnothing$, (2) $\overline{D_2} \cap \overline{D_3} \subset \Gamma$ and (3) $\overline{D_1} \cap \overline{D_2} \subset \partial L_c$.
	
	\begin{enumerate}
		\item
		Since $\Omega^-$ is a subdomain of $\Omega$, it is open and connected. We thus have $\overline{D_1} = D_1 \subset \overline{\Omega^-}$ and $\overline{D_3} = D_3 \subset \overline{\Omega^+}$ that $\overline{D_1} \cap \overline{D_3} \subset \Gamma$.
		Then, since  $\Gamma \subset U$ and $\Gamma \cap L_c = \varnothing$, we have $\overline{D_1} \cap \Gamma = D_1 \cap \Gamma = \varnothing$. Hence $\overline{D_1} \cap \overline{D_3} = \varnothing$.
		
		\item
		Note that
		\begin{equation}\label{eq:D2-decomp}
			\overline{D_2} \subset \overline{U} \cap \overline{\Omega^-} \cap \overline{\bR^d \backslash L_c} \subset \overline{\Omega^-}
		\end{equation}
		and $\overline{D_3} \subset \overline{\Omega^+}$, we have $\overline{D_2} \cap \overline{D_3} \subset \Gamma$.
		
		\item
		Denote $c_0 := \prod_{i=1}^n c_i$.
		If $x \in \overline{D_1} \cap \overline{D_2} = D_1 \cap \overline{D_2}$, then by the decomposition of $\overline{D_2}$ in \eqref{eq:D2-decomp}, we have $x \in \overline{U} \backslash U \subset \partial U$ or $x \in \overline{(\bR^d \backslash L_c)} \cap L_c \subset \partial L_c$.
		If $x \in \partial U$, by \Cref{ass:gamma-regularity-4} we have $\psi(x) > c_0$.
		But $x \in \overline{D_2}$ implies that $x \in \overline{\bR^d \backslash L_c}$, i.e. $\psi(x) \leq c < c_0$, contradicting to $\psi(x) > c_0$.
		We thus have that $x \in \partial L_c$.
	\end{enumerate}

	By the above claims (1)-(3), we have
	\[ \begin{aligned}
		x \in \overline{D_1} \cap \overline{D_2} \implies x \in \partial L_c \implies \psi(x) = c \implies c^3 = c^3 - (c - \psi(x))^3, \\
		x \in \overline{D_2} \cap \overline{D_3} \implies x \in \Gamma \implies \psi(x) = 0 \implies c^3 - (c - \psi(x))^3 = 0.
	\end{aligned}  \]
	Moreover, $\overline{D_1} \cap \overline{D_3} = \varnothing$.
	Therefore, the piecewise-definition of $\phi$ is consistent, and hence $\phi$ is well-defined.
	
	Since $\phi$ is smooth in $D_1, D_2, D_3$ and clearly continuous on $\Gamma$ and $\partial L_c$, we have $\phi|_{\Omega^+} \in C^2(\overline{\Omega^+})$ and $\phi \in C(\overline{\Omega})$.
	The first and second order derivatives of $\phi(x)$ all tends to zero as $x$ approaches $\partial L_c$ in $D_2$, so $\phi|_{\Omega^-} \in C^2(\overline{\Omega^-})$.
	
	By \Cref{ass:gamma-regularity-3} and  \eqref{eq:def_phi}, it is clear that $\phi(x) = 0$ for all $x \in \Gamma$ and $\phi(x) > 0$ on $\Omega^-$.
	Finally, we have
	\[ [\beta \partial_{\bn} \phi]_\Gamma(x) = -3 \beta^- (c - \psi(x))^2 \left( \sum_{i=1}^n (\partial_{\bn} \psi_i)(x) \prod_{j=1, j \neq i}^n \psi_j(x) \right), \quad \forall x \in \Gamma.  \]
	Since $c > 0$, by \Cref{ass:gamma-regularity-3}, $[\beta \partial_{\bn} \phi]_\Gamma(x)$ equals to $0$ if and only if there exists at least two distinct indexes $i, j \in \{1,\ldots,n\}$ such that $\psi_i(x) = \psi_j(x) = 0$.
	However, it follows from \Cref{ass:gamma-regularity-2} that $M_i \cap M_j$ is of measure zero for all $i, j$ with $i\neq j$, so the set $\{x \in \Gamma :  \psi_i(x) = \psi_j(x) = 0\}$ is also of measure zero.
	This shows that $[\beta \partial_{\bn} \phi]_\Gamma \neq 0$ a.e. on $\Gamma$.
\end{proof}

\Cref{thm:aux-func-construction} provides a generic method for constructing an auxiliary function $\phi(x)$ when Assumptions \ref{ass:gamma-regularity-2}-\ref{ass:gamma-regularity-4} are satisfied.
This method is independent of the PDE and is only related to the shape of the interface $\Gamma$, and can be easily applied to other types of  interface problems with slight modifications.

Note that although Assumptions \ref{ass:gamma-regularity-2}-\ref{ass:gamma-regularity-4} look complicated, they can be satisfied by a large class of  interfaces, which are of great practical interest.
We present an example in  \Cref{ex:aux-func-ex-3} below for explanations.
Note that the interfaces in \Cref{ex:aux-func-ex-3} have complex geometry and are challenging to be addressed by traditional mesh-based numerical methods.

\begin{example}\label{ex:aux-func-ex-3}
	(Star-shaped interfaces)
	Let $\Omega\subset\mathbb{R}^2$ be a bounded domain and the star-shaped interface $\Gamma\subset \Omega$ be defined  by the zero level set of the following function in polar coordinates:
	$\psi(r, \theta) = r - a - b \sin(5\theta)$
	with constants $b < a$.
	The domain $\Omega$ is divided into $\Omega^- = \{ (r, \theta) \in \bR^2 : r < a + b \sin(5\theta)\}$ and $\Omega^+ = \{ (r, \theta) \in \Omega : r > a + b \sin(5\theta) \}$.
	Note that $\psi(r, \theta)$ is not differentiable on $\Omega$, since the polar angle is not differentiable at the origin.
	In this case,  it follows from \Cref{thm:aux-func-construction} that we can define
	\small
	\begin{equation*}
		\phi(r, \theta) = \begin{cases}
			\left(\frac{a-b}{2} \right)^3, & \text{~if~} a + b \sin(5\theta) - r \geq \frac{a-b}{2}, \\
			{\left(\frac{a-b}{2} \right)^3 - \left(\frac{a-b}{2} + \psi(r, \theta) \right)^3}, & \text{~if~} 0 < a + b \sin(5\theta) - r < \frac{a-b}{2}, \\
			0, & \text{~otherwise},
		\end{cases}
	\end{equation*}
     \normalsize
	and one can check that $\phi$ satisfies \eqref{eq:aux-func-prop}.
\end{example}

\begin{remark}
	We mention that, if it is difficult to construct an auxiliary function $\phi(x)$ with an analytic form, we can   train a DCSNN to represent $\phi(x)$.  To this end, we impose the constraints
	$ [\phi]_\Gamma = 0, \quad [\beta \partial_{\bn} \phi]_\Gamma = \gamma, $
	where the function $\gamma: \Gamma \to \bR$ is nonzero almost everywhere.
	Then,  we train a DCSNN $\hat{\phi}(x, z; \theta_\phi)$ with smooth activation functions by minimizing the following loss function:
	\small
	\begin{equation}\label{eq:loss_phi}
		\begin{aligned}
			& \frac{w_1}{M_{\Gamma}}\sum_{i=1}^{M_\Gamma} \left|\hat{\phi}(x_\Gamma^i, 1; \theta_{\phi}) - \hat{\phi}(x_\Gamma^i, -1; \theta_{\phi})\right|^2 \\
			& \quad + \frac{w_2}{M_{\Gamma}} \sum_{i=1}^{M_\Gamma} \left| \bn \cdot (\beta^+ \nabla_x \hat{\phi}(x_i^{\Gamma}, 1; \theta_{\phi}) - \beta^- \nabla_x \hat{\phi}(x_i^{\Gamma}, -1; \theta_{\phi})) - \gamma (x_{\Gamma}^i) \right|^2,
		\end{aligned}
	\end{equation}
	\normalsize
	where  $\{x_\Gamma^i\}_{i=1}^{M_\Gamma} \subset \Gamma$ and $w_1,w_2>0$ are the weights.
	It is clear that the trained $\hat{\phi}$ satisfies the smoothness requirements in \eqref{eq:aux-func-prop}.
\end{remark}

\subsection{The hard-constraint PINN method for (\ref{eq:ocip-elliptic})-(\ref{eq:control-constraint-elliptic})}\label{subsec:pinn-hc}

In this subsection, we propose a hard-constraint PINN method for problem \eqref{eq:ocip-elliptic}-\eqref{eq:control-constraint-elliptic} based on the discussions in \Cref{subsec:hc-bc,subsec:construct-aux-func}. For this purpose, we first approximate the state variable $y$ and the adjoint variable $p$ by the neural networks $\hat{y}(x;\theta_y)$ and $\hat{p}(x;\theta_p)$ given in  \eqref{eq:neural-form-hard-bcij} and  \eqref{eq:neural-form-p}, respectively. As a result, the boundary and interface conditions  in  \eqref{eq:adjoint} and  \eqref{eq:state_reduced} are satisfied automatically.
Then,  equations  \eqref{eq:adjoint} and  \eqref{eq:state_reduced} can be solved by minimizing the following loss function:
\begin{equation}\label{eq:loss-elliptic-hc}
	{\small
		\begin{aligned}
			\cL_{HC}(\theta_y, \theta_p) = & \frac{w_{y, r}}{M} \sum_{i=1}^{M} \left| -\Delta_x \hat{y}(x^i; \theta_{y})-\frac{\mathcal{P}_{[u_a(x^i), u_b(x^i)]}(-\frac{1}{\alpha} \hat{p}(x^i; \theta_p)) + f(x^i)}{\beta^\pm} \right| ^2 \\
			& + \frac{w_{y, \Gamma_n}}{M_{\Gamma}} \sum_{i=1}^{M_\Gamma} \left| [\beta\partial_{\bm{n}}\hat{y}]_\Gamma (x_{\Gamma}^i; \theta_{y}) - g_1(x_\Gamma^i) \right|^2 \\
			& + \frac{w_{p, r}}{M} \sum_{i=1}^{M} \left|-\Delta_x \hat{p}(x_i; \theta_{p})-\frac{\hat{y}(x_i; \theta_{y}) - y_d(x_i)}{\beta^\pm} \right|^2  + \frac{w_{p, \Gamma_n}}{M_{\Gamma}}\sum_{i=1}^{M_\Gamma} \left|[\beta\partial_{\bm{n}}\hat{p}]_\Gamma(x_i^{\Gamma}; \theta_{p}) \right|^2.
		\end{aligned}
	}
\end{equation}
Clearly, the loss function $\cL_{HC}$ can be calculated by \eqref{eq:neural-form-hard-bcij} and \eqref{eq:ij-gradient-comp}-\eqref{eq:neural-form-p-derivatives}.  Similar to  \eqref{eq:loss-elliptic}, the loss function $ \mathcal{L}_{HC}(\theta_y, \theta_p) $ in  \eqref{eq:loss-elliptic-hc} can be minimized by a stochastic optimization method,  where all the derivatives $\Delta_x \hat{y}$, $\nabla_x\hat{y}$, $\Delta_x \hat{p}$, $\nabla_x\hat{p}$ and the gradients $\frac{\partial \mathcal{L}_{HC}}{\partial \theta_{y}}$, $\frac{\partial \mathcal{L}_{HC}}{\partial \theta_{p}}$ are computed by automatic differentiation.

We remark that if $g, h$, and $\phi$ cannot be constructed with analytical forms, then we can represent them by training three neural networks $\hat{g}$, $\hat{h}$, and $\hat{\phi}$ as shown in \eqref{eq:loss_g}, \eqref{eq:loss_h}, and \eqref{eq:loss_phi}.  Note that these neural networks are expected to be easy to train due to the simple structures of the loss functions \eqref{eq:loss_g}, \eqref{eq:loss_h}, and \eqref{eq:loss_phi}. More importantly, with the pre-trained $\hat{g}$, $\hat{h}$, and $\hat{\phi}$,  the boundary and interface conditions is decoupled from the learning of the PDE. Hence, this hard-constraint approach is still superior since it can reduce the training difficulty and improve the numerical accuracy of \Cref{alg:pinn}.

We summarize the proposed hard-constraint PINN method for \eqref{eq:ocip-elliptic}-\eqref{eq:control-constraint-elliptic} in \Cref{alg:pinn-hc}. 
We provide two options for inputting the functions $g$, $h$, and $\phi$ in \eqref{eq:neural-form-hard-bcij} and \eqref{eq:neural-form-p}: In Option I, the input functions $g$, $h$, and $\phi$ have closed-form expressions, while in Option II, they are numerically approximated by pretrained neural networks.
We reiterate that  the boundary and interface conditions for $y$ and $p$ are imposed as hard constraints in the neural networks $\hat{y}(x; \theta_y)$ and $\hat{p}(x; \theta_p)$. Hence, compared with \Cref{alg:pinn},  \Cref{alg:pinn-hc} reduces the numerical error at the boundary and interface and is easier and cheaper to implement.

\begin{algorithm}[ht]
	\caption{The hard-constraint PINN method for \eqref{eq:ocip-elliptic}-\eqref{eq:control-constraint-elliptic}}
	\begin{algorithmic}[1]
		{\renewcommand{\algorithmicrequire}{\textbf{Input (Option I):}}
		\REQUIRE Weights $w_{y, r}, w_{p, \Gamma_n}, w_{p, r}, w_{p, \Gamma_n}$, functions $g$, $h$, and $\phi$ with closed-form expressions and satisfying \eqref{eq:g-func-prop}, \eqref{eq:h-func-prop}, \eqref{eq:aux-func-prop}, respectively.}
		{\renewcommand{\algorithmicrequire}{\textbf{Input (Option II):}}
		\REQUIRE Weights $w_{y, r}, w_{p, \Gamma_n}, w_{p, r}, w_{p, \Gamma_n}$, pretrained neural networks $\hat{g}$, $\hat{h}$, and $\hat{\phi}$ that approximate any functions satisfying \eqref{eq:g-func-prop}, \eqref{eq:h-func-prop}, \eqref{eq:aux-func-prop}, respectively.}
		\vspace{-1em}
		\STATE Initialize the neural networks $\cN_y(x, \phi(x); \theta_y)$ and $\cN_p(x, \phi(x); \theta_p)$ with parameters $\theta_y^0$ and $\theta_p^0$.
		\STATE Sample training sets $\cT = \{x^i\}_{i=1}^M \subset \Omega$ and $\cT_\Gamma = \{x_\Gamma^i\}_{i=1}^{M_\Gamma} \subset \Gamma$.
		\STATE Calculate the function values of $g, h, \phi$ and their first and second order derivatives over $\cT$, and the values of $[\beta \partial_{\bn} g]_\Gamma, [\beta \partial_{\bn} \phi]_\Gamma$ over $\cT_\Gamma$.
		\STATE Train the neural networks $\hat{y}(x; \theta_y)$ and $\hat{p}(x; \theta_p)$ to identify the optimal parameters $\theta_y^*$ and $\theta_p^*$ by minimizing \eqref{eq:loss-elliptic-hc}.
		\STATE$\hat{u}(x) \gets \min\{u_b(x), \max\{u_a(x), -\frac{1}{\alpha} \hat{p}(x;\theta_p^*)\}\}$.\\
    \ENSURE Approximate solutions $\hat{y}(x;\theta^*_y)$ and $\hat{u}(x)$.
	\end{algorithmic}
	\label{alg:pinn-hc}
\end{algorithm}

\section{Numerical results}\label{sec:ocip-elliptic-numexp}

In this section, we report some numerical results of  \Cref{alg:pinn,alg:pinn-hc} for solving problem (\ref{eq:ocip-elliptic})-(\ref{eq:control-constraint-elliptic}) and numerically verify the superiority of   \Cref{alg:pinn-hc} to  \Cref{alg:pinn}.  All the codes of our numerical experiments were written with PyTorch (version 1.13) and are available at \url{https://github.com/tianyouzeng/PINNs-interface-optimal-control}. In particular, we use the hyperbolic tangent function as the active functions in all the neural networks.

To test the accuracy of the results computed by \Cref{alg:pinn,alg:pinn-hc}, we select $256 \times 256$ testing points $\{x^i\}_{i=1}^{M_T} \subset \Omega$ following the Latin hypercube sampling \cite{mckay2000comparison}. We then compute
\begin{equation}\label{eq:abs-rel-error-def}
	\varepsilon_{\text{abs}} = \sqrt{\frac{1}{M_T} \sum_{i=1}^{M_T} (\hat{u}(x^i) - u^*(x^i))^2},~\text{and}~ \varepsilon_{rel} = \varepsilon_{abs} \sqrt{A(\Omega)} / ||u^*||_{L^2(\Omega)}
\end{equation}
as the absolute and relative  $L^2$-errors of $\hat{u}$, where $A(\Omega)$ is the area of $\Omega$ (i.e. the Lebesgue measure of $\Omega$), and $||u^*||_{L^2(\Omega)}$ is computed using the numerical integration function \texttt{dblquad} implemented in the SciPy library of Python.

\medskip

\noindent\textbf{Example 1.}\exmplabel{ex:elliptic-reg-cc} 
We first demonstrate an example of (\ref{eq:ocip-elliptic})-(\ref{eq:control-constraint-elliptic}), where $ U_{ad} = \{u \in L^2(\Omega): -1 \leq u \leq 1 \text{~a.e.~in~} \Omega\}$.
	We set $\Omega = (-1, 1) \times (-1, 1)$, $\Gamma = \{x \in \Omega: \lVert x \rVert_2 \leq r_0\} \subset \Omega$ with $r_0 = 0.5$,  $\alpha = 1$, $\beta^- = 1$ and $\beta^+ = 10$. 
	We further choose $g_0=0$, $g_1=0$ and $h_0= \dfrac{(x_1^2+x_2^2)^{3/2}}{\beta^+} + (\dfrac{1}{\beta^-} - \dfrac{1}{\beta^+}) r_0^3$. Following \cite{zhang2015immersed}, we let
	\begin{equation}\label{eq:exact_elliptic_p}
		p^*(x_1, x_2) = \begin{cases}
			-{5(x_1^2 + x_2^2 - r_0^2)(x_1^2 - 1)(x_2^2-1)}/{\beta^-} ~ \text{in} ~ \Omega^-, \\
			-{5(x_1^2 + x_2^2 - r_0^2)(x_1^2 - 1)(x_2^2-1)}/{\beta^+} ~ \text{in} ~ \Omega^+,
		\end{cases}
	\end{equation}
	\begin{equation}\label{eq:exact_elliptic_y}
		y^*(x_1, x_2) = \begin{cases}
			{(x_1^2+x_2^2)^{3/2}}/{\beta^-}, ~ \text{in} ~ \Omega^-, \\
			{(x_1^2+x_2^2)^{3/2}}/{\beta^+} + ({1}/{\beta^-} - {1}/{\beta^+}) r_0^3, ~ \text{in} ~ \Omega^+,
		\end{cases}
	\end{equation}
	\begin{equation}\label{eq:exact_elliptic_u}
	u^*(x_1, x_2) = \max\{-1, \min\{1, -\frac{1}{\alpha} p^*(x_1, x_2)\}\}.
	\end{equation}
	and
	\begin{equation*}
		\begin{aligned}
			& f(x_1, x_2) = -u^*(x_1, x_2) - \nabla \cdot (\nabla \beta^\pm y^*(x_1, x_2)) ~\text{in}~\Omega^\pm,\\
			& y_d(x_1, x_2) = y^*(x_1, x_2) + \nabla \cdot (\nabla \beta^\pm p^*(x_1, x_2)) ~\text{in}~\Omega^\pm.
		\end{aligned}
	\end{equation*}
	Then, it is easy to verify that $(u^*,y^*)^\top$ is the unique solution of this example.

	To implement the soft-constraint PINN method in  \Cref{alg:pinn}, we approximate $y$ and $p$ respectively by two fully connected neural networks $\hat{y}(x, z; \theta_y)$ and $\hat{p}(x, z; \theta_p)$, where $z$ is the augmented coordinate variable in the DCSNN (see Section 2.2).
	All the neural networks consist of only one hidden layer with $100$ neurons and \texttt{tanh} activation functions.
	To implement the hard-constraint PINN method in   \Cref{alg:pinn-hc}, we define two neural networks $\cN_y(x, \phi(x); \theta_y)$ and $\cN_p(x, \phi(x); \theta_p)$ with the same structures as those of $\hat{y}$ and $\hat{p}$.
	Then the state and adjoint variables are approximated by the neural networks given in \eqref{eq:neural-form-hard-bcij} and  \eqref{eq:neural-form-p}.
	We consider both input Options I and II in \Cref{alg:pinn-hc}. 
	For Option I, we choose $g(x) = \frac{(x_1^2 x_2^2 + 1)^{3/2}} {\beta^+} + (\frac{1}{\beta^-} - \frac{1}{\beta^+}) r_0^3$ and $h(x) = (x_1^2 - 1)(x_2^2 - 1)$.
	The auxiliary function $\phi$ is defined analytically by
	\begin{equation}\label{eq:def_phi_ex}
		\phi(x) = \begin{cases}
			4 \lVert x \rVert _2^2, & \text{~if~} x \in \Omega^-, \\
			1, & \text{~otherwise}.
		\end{cases}
	\end{equation}
	It is easy to check that these functions satisfy the requirements in \eqref{eq:g-func-prop}, \eqref{eq:h-func-prop} and \eqref{eq:aux-func-prop}.
	For Option II, we first randomly sample the training sets $\mathcal{T}=\{x^i\}_{i=1}^{M} \subset \Omega$, $\mathcal{T}_B=\{x_B^i\}_{i=1}^{M_B} \subset \partial \Omega$ and $\mathcal{T}_{\Gamma}=\{x_\Gamma^i\}_{i=1}^{M_\Gamma} \subset \Gamma$ with $M = 1024$ and $M_B = M_\Gamma = 256$.
	Then we train two fully connected shallow neural networks\footnote{\blue{Here, due to the homogeneous interface condition $g_0 = 0$, it suffices to approximate $\hat{g}$ by a fully connected neural network, rather than by a DSCNN as described in the previous section.}} $\hat{g}(x; \theta_g)$ and $\hat{h}(x; \theta_h)$ with one hidden layer, $500$ neurons and $\texttt{tanh}$ activation functions by minimizing the loss functions
	\begin{equation}\label{eq:loss_g_ex1}
		\cL_g(\theta_g) = \frac{1}{M_B} \sum_{i=1}^{M_B} \left| \hat{g}(x_B^i; \theta_g) - h_0(x_B^i) \right|^2,
	\end{equation}
	and
	\begin{equation}
		\cL_h(\theta_h) = \frac{0.01}{M} \sum_{i=1}^{M} \left| \hat{h}(x^i; \theta_h) - \bar{h}(x^i) \right|^2 + \frac{1}{M_B} \sum_{i=1}^{M_B} \left| \hat{h}(x_B^i; \theta_h) \right|^2,
	\end{equation}
	where $\bar{h}(x) := \cos(\frac{\pi}{2}x_1)\cos(\frac{\pi}{2}x_2)$.
	Each neural network is trained with the L-BFGS optimizer with stepsize $1$ and strong Wolfe condition for $200$ iterations.
	Then we train a DCSNN $\hat{\phi}(x, z; \theta_\phi)$ with one hidden layer, $200$ neurons and \texttt{tanh} activation functions by minimizing the following loss function
	\begin{equation}
		\begin{aligned}
			\cL_\phi(\theta_\phi) = & \frac{1}{M_\Gamma} \sum_{i=1}^{M_\Gamma} \left| \hat{\phi}(x_\Gamma^i, 1; \theta_\phi) - \hat{\phi}(x_\Gamma^i, -1; \theta_\phi) \right|^2 \\
			& + \frac{1}{M_\Gamma} \sum_{i=1}^{M_\Gamma} \left| \bn \cdot \nabla_x \hat{\phi}(x_\Gamma^i, 1; \theta_\phi) - 5 \right|^2 + \frac{1}{M_\Gamma} \sum_{i=1}^{M_\Gamma} \left| \bn \cdot \nabla_x \hat{\phi}(x_\Gamma^i, -1; \theta_\phi)  \right|^2.
		\end{aligned}
	\end{equation}
	The ADAM optimizer \cite{kingma2014adam} is applied to train $\hat{\phi}(x, z; \theta_\phi)$ for 30,000 iterations.
	The learning rate is initially set to be $5 \times 10^{-4}$ and finally reduces to $10^{-4}$ by a preset scheduler.
	It can be verified that the functions $\hat{g}$, $\hat{h}$ and $\hat{\phi}$ are the approximations of some functions $g$, $h$, and $\phi$ satisfying the requirements \eqref{eq:g-func-prop}, \eqref{eq:h-func-prop} and \eqref{eq:aux-func-prop}.
	
	To train the neural networks $\hat{y}(x; \theta_y)$ and $\hat{p}(x; \theta_p)$, we adopt the same training sets $\mathcal{T}$, $\mathcal{T}_B$ and $\mathcal{T}_\Gamma$ as sampled above.
	All the neural networks are initialized using the default initializer of PyTorch.
	The weights are tuned so that the magnitude of each term in the loss function \eqref{eq:loss-elliptic} or \eqref{eq:loss-elliptic-hc} is balanced. In particular,  by adjusting the weights carefully,
	we set  $w_{y, r} = w_{y, \Gamma} = w_{y, \Gamma_n} = w_{p, r} = w_{p, \Gamma} = w_{p, \Gamma_n} = 1, w_{y, b} = 2$ and $w_{p, b} = 10$ for \Cref{alg:pinn}, $w_{y, r} = 3$, $w_{y, \Gamma_n} = w_{p, r} = w_{p, \Gamma_n} = 1$ for \Cref{alg:pinn-hc} with Option I, and $w_{y, r} = w_{y, \Gamma_n} = w_{p, \Gamma_n} = 1$,$w_{p, r} = 3$ for \Cref{alg:pinn-hc} with Option II.
	
	The  ADAM optimizer is used to train the neural networks $\hat{y}(x; \theta_y)$ and $\hat{p}(x; \theta_p)$ in \Cref{alg:pinn,alg:pinn-hc}.  We fix the number of ADAM iterations to 60,000 for \Cref{alg:pinn} and \Cref{alg:pinn-hc} with Option II, and to 40,000 for \Cref{alg:pinn-hc} with Option I.
	For \Cref{alg:pinn}, the learning rate is initialized to be $ 10^{-2}$ and is reduced to $5\times 10^{-4}$ during the training by a preset scheduler.
	For Option I of \Cref{alg:pinn-hc}, the learning rate is initially set to be $ 10^{-3}$ and finally reduces to $3\times 10^{-5}$.
	For Option II of \Cref{alg:pinn-hc}, the learning rate is initially set to be $ 5\times 10^{-3}$ and finally reduces to $3 \times 10^{-4}$.

	The numerical results for Algorithms \ref{alg:pinn} and  \ref{alg:pinn-hc} are shown in \Cref{fig:ex2-pinn-training-results,fig:ex2-pinnhcnn-training-results,fig:ex2-pinnhc-training-results}, respectively. 
	All the figures are plotted over a $200 \times 200$ uniform grid in $\overline{\Omega}$. 
	The $L^2$-errors of the computed control are summarized in \Cref{tab:ex2-computed-u-error}, where the results imply that the controls computed by \Cref{alg:pinn-hc} are far more accurate than the one by \Cref{alg:pinn}.

	Moreover, it can be seen from \Cref{fig:ex2-pinn-training-results,fig:ex2-pinnhcnn-training-results,fig:ex2-pinnhc-training-results} that, for  \Cref{alg:pinn}, the numerical errors are mainly accumulated on $\partial \Omega$ and $\Gamma$, and this issue is significantly alleviated in  \Cref{alg:pinn-hc} with both Options I and II. 
	We can also see that \Cref{alg:pinn,alg:pinn-hc} are effective in dealing with the control constraint $u\in U_{ad}$. All these results validate the advantage of   \Cref{alg:pinn-hc} over  \Cref{alg:pinn} and the necessity of imposing the boundary conditions and interface conditions as hard constraints.

	\begin{table}[!ht]
		\centering
		\footnotesize
		\caption{The $L^2$-errors of the computed control $u$ of \Cref{ex:elliptic-reg-cc} for \Cref{alg:pinn,alg:pinn-hc}.}
		\begin{tabular}{c c c} 
			\toprule
			~ & $\varepsilon_{\text{abs}}$ & $\varepsilon_{\text{rel}}$ \\
			\midrule
			\Cref{alg:pinn} & $6.6853 \times 10^{-4}$ & $2.3559 \times 10^{-3}$ \\
			\Cref{alg:pinn-hc} with Option I & $6.7087 \times 10^{-5}$ & $2.3062 \times 10^{-4}$ \\
			\Cref{alg:pinn-hc} with Option II & $2.0921 \times 10^{-4}$ & $7.1919 \times 10^{-4}$ \\
			\bottomrule
		 \end{tabular}
		\normalsize
		\label{tab:ex2-computed-u-error}
	\end{table}
	
	\begin{figure}[!htpb]
		\centering
		\subfloat[Exact control $u$.]{\includegraphics[width=0.32\textwidth]{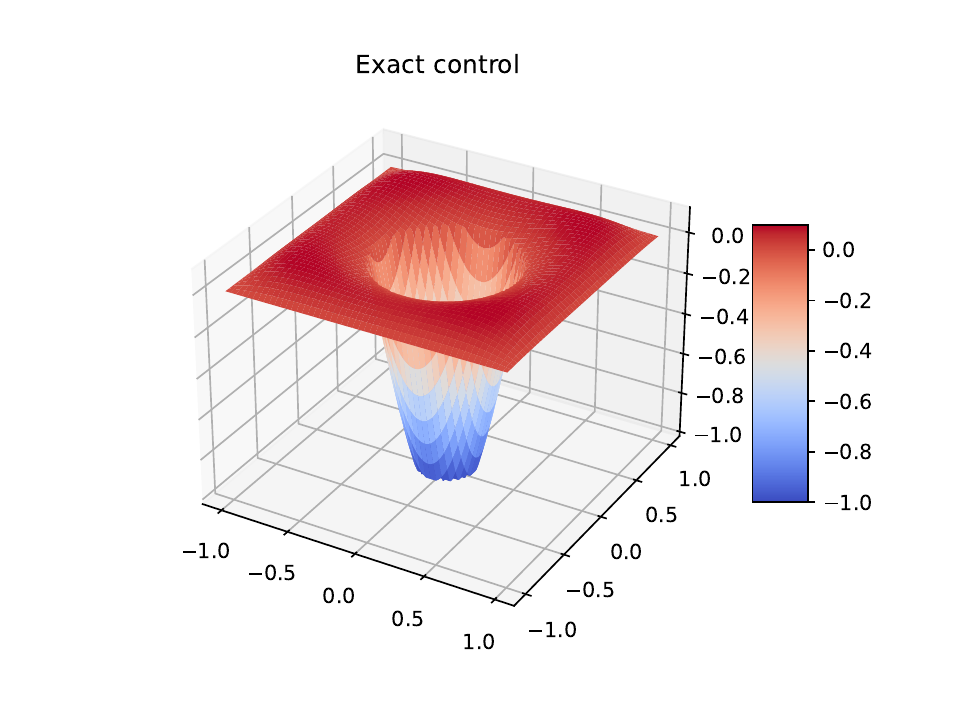}}
		\subfloat[Computed control $u$.]{\includegraphics[width=0.32\textwidth]{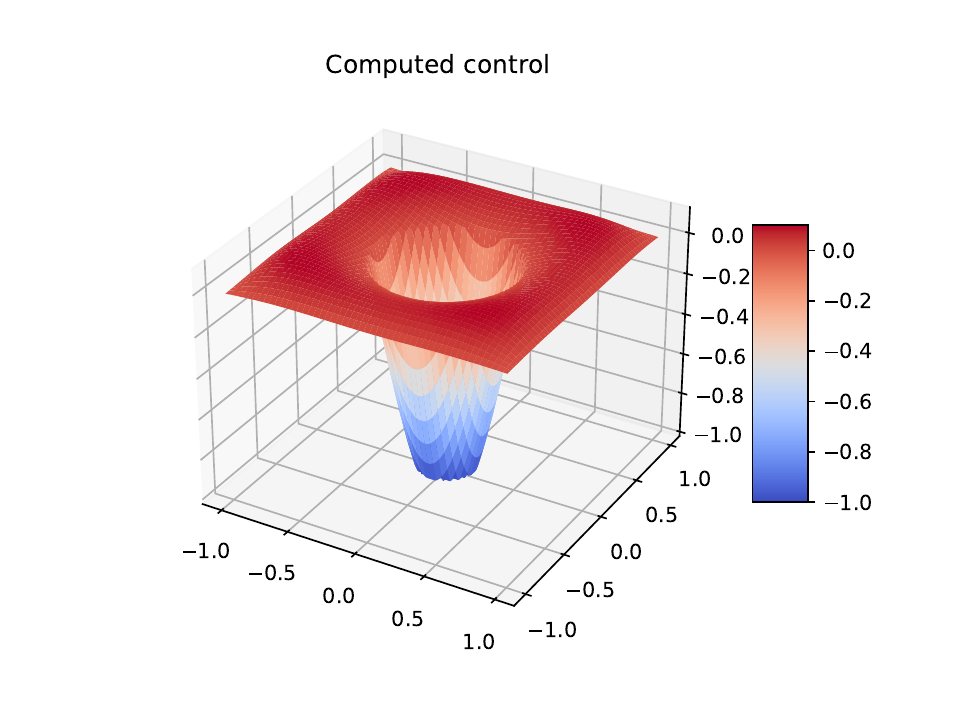}}
		\subfloat[Error of control $u$.]{\includegraphics[width=0.32\textwidth]{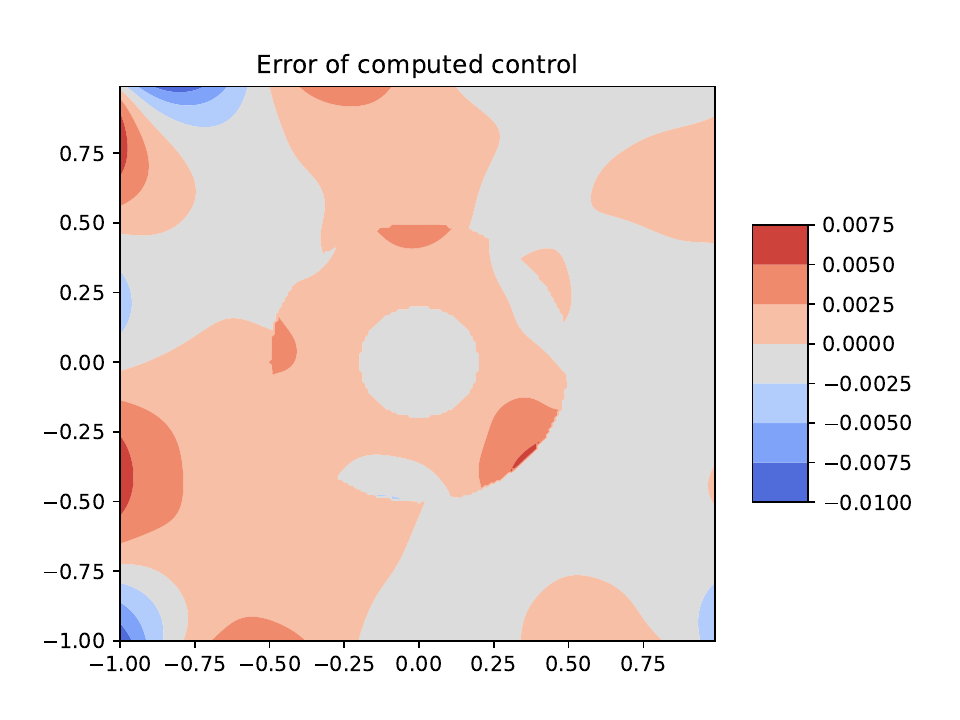}}\\
		\subfloat[Exact state $y$.]	{\includegraphics[width=0.32\textwidth]{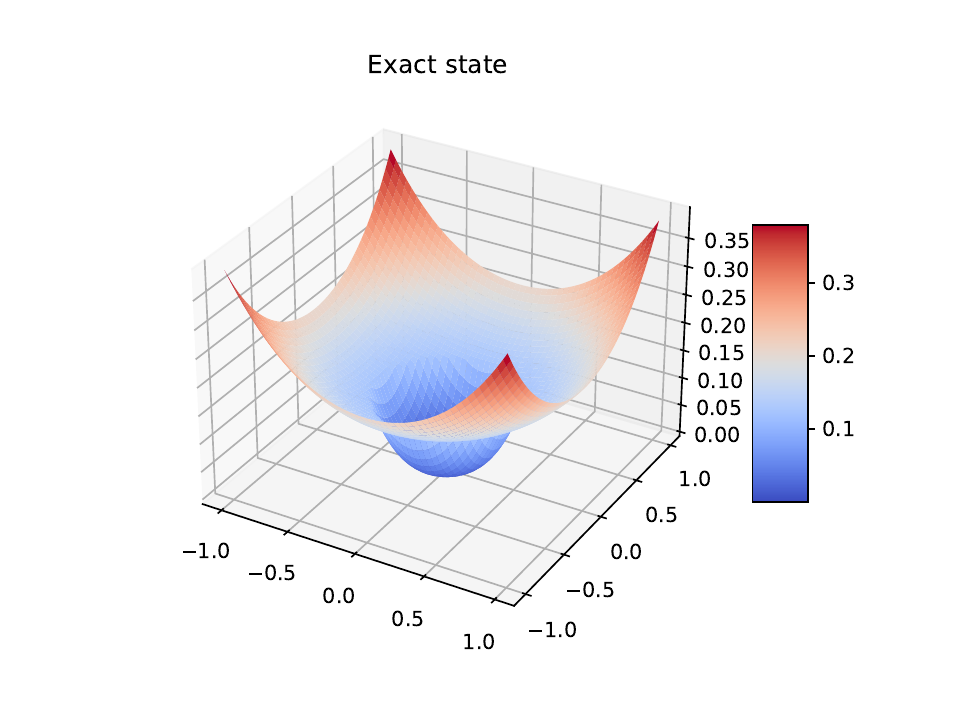}}
		\subfloat[Computed state $y$.]	{\includegraphics[width=0.32\textwidth]{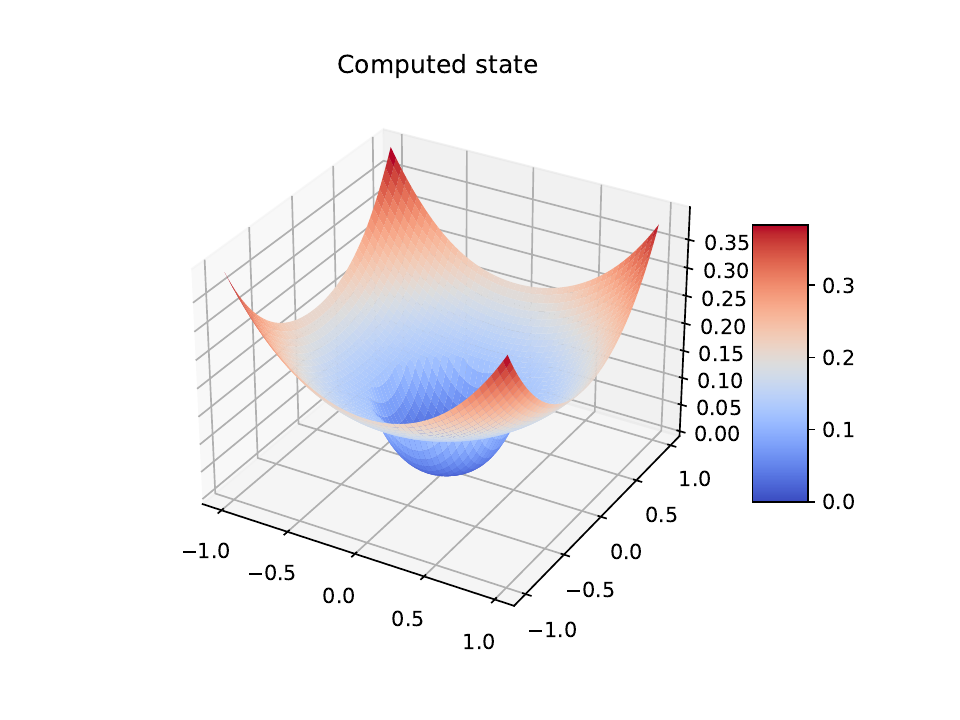}}
		\subfloat[Error of state $y$.]	{\includegraphics[width=0.32\textwidth]{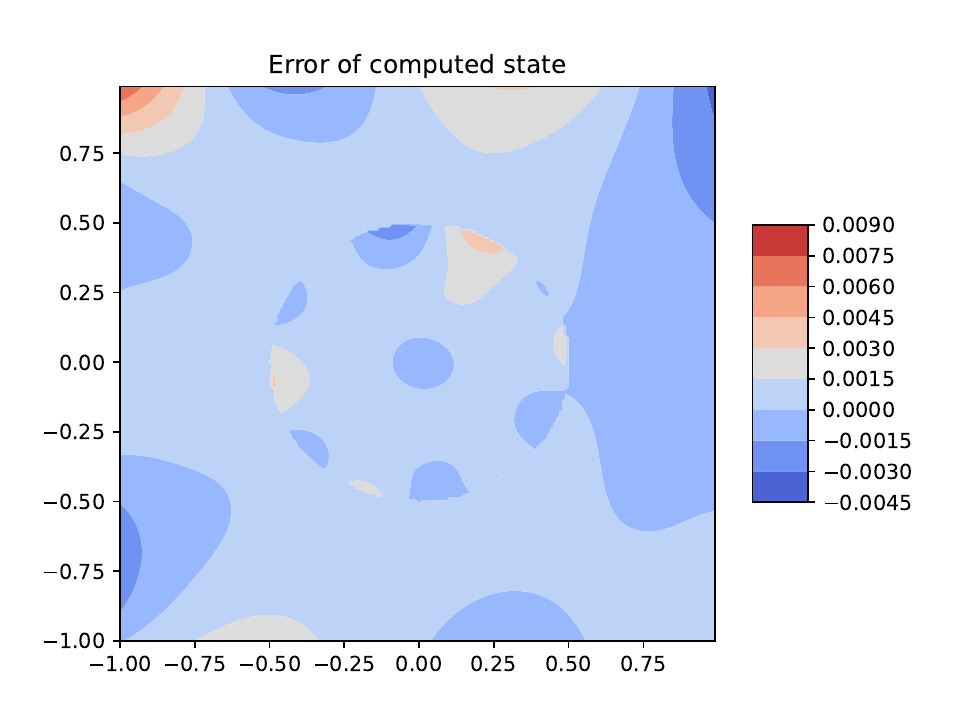}}
		\caption{Numerical results of \Cref{alg:pinn} for \Cref{ex:elliptic-reg-cc}.}
		\label{fig:ex2-pinn-training-results}
	\end{figure}

	\begin{figure}[!htpb]
		\centering
		\subfloat[Exact control $u$.]{\includegraphics[width=0.32\textwidth]{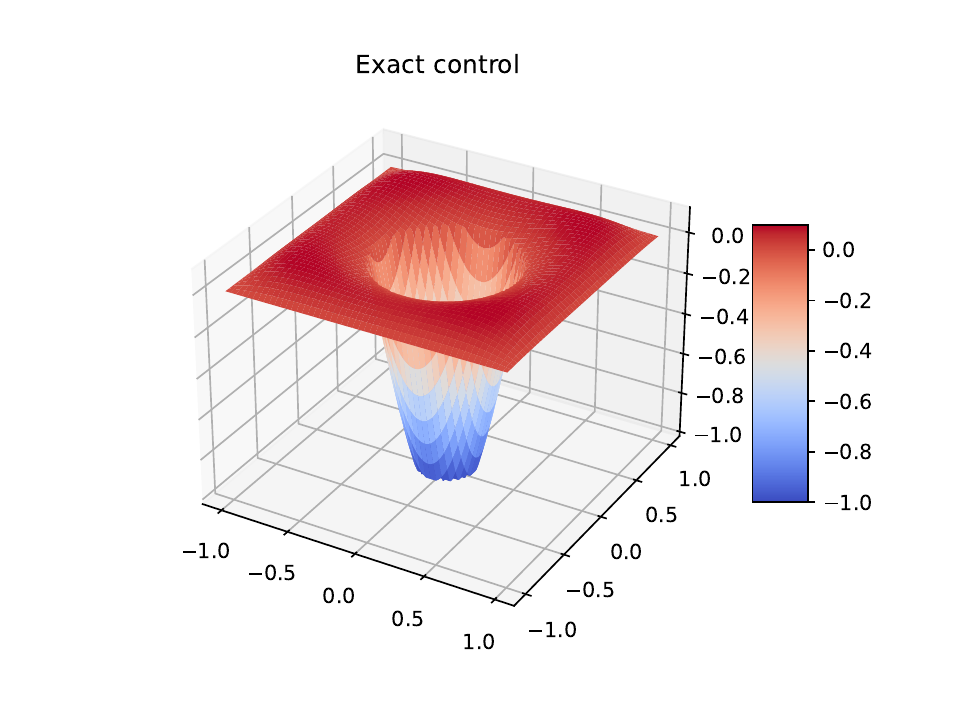}}
		\subfloat[Computed control $u$.]{\includegraphics[width=0.32\textwidth]{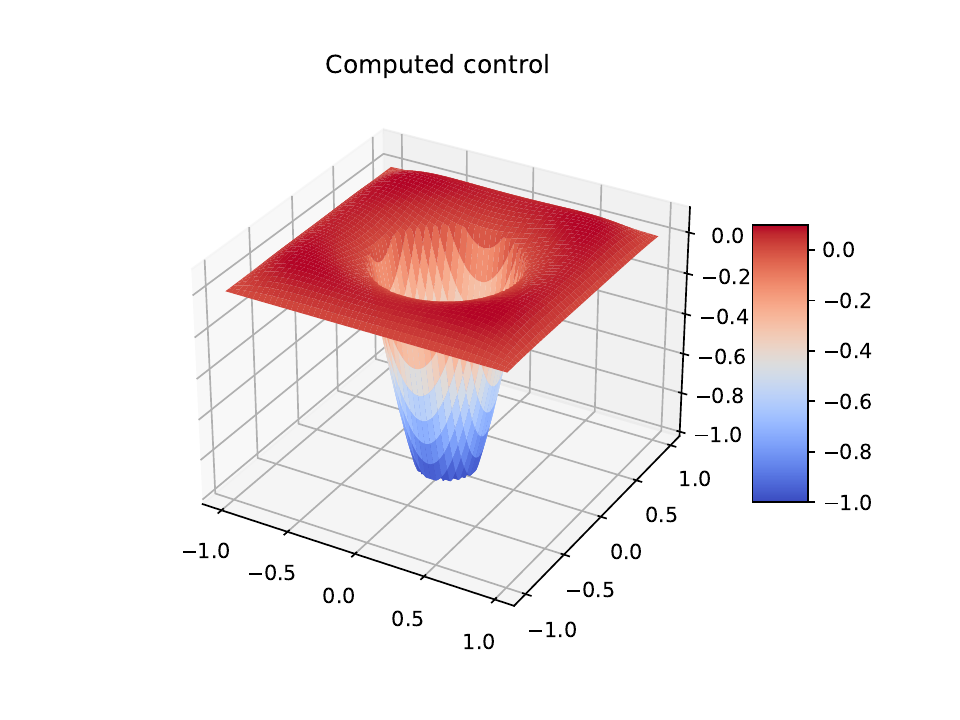}}
		\subfloat[Error of control $u$.]{\includegraphics[width=0.32\textwidth]{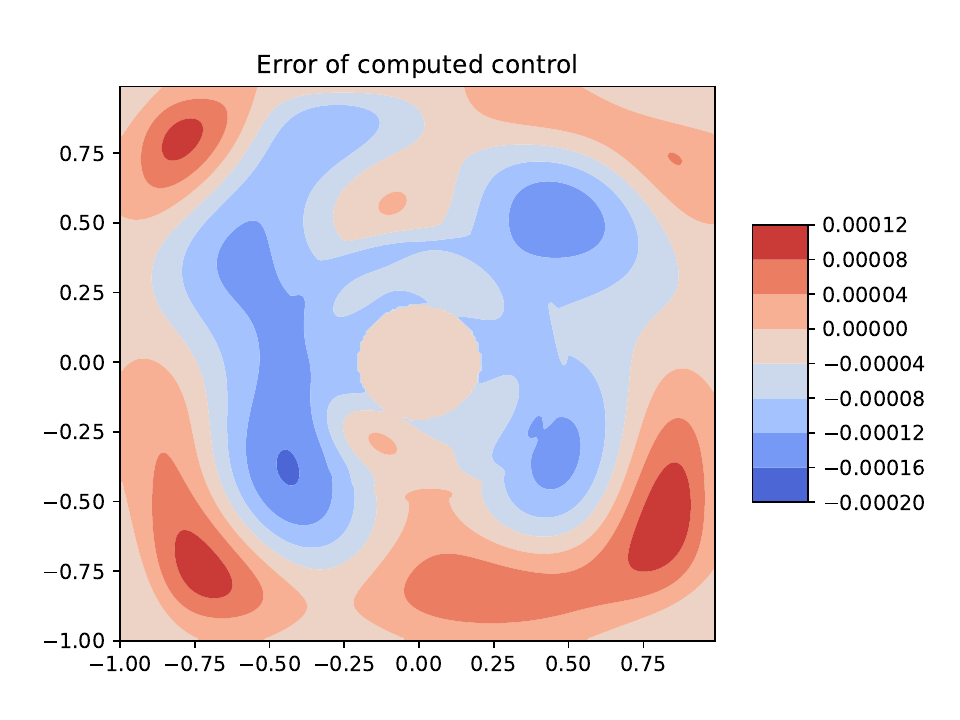}}\\
		\subfloat[Exact state $y$.]	{\includegraphics[width=0.32\textwidth]{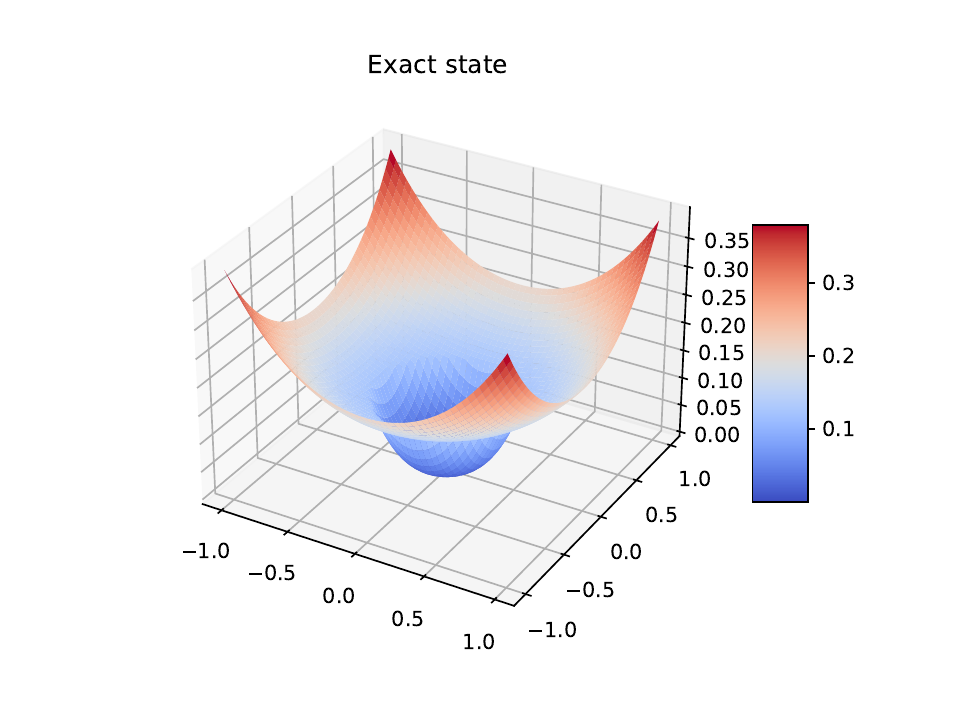}}
		\subfloat[Computed state $y$.]	{\includegraphics[width=0.32\textwidth]{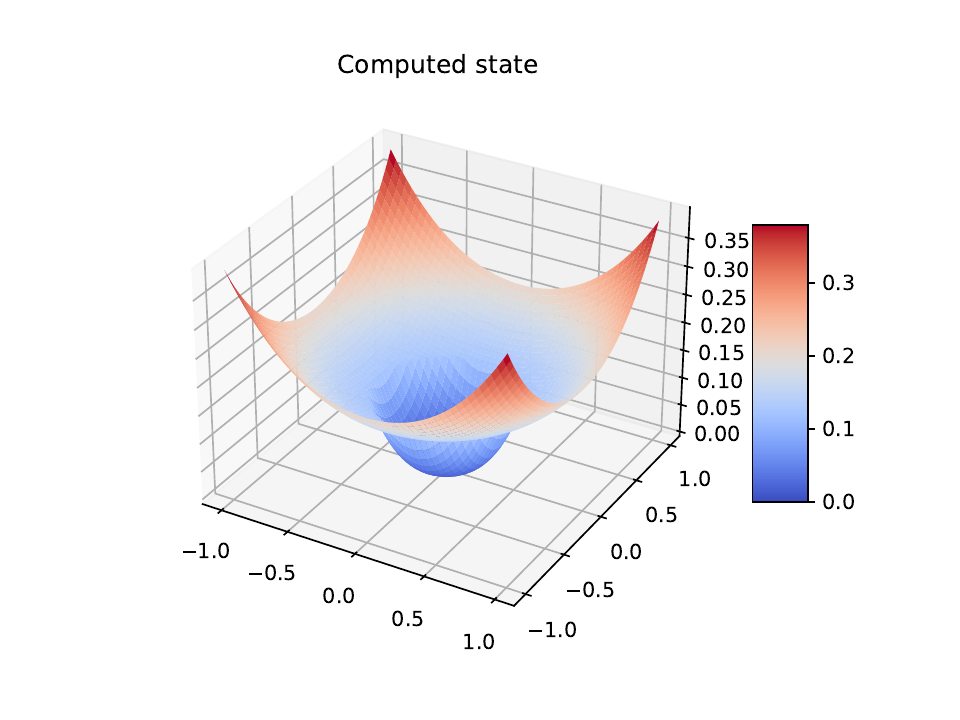}}
		\subfloat[Error of state $y$.]	{\includegraphics[width=0.32\textwidth]{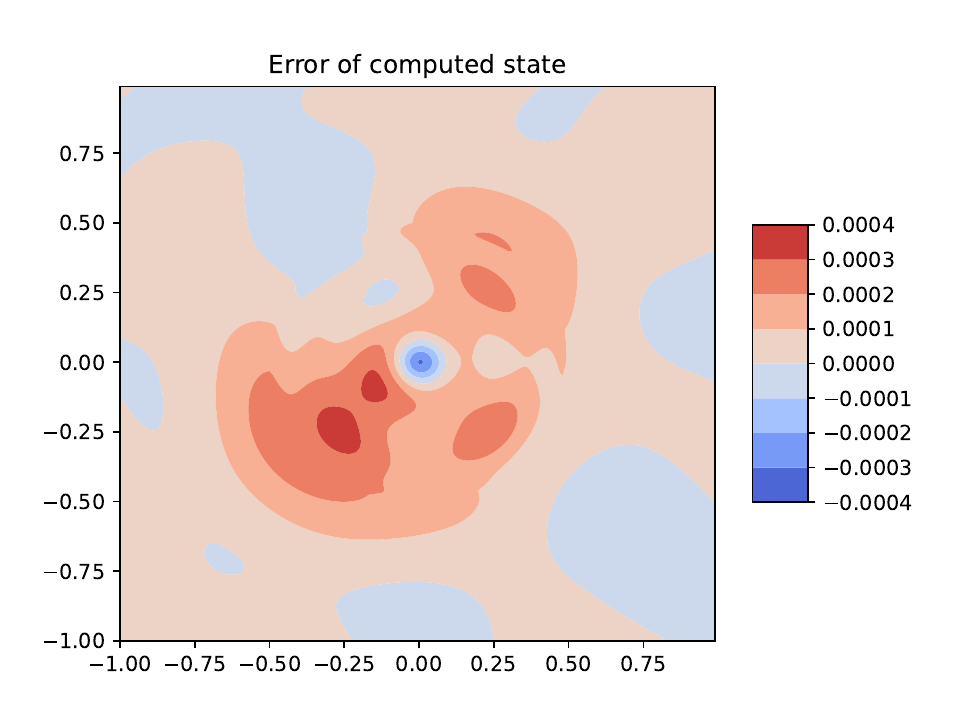}}
		\caption{Numerical results of \Cref{alg:pinn-hc} with Option I for \Cref{ex:elliptic-reg-cc}.}
		\label{fig:ex2-pinnhc-training-results}
	\end{figure}
	
	\begin{figure}[!htpb]
		\centering
		\subfloat[Exact control $u$.]{\includegraphics[width=0.32\textwidth]{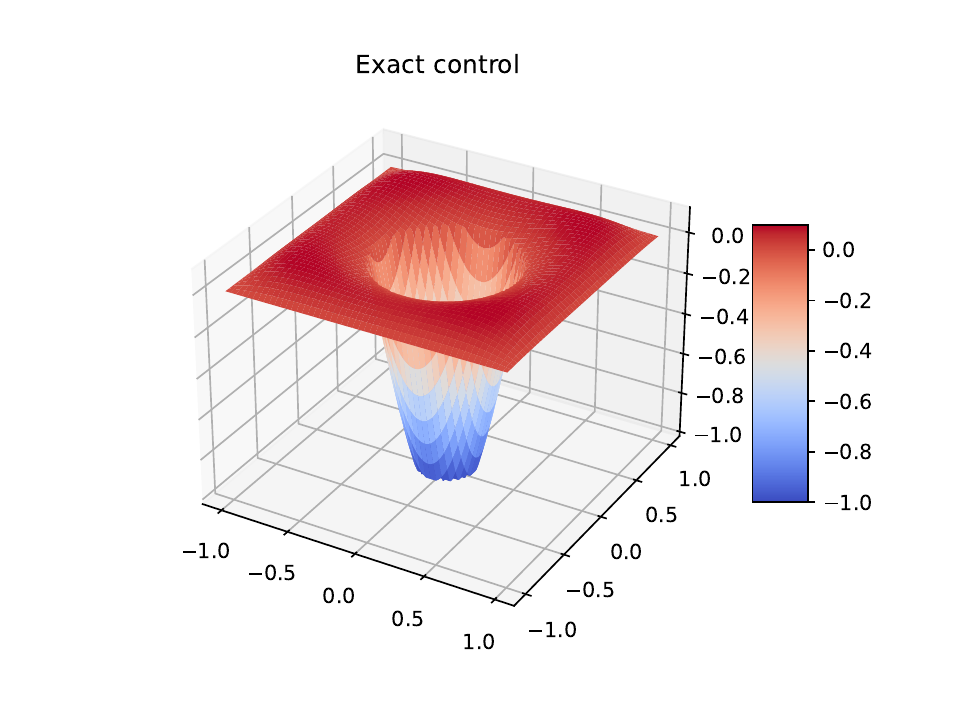}}
		\subfloat[Computed control $u$.]{\includegraphics[width=0.32\textwidth]{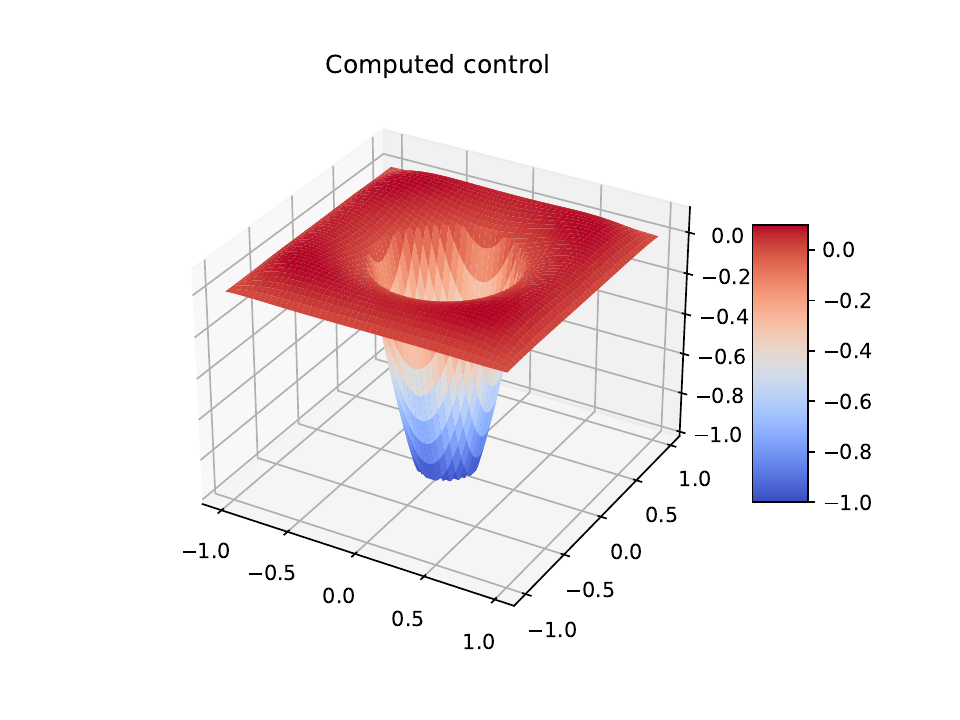}}
		\subfloat[Error of control $u$.]{\includegraphics[width=0.32\textwidth]{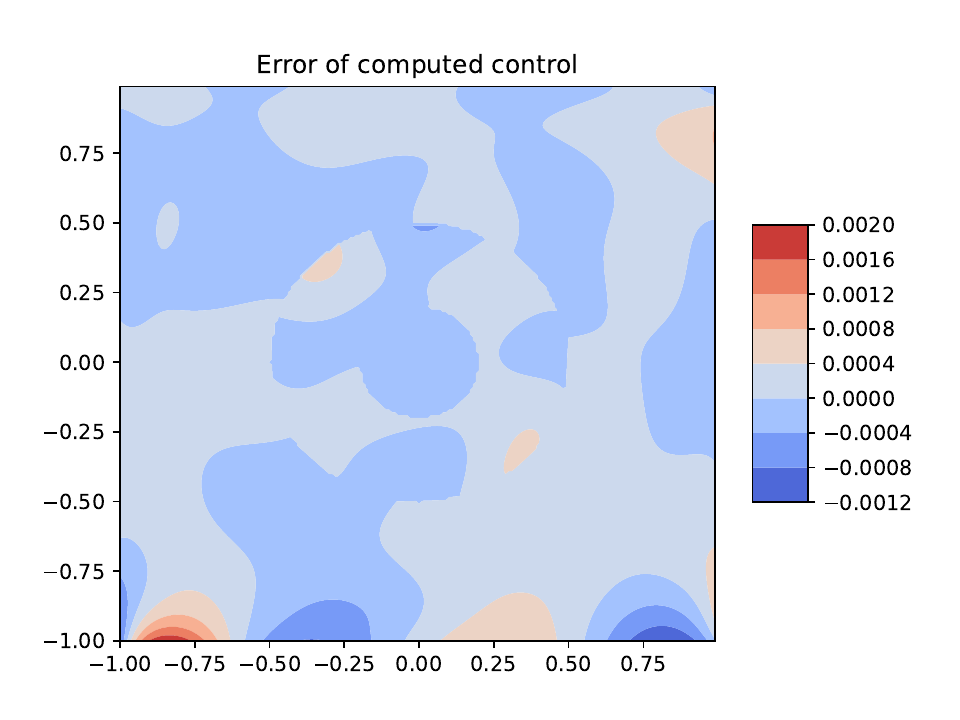}}\\
		\subfloat[Exact state $y$.]	{\includegraphics[width=0.32\textwidth]{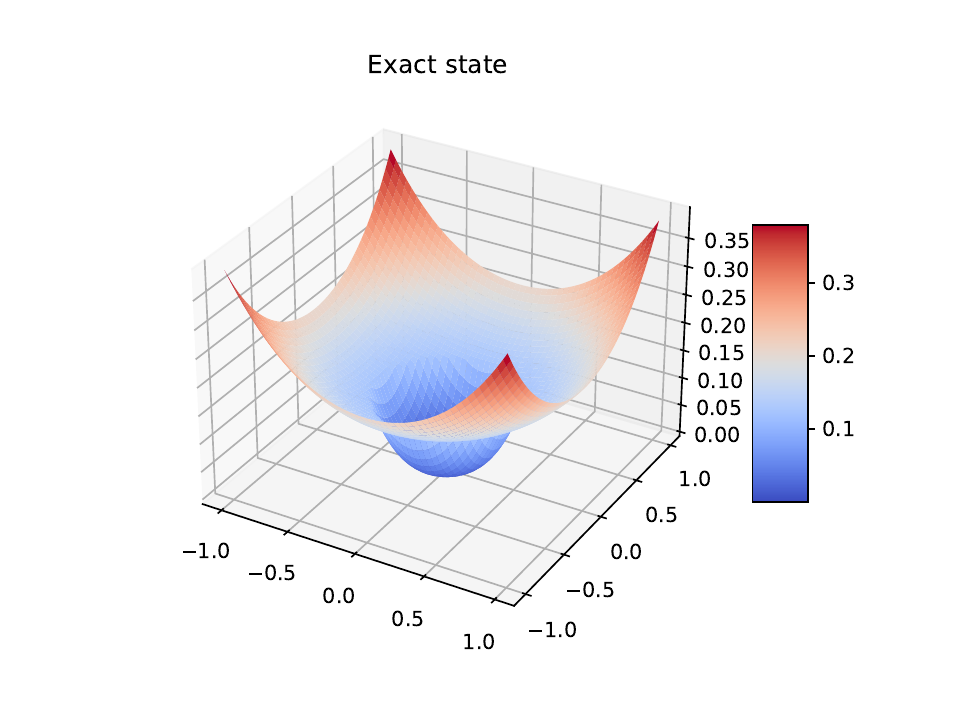}}
		\subfloat[Computed state $y$.]	{\includegraphics[width=0.32\textwidth]{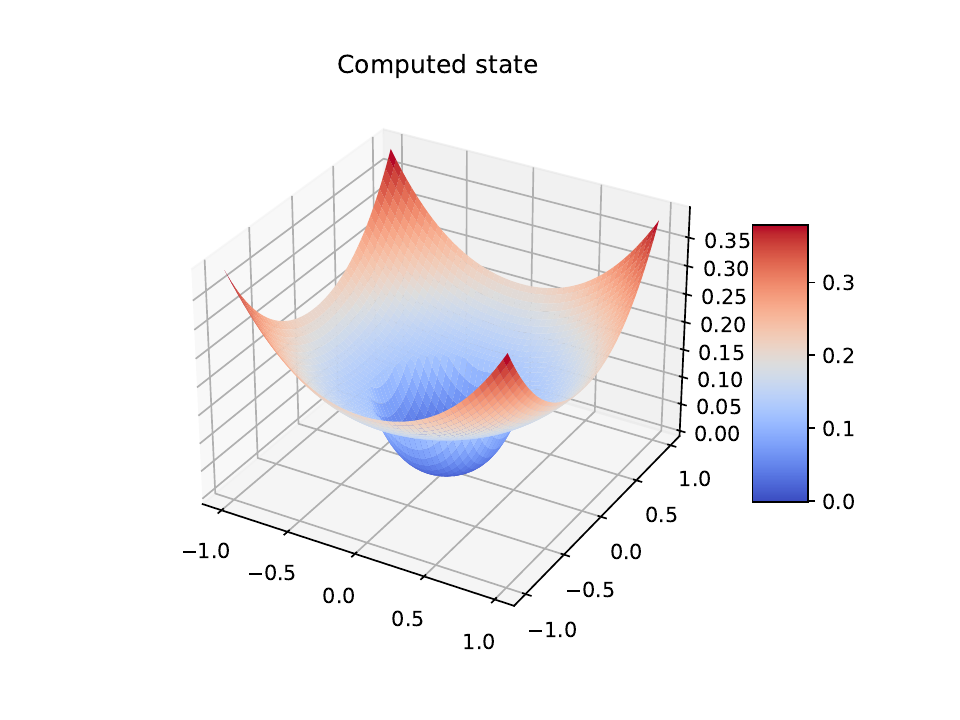}}
		\subfloat[Error of state $y$.]	{\includegraphics[width=0.32\textwidth]{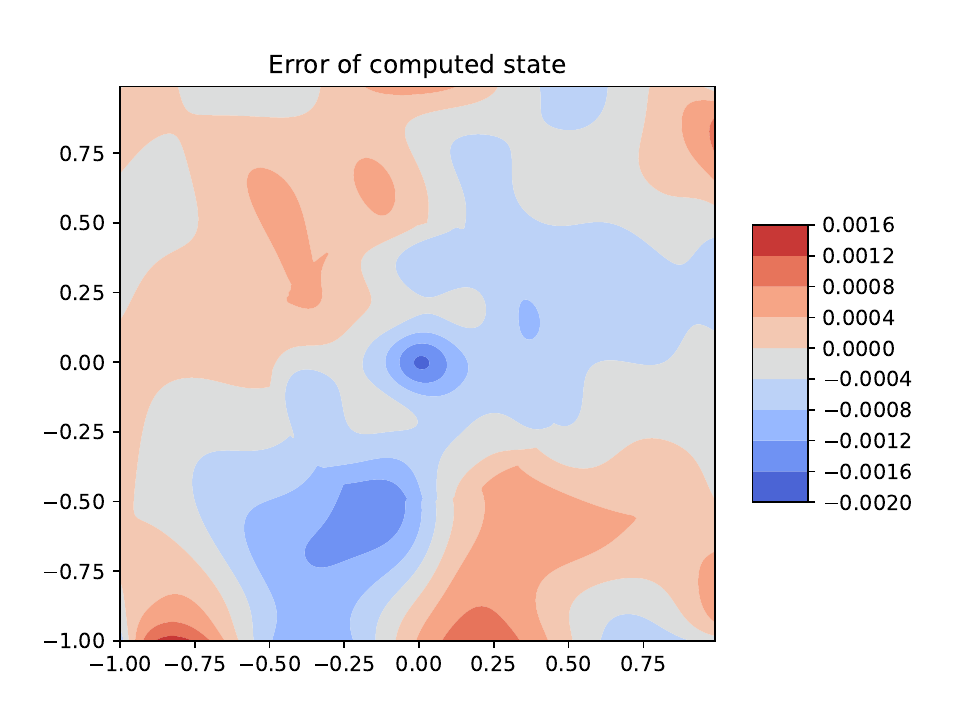}}
		\caption{Numerical results of \Cref{alg:pinn-hc} with Option II for \Cref{ex:elliptic-reg-cc}.}
		\label{fig:ex2-pinnhcnn-training-results}
	\end{figure}

To further validate the effectiveness of \Cref{alg:pinn-hc}, we compare the numerical results with the immersed finite element method (IFEM) in \cite{zhang2015immersed}, which is a benchmark mesh-based traditional numerical algorithm for solving elliptic interface optimal control problems.
For this purpose, we consider another set of PDE coefficient $\beta^- = 1$ and $\beta^+ = 5$. 
With the updated $\beta$ value, the numerical example becomes identical to the one in \cite[Example 2]{zhang2015immersed}, enabling us to directly compare the errors of the computed solutions with the ones reported in \cite[Table 4]{zhang2015immersed}.
In \cite{zhang2015immersed},  the solutions are computed on an $N \times N$ uniform grid over $\overline{\Omega}$ with different mesh resolutions $N$.
Here, we fix the training resolution $N = 32$ for \Cref{alg:pinn-hc} throughout the remaining experiments.

Again, we randomly sample the training sets $\mathcal{T}=\{x^i\}_{i=1}^{M} \subset \Omega$, $\mathcal{T}_B=\{x_B^i\}_{i=1}^{M_B} \subset \partial \Omega$ and $\mathcal{T}_{\Gamma}=\{x_\Gamma^i\}_{i=1}^{M_\Gamma} \subset \Gamma$ with $M = N^2$, $M_B = M_\Gamma = 8 \times N$ and $N = 32$.
We adopt the same configurations for the neural networks $\hat{g}(x; \theta_g)$, $\hat{h}(x; \theta_h)$, $\hat{\phi}(x, z; \theta_\phi)$, $\hat{y}(x; \theta_y)$, and $\hat{p}(x; \theta_p)$ as above.
We also apply the same optimization algorithms and stepsizes for training $\hat{g}(x; \theta_g)$, $\hat{h}(x; \theta_h)$, and $\hat{\phi}(x, z; \theta_\phi)$.
To train the neural networks $\hat{y}(x; \theta_y)$ and $\hat{p}(x; \theta_p)$ in  \Cref{alg:pinn-hc}, we fix the number of ADAM iterations to 40,000 for Option I and to 60,000 for Option II.
For Option I, the learning rate is initially set to be $1\times 10^{-3}$ and finally decreases to $3 \times 10^{-5}$ by a preset scheduler.
For Option II, the learning rate is initially set to be $5 \times 10^{-3}$ and finally reduces to $3\times 10^{-4}$.
After the above iterations, we fix the parameters in $\hat{p}(x; \theta_p)$ and train $\hat{y}(x; \theta_y)$ for 300 L-BFGS iterations with the stepsize $1$ to further improve the accuracy of $\hat{y}$

We use the $L^2$-errors defined in \cite{zhang2015immersed} to evaluate and compare the numerical accuracy of the computed solutions.
Following \cite{zhang2015immersed}, we evaluate the $L^2$-errors of the solution on an $N \times N$ uniform grid over $\overline{\Omega}$ with $N = 16$, $32$, $64$, $128$ and $256$.
The results are reported in \Cref{tab:ex1-err-comp}.

When $N = 32$, the $L^2$-errors of the computed $\hat{u}$, $\hat{y}$, and $\hat{p}$ by \Cref{alg:pinn-hc} with both Options I and II are significantly lower than those by IFEM.
Even if  the mesh resolution increases to $N = 256$,  \Cref{alg:pinn-hc} with Option I and Option II are still comparable with IFEM. Moreover, note that after training the neural networks at the resolution $N=32$, the evaluation of  \Cref{alg:pinn-hc} for a new resolution requires only a forward pass of these neural networks . In contrast, for each resolution, IFEM requires solving the elliptic interface optimal control problem from scratch, which is much more computationally expensive. 
These results validate that the mesh-free nature and the generalization ability of \Cref{alg:pinn-hc} make it effective and numerically favorable for elliptic interface optimal control problems.

\begin{table}[!ht]
	\centering
	\footnotesize
	\caption{The $L^2$-errors of the computed solutions for \Cref{ex:elliptic-reg-cc} with $\beta^- = 1$ and $\beta^+ = 5$ evaluated on different grid resolutions $N$. 
	Here, $u^*$, $y^*$, and $p^*$ are the exact solutions given in \eqref{eq:exact_elliptic_p}-\eqref{eq:exact_elliptic_u}, and $\hat{u}$, $\hat{y}$, and $\hat{p}$ are the computed solutions by each corresponding algorithm.
	The results of IFEM~\cite{zhang2015immersed} are computed and evaluated with each mesh resolution $N$, while the results of \Cref{alg:pinn-hc} are computed with fixed sampling resolution $N=32$ and evaluated with each mesh resolution $N$.}
 	\begin{tabular}{c c c c } 
 		\toprule
 		\multirow{2}[3]{*}{$N$} & \multicolumn{3}{c}{$\|\hat{u} - u^*\|_{L^2(\Omega)}$} \\ [0.5ex]
 		\cmidrule(lr){2-4}
		~ & ~~~~~~~~~IFEM \cite{zhang2015immersed}~~~~~~~~~ & \Cref{alg:pinn-hc} with Option I & \Cref{alg:pinn-hc} with Option II  \\
		\midrule
 		$16$ & $2.0049 \times 10^{-2}$ & $7.8972 \times 10^{-5}$ & $9.4242 \times 10^{-4}$ \\
 		$32$ & $5.8477 \times 10^{-3}$ & $8.0359 \times 10^{-5}$ & $7.8161 \times 10^{-4}$ \\
 		$64$ & $1.4215 \times 10^{-3}$ & $8.1759 \times 10^{-5}$ & $7.1714 \times 10^{-4}$ \\
 		$128$ & $3.6148 \times 10^{-4}$ & $8.2311 \times 10^{-5}$ & $6.8934 \times 10^{-4}$ \\
 		$256$ & $9.6419 \times 10^{-5}$ & $8.2656 \times 10^{-5}$ & $6.7645 \times 10^{-4}$ \\
		\midrule
		\multirow{2}[3]{*}{$N$} & \multicolumn{3}{c}{$\|\hat{p} - p^*\|_{L^2(\Omega)}$}  \\ [0.5ex] 
 		\cmidrule(lr){2-4}
		~ & IFEM \cite{zhang2015immersed} & \Cref{alg:pinn-hc} with Option I & \Cref{alg:pinn-hc} with Option II  \\
		\midrule
		$16$ & $2.5823 \times 10^{-2}$ & $7.8993 \times 10^{-5}$ & $9.4317 \times 10^{-4}$ \\ 
		$32$ & $6.6744 \times 10^{-3}$ & $8.1344 \times 10^{-5}$ & $7.8366 \times 10^{-4}$ \\
		$64$ & $1.6418 \times 10^{-3}$ & $8.2633 \times 10^{-5}$ & $7.1928 \times 10^{-4}$ \\
		$128$ & $4.0256 \times 10^{-4}$ & $8.3286 \times 10^{-5}$ & $6.9175 \times 10^{-4}$ \\
		$256$ & $1.0293 \times 10^{-4}$ & $8.3614 \times 10^{-5}$ & $6.7891 \times 10^{-4}$ \\
 		\midrule
 		\multirow{2}[3]{*}{$N$} & \multicolumn{3}{c}{$\|\hat{y} - y^*\|_{L^2(\Omega)}$}  \\ [0.5ex] 
 		\cmidrule(lr){2-4}
		~ & IFEM \cite{zhang2015immersed} & \Cref{alg:pinn-hc} with Option I & \Cref{alg:pinn-hc} with Option II  \\
		\midrule
		$16$ & $5.6594 \times 10^{-3}$ & $1.3021 \times 10^{-4}$ & $7.9867 \times 10^{-4}$ \\
		$32$ & $1.4803 \times 10^{-3}$ & $1.3877 \times 10^{-4}$ & $6.8886 \times 10^{-4}$ \\
		$64$ & $3.6993 \times 10^{-4}$ & $1.4114 \times 10^{-4}$ & $6.4501 \times 10^{-4}$ \\
		$128$ & $9.4048 \times 10^{-5}$ & $1.4226 \times 10^{-4}$ & $6.2584 \times 10^{-4}$ \\
		$256$ & $2.2873 \times 10^{-5}$ & $1.4282 \times 10^{-4}$ & $6.1696 \times 10^{-4}$ \\
		\bottomrule
 	\end{tabular}
	\normalsize
	\label{tab:ex1-err-comp}
\end{table}

\medskip

\noindent\textbf{Example 2.}\exmplabel{ex:elliptic-comp} To further validate the effectiveness of  \Cref{alg:pinn-hc} with Option I, we consider problem \eqref{eq:ocip-elliptic}-\eqref{eq:control-constraint-elliptic} with a complicated interface.
	In particular, we take $\Omega = (-1, 1) \times (-1, 1) \subset \bR^2$, and the interface $\Gamma$ is the curve defined by the polar coordinate equation
	$  r = 0.5 + 0.2\sin (5\theta).$
	The shape of $\Gamma, \Omega^-$ and $\Omega^+$ is illustrated in \Cref{fig:ex3-interface-shape} (a). We then set $\alpha=1$, $\beta^- = 1$, $\beta^+ = 10$, $g_0 = g_1 = 0$, $h_0 = 0$, $y_d(x) = (x_1^2 - 1)(x_2^2 - 1),$ and 
			$f(x) =
				2 \beta^\pm (2 - x_1^2 - x_2^2), \text{~if~} x \in \Omega^\pm.$
	Compared with Example 1, this example is more general and its exact solution is unknown.

		\begin{figure}[htbp]
		\centering
		\subfloat[The shape $\Gamma$.]{\includegraphics[width=0.25\textwidth]{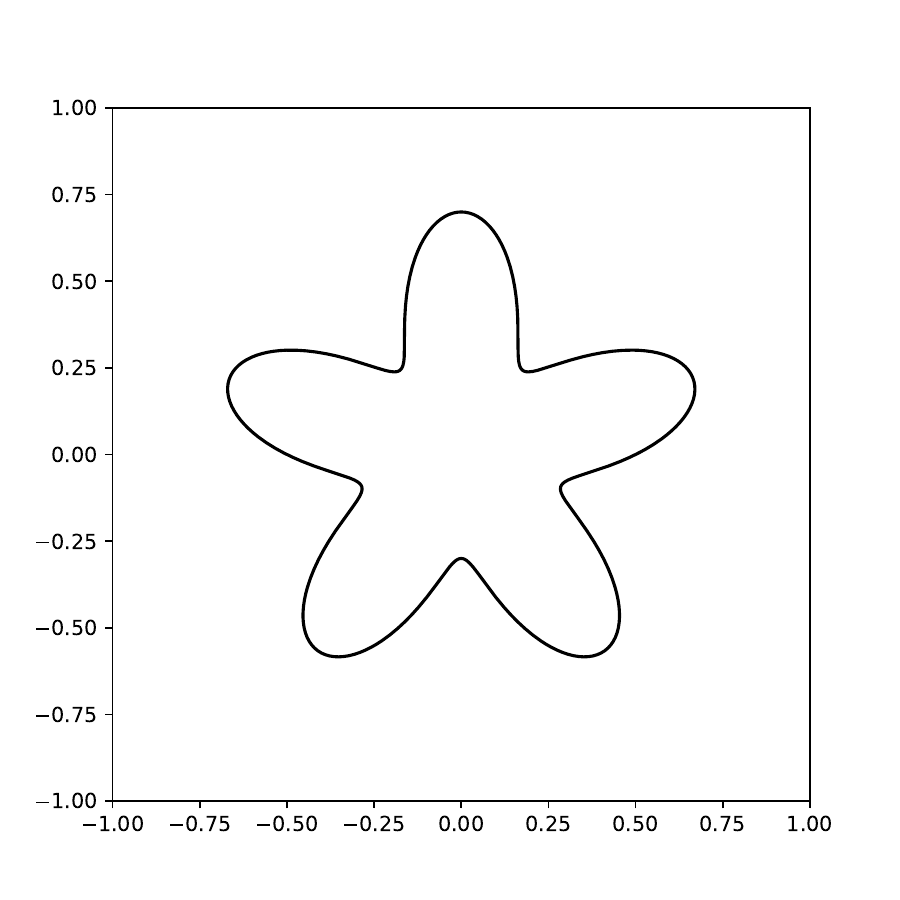}}
			\subfloat[Computed control $u$.]{\includegraphics[width=0.32\textwidth]{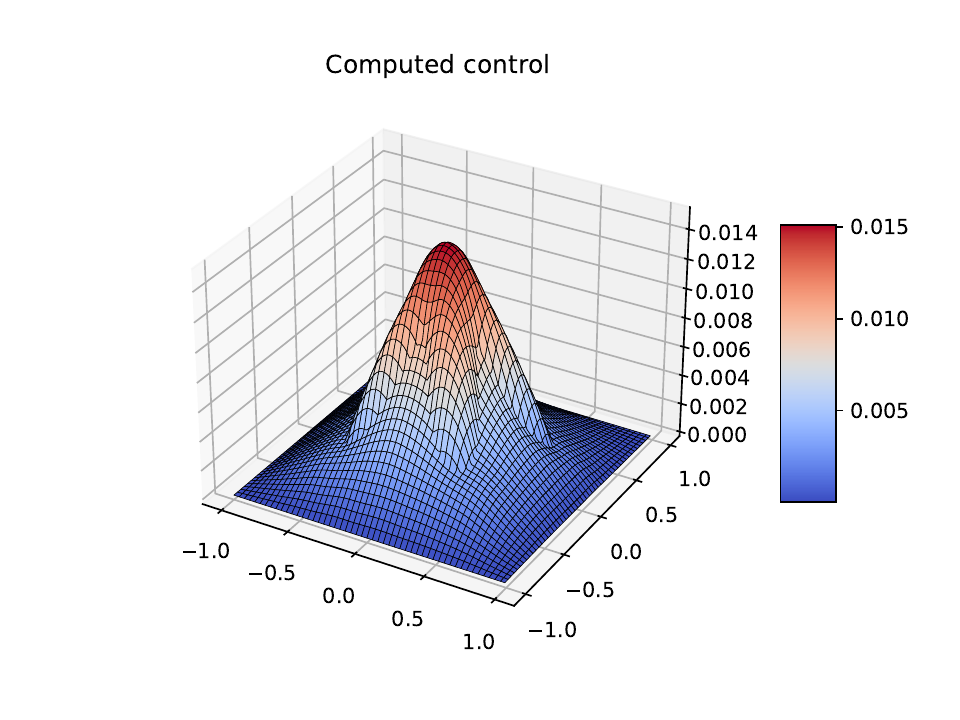}}
				\subfloat[Computed control $u$.]{\includegraphics[width=0.32\textwidth]{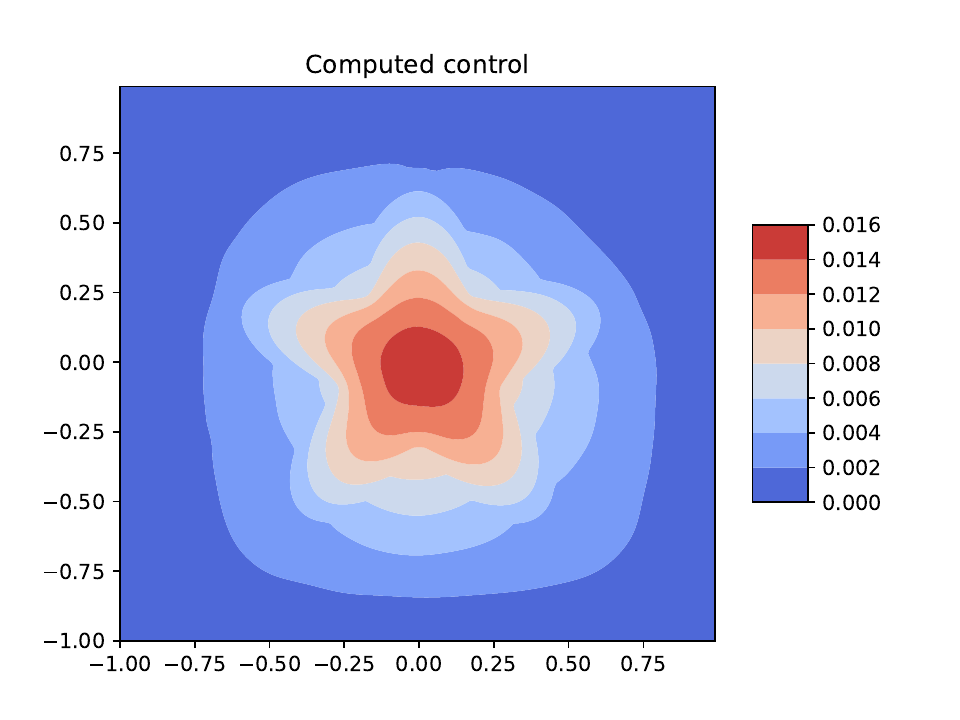}}\\
			\subfloat[The graph of $-\phi$.]{\includegraphics[width=0.28\textwidth]{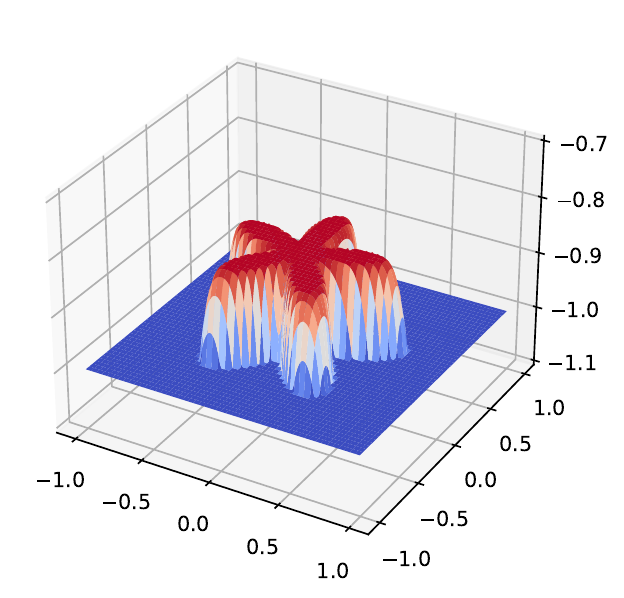}}
		\subfloat[Computed state $y$.]{\includegraphics[width=0.32\textwidth]{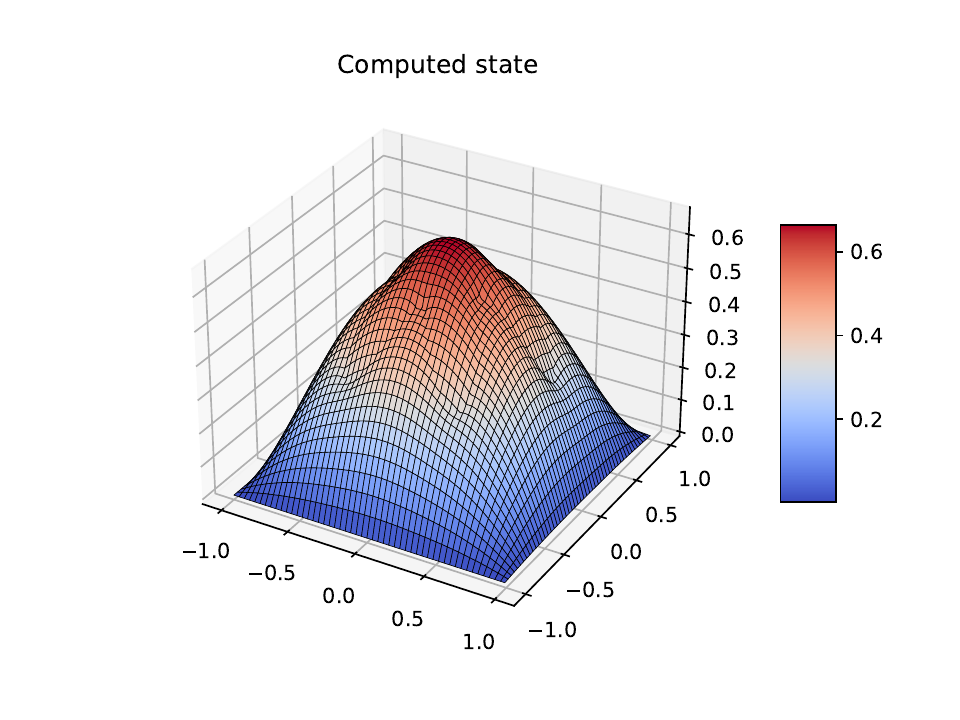}}
		\subfloat[Computed state $y$.]	{\includegraphics[width=0.32\textwidth]{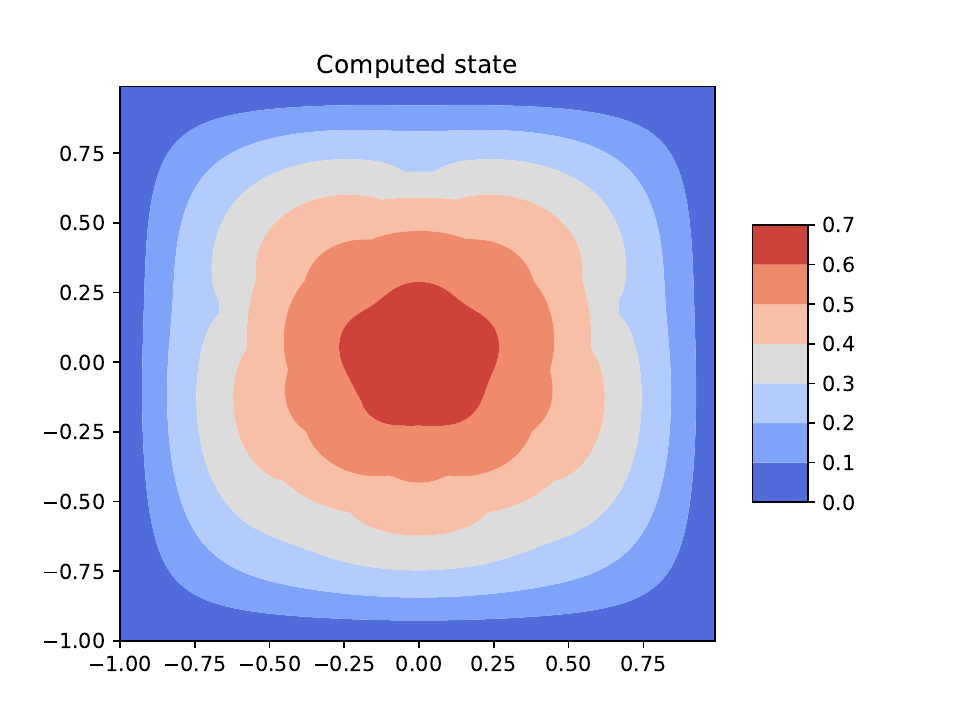}}
			\caption{Numerical settings and results for \Cref{ex:elliptic-comp}. (a) and (d): Star shaped $\Gamma$ and the auxiliary function $\phi$. (b) and (c): The computed control $u$. (e) and (f): The computed state $y$. }
		\label{fig:ex3-interface-shape}
	\end{figure}

	To implement \Cref{alg:pinn-hc} with Option I, we first define two fully-connected neural networks $\cN_y(x, \phi(x); \theta_y)$ and $\cN_p(x, \phi(x); \theta_p)$, which consist of three hidden layers with 100 neurons per hidden layer.
	Then, $y$ and $p$ are respectively approximated by the $\hat{y}(x;\theta_y)$ and  $\hat{p}(x;\theta_p)$ given in \eqref{eq:neural-form-hard-bcij} and  \eqref{eq:neural-form-p}, but with $g = 0$ and $h(x) = (x_1^2 - 1)(x_2^2 - 1)$.
	The auxiliary function $\phi$ can be constructed by \Cref{thm:aux-func-construction} (see \Cref{ex:aux-func-ex-3}).
	Here, we first define
	\begin{equation*}
		\phi_0(r, \theta) = \begin{cases}
			0.2^3, & \text{~if~} r - 0.2\sin(5\theta) \leq 0.3, \\
			{0.2 ^ 3 - (0.2 - f(r, \theta))^3}, & \text{~if~} 0.3 \leq r - 0.2\sin(5\theta) \leq 0.5, \\
			0, & \text{~if~} r - 0.2\sin(5\theta) \geq 0.5,
		\end{cases}
	\end{equation*}
	where $f(r, \theta) = 0.5 + 0.2\sin(5\theta) - r$.
	Then it is easy to check that $\phi := 1 - 20 \cdot \phi_0$ satisfies \eqref{eq:aux-func-prop}. The graph of $-\phi$ is shown in \Cref{fig:ex3-interface-shape} (d).
	
	We randomly sample the training sets $\mathcal{T}=\{x^i\}_{i=1}^{16384} \subset \Omega$ and $\mathcal{T}_{\Gamma}=\{x_\Gamma^i\}_{i=1}^{1024} \subset \Gamma$  with respect to the polar angle.
	The neural networks are initialized randomly  following the default settings of PyTorch.
	We set $w_{p, \Gamma_n} = 3$ and $w_{y, r} = w_{y, \Gamma_n} = w_{p, r} = 1$.
	The neural networks are trained with $40,000$ ADAM iterations, where $\theta_y$ and $\theta_p$ are updated simultaneously in each iteration.
	The initial learning rate is $10^{-2}$ in the first $5,000$ iterations, then $3\times 10^{-3}$ in $5,001$ to $10,000$ iterations, then $10^{-3}$ in $10,001$ to $20,000$ iterations, finally $3\times 10^{-3}$ in $20,001$ to $40,000$ iterations.
	
	The computed $u$ and $y$ are presented in \Cref{fig:ex3-interface-shape}.
	We can see that the computed control and state by Option I of \Cref{alg:pinn-hc} capture the nonsmoothness across the interface even if the geometry of the interface is complicated.

\section{Extensions}\label{sec:furtuer-extension}

In this section, we show that  \Cref{alg:pinn-hc} can be easily extended to other types of  interface optimal control problems. For this purpose, we investigate an elliptic interface optimal control problem, where the control variable acts on the interface, and a parabolic interface optimal control problem.

\subsection{Control on the interface}

We consider the optimal control problem:
\begin{equation}\label{eq:ocip-elliptic-ic}
	\begin{aligned}
		&\min_{y\in L^2(\Omega), u\in L^2(\Gamma)}  J(y, u) := \dfrac{1}{2} \int_\Omega (y - y_d)^2 dx + \dfrac{\alpha}{2} \int_\Gamma u^2 dx, \\
		\mathrm{s.t.} ~ &
		\begin{aligned}
			-\nabla\cdot(\beta\nabla y)= f ~\text{in}~\Omega \backslash \Gamma, ~
			 [y]_{\Gamma}=g_0,~ [\beta\partial_{\bm{n}}y]_{\Gamma}= u + g_1~ \text{on}~\Gamma,~
			 y= h_0~ \text{on}~\partial\Omega,
		\end{aligned}
	\end{aligned}
\end{equation}
together with the control constraint
\begin{equation}\label{eq:control-constraint-elliptic-ic}
	u \in U_{ad} := \{u \in L^2(\Gamma): u_a \leq u \leq u_b \text{~a.e.~on}~\Gamma\},
\end{equation}
where $u_a, u_b \in L^2(\Gamma)$. Above, all the notations are the same as those in  \eqref{eq:ocip-elliptic}-(\ref{eq:ip-elliptic}) but the control variable $u\in L^2(\Gamma)$ in   \eqref{eq:ocip-elliptic-ic}-\eqref{eq:control-constraint-elliptic-ic} acts on the interface rather than the source term.  Existence and uniqueness of the solution to  problem  \eqref{eq:ocip-elliptic-ic} can be found in \cite{yang2018interface}, and we have the following results.

\begin{theorem}[cf. \cite{yang2018interface}]\label{thm:optcond-elliptic-ic}
	Problem \eqref{eq:ocip-elliptic-ic}--\eqref{eq:control-constraint-elliptic-ic} admits a unique optimal control $(u^*,y^*)^\top \in U_{ad}\times L^2(\Omega)$, and the following first-order optimality system holds
	\begin{equation}\label{eq:oc-coij}
		u^* = \cP_{U_{ad}}\Big(  \left. -\dfrac{p^* }{\alpha} \right|_\Gamma (x) \Big),
	\end{equation}
	where $\mathcal{P}_{U_{ad}}(\cdot)$ denotes the projection  onto $U_{ad}$, and  $p^*$ is the adjoint variable associated with $u^*$, which is obtained from the successive solution of the following two equations
	\begin{equation}\label{eq:state-coij}
			 -\nabla\cdot(\beta\nabla y^*)= f ~\text{in}~\Omega \backslash \Gamma, ~
			 [y^*]_{\Gamma}=g_0,~ [\beta\partial_{\bm{n}}y^*]_{\Gamma}= g_1 + u^* ~\text{on}~\Gamma,~
			 y^*= h_0 ~\text{on}~\partial\Omega,
	\end{equation}
	\begin{equation}\label{eq:adjoint-coij}
			-\nabla\cdot(\beta\nabla p^*)=y^*-y_d ~\text{in}~\Omega \backslash \Gamma, ~
			 [p^*] _{\Gamma}=0,~[\beta \partial_{\bn}p^*]_{\Gamma}=0 ~\text{on}~\Gamma, ~
			 p^*=0 ~\text{on}~\partial\Omega.
	\end{equation}
\end{theorem}

It is easy to see that problem \eqref{eq:ocip-elliptic-ic} is convex and hence the optimality system  \eqref{eq:oc-coij} is also sufficient.
Substituting \eqref{eq:oc-coij} into \eqref{eq:state-coij} yields
\small
\begin{equation}\label{eq:state-reduced-coij}
		 -\nabla\cdot(\beta\nabla y^*)= f ~ \text{in}~\Omega \backslash \Gamma, ~
		 [y^*]_{\Gamma}=g_0,~ [\beta\partial_{\bm{n}}y^*]_{\Gamma}= g_1 + \cP_{U_{ad}}\Big(  \left. -\dfrac{p^* }{\alpha} \right|_\Gamma (x) \Big) ~\text{on}~\Gamma,~
		 y^*= h_0 ~ \text{on}~\partial\Omega.
\end{equation}
\normalsize
Therefore, solving \eqref{eq:ocip-elliptic-ic}--\eqref{eq:control-constraint-elliptic-ic} is equivalent to solving \eqref{eq:adjoint-coij} and \eqref{eq:state-reduced-coij} simultaneously.

Next, we demonstrate the extension of \Cref{alg:pinn-hc} to solve problem  \eqref{eq:ocip-elliptic-ic}--\eqref{eq:control-constraint-elliptic-ic}.
First, the neural networks $\hat{y}$ and  $\hat{p}$ for approximating $y$ and $p$ are constructed in the same way as that in \Cref{sec:elliptic-pinns-hc}, see  \eqref{eq:neural-form-hard-bcij} and  \eqref{eq:neural-form-p}. The loss function is now defined as
\small
\begin{gather}\label{eq:loss-elliptic-coij-hc}
	\begin{aligned}
		\cL_{HC}(\theta_y, \theta_p) = & \frac{w_{y, r}}{M} \sum_{i=1}^{M} \left| -\Delta_x \hat{y}(x^i; \theta_{y})-\frac{f(x^i)}{\beta^\pm} \right| ^2 \\
		& + \frac{w_{y, \Gamma_n}}{M_{\Gamma}} \sum_{i=1}^{M_\Gamma} \left| [\beta\partial_{\bm{n}}\hat{y}]_\Gamma (x_{\Gamma}^i; \theta_{y}) - g_1(x_\Gamma^i) - \mathcal{P}_{[u_a(x_\Gamma^i), u_b(x_\Gamma^i)]}(-\frac{1}{\alpha} \hat{p}(x^i; \theta_p)) \right|^2 \\
		& + \frac{w_{p, r}}{M} \sum_{i=1}^{M} \left|-\Delta_x \hat{p}(x_i; \theta_{p})-\frac{\hat{y}(x_i; \theta_{y}) - y_d(x_i)}{\beta^\pm} \right|^2  + \frac{w_{p, \Gamma_n}}{M_{\Gamma}}\sum_{i=1}^{M_\Gamma} \left|[\beta\partial_{\bm{n}}\hat{p}]_\Gamma(x_i^{\Gamma}; \theta_{p}) \right|^2.
	\end{aligned}
\end{gather}
\normalsize
Then, we can easily obtain the hard-constraint PINN method for  \eqref{eq:ocip-elliptic-ic}-\eqref{eq:control-constraint-elliptic-ic} and we omit the details for succinctness.

\medskip

\noindent\textbf{Example 3.}\exmplabel{ex:elliptic-reg-ic}
	We consider an example of  \eqref{eq:ocip-elliptic-ic}-\eqref{eq:control-constraint-elliptic-ic} with $\alpha = 1$, $\beta^- = 1$, $\beta^+ = 10$,
	$u\in U_{ad} = \{u \in L^2(\Gamma): \sin(2 \pi x_1) \leq u(x) \leq 1 \text{~a.e. on~} \Gamma  \}$ and
	$
	u^*(x_1, x_2) = \max\{\sin(2\pi x_1), \min\{1, -\frac{1}{\alpha} p^*|_\Gamma(x_1, x_2)\}\}.
	$
	The rest of settings are the same as those in \Cref{ex:elliptic-reg-cc}.
	Then, we can see that $(u^*, y^*)^\top$, with $y^*$ defined in (\ref{eq:exact_elliptic_y}), is the solution to this problem.
	
	The neural networks $\cN_y(x, \phi(x); \theta_y), \cN_p(x, \phi(x); \theta_p)$ and the functions $g$, $h$, and $\phi$ are the same as those in Example 1. Moreover, we randomly sample the training sets $\mathcal{T}=\{x^i\}_{i=1}^{1024} \subset \Omega$ and $\mathcal{T}_{\Gamma}=\{x_\Gamma^i\}_{i=1}^{256} \subset \Gamma$. We initialize the neural network parameters $\theta_y, \theta_p$ following the default settings of PyTorch. All the weights in \eqref{eq:loss-elliptic-coij-hc} are set to be $1$.
	
	We employ 40,000 iterations of ADAM to train the neural networks and the parameters $\theta_y$ and $\theta_p$ are updated simultaneously in each iteration.
	The learning rate is set to be $5\times 10^{-3}$ in the first $5,000$ iterations, then $1\times 10^{-3}$ in the $5,001$ to $15,000$ iterations, then $5\times 10^{-4}$ in the $15,001$ to $30,000$ iterations, finally $1\times 10^{-4}$ in the $30,001$ to $40,000$ iterations.
	The strategy for computing the $L^2$-error of the control is similar to \eqref{eq:abs-rel-error-def},  instead that now we sample the testing points from the interface $\Gamma$.
	We use the same method as that in Example 1 for visualizing $y$ and its error.
	For the computed control $u$, we present its graph along {the interface circle  $\{(\cos (2 \pi \theta), \sin (2 \pi \theta)) \in \bR^2 : 0 \leq \theta \leq 1 \}$ with respect to the angle parameter $\theta$.}
	
	The numerical results  are shown in \Cref{fig:ex4-pinnhc-training-results}.
	It can be observed that the maximum absolute errors of $y$ and $u$ obtained by  the hard-constraint PINN method are approximately $6\times 10^{-4}$ and $7\times 10^{-5}$, respectively.
	Moreover,  the $L^2$-errors of the computed control $u$ are $\varepsilon_\text{abs} = 2.0386 \times 10^{-5}$ and $\varepsilon_\text{rel} = 7.8737 \times 10^{-5}$.
	These results show that the proposed hard-constraint PINN method is also efficient for  \eqref{eq:ocip-elliptic-ic}-\eqref{eq:control-constraint-elliptic-ic}, producing numerical solutions with high accuracy.
	\begin{figure}[htpb]
		\centering
		\subfloat[Exact control $u$.]{\includegraphics[width=0.32\textwidth]{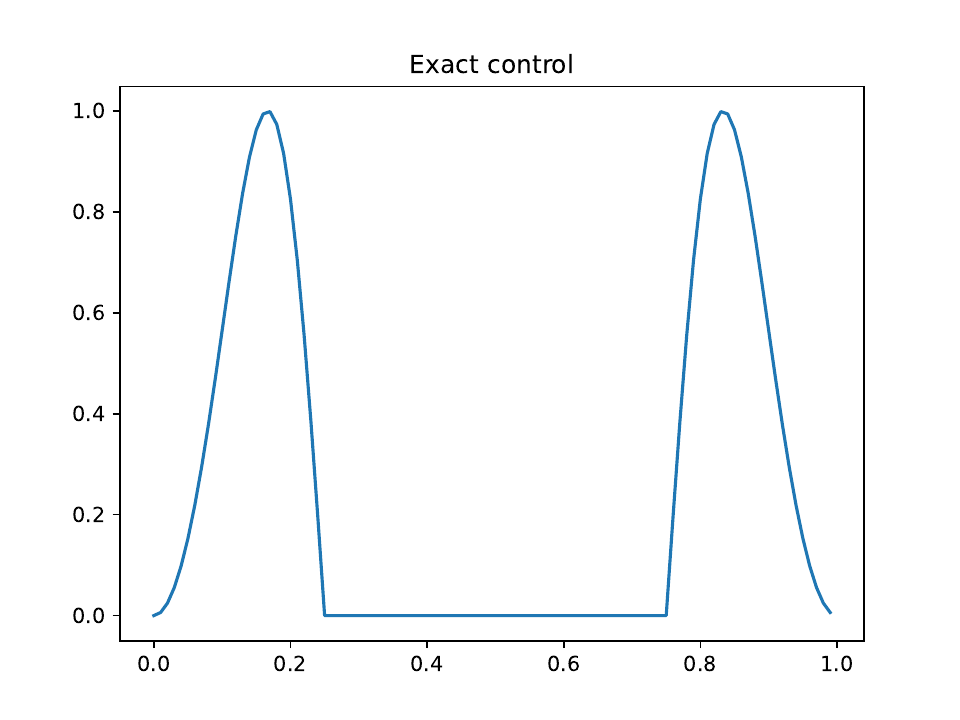}}
		\subfloat[Computed control $u$.]{\includegraphics[width=0.32\textwidth]{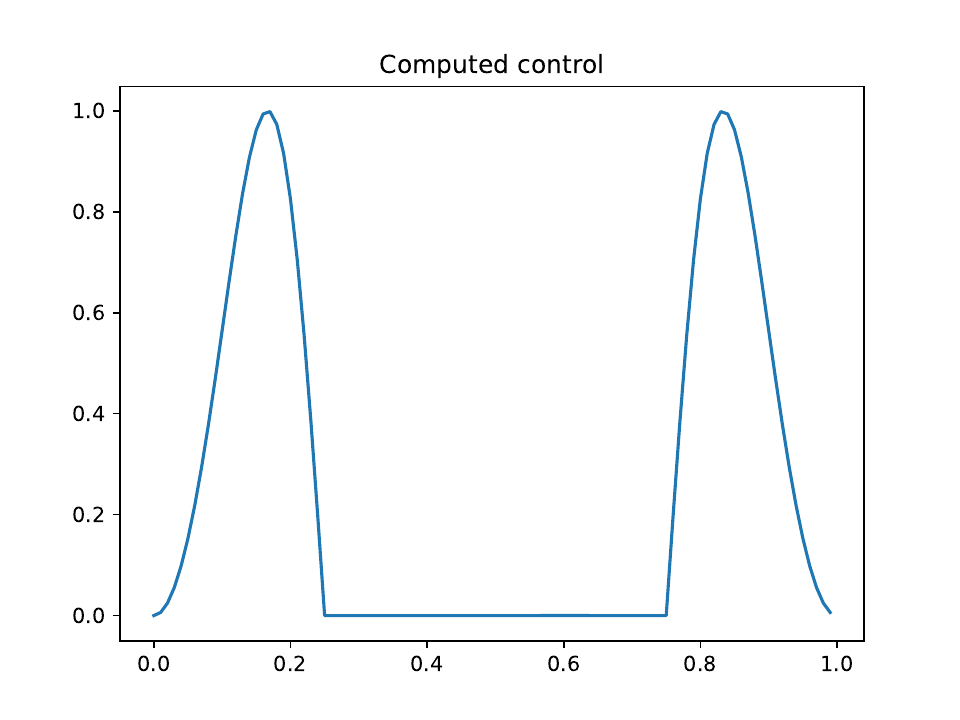}}
		\subfloat[Error of control $u$.]{\includegraphics[width=0.32\textwidth]{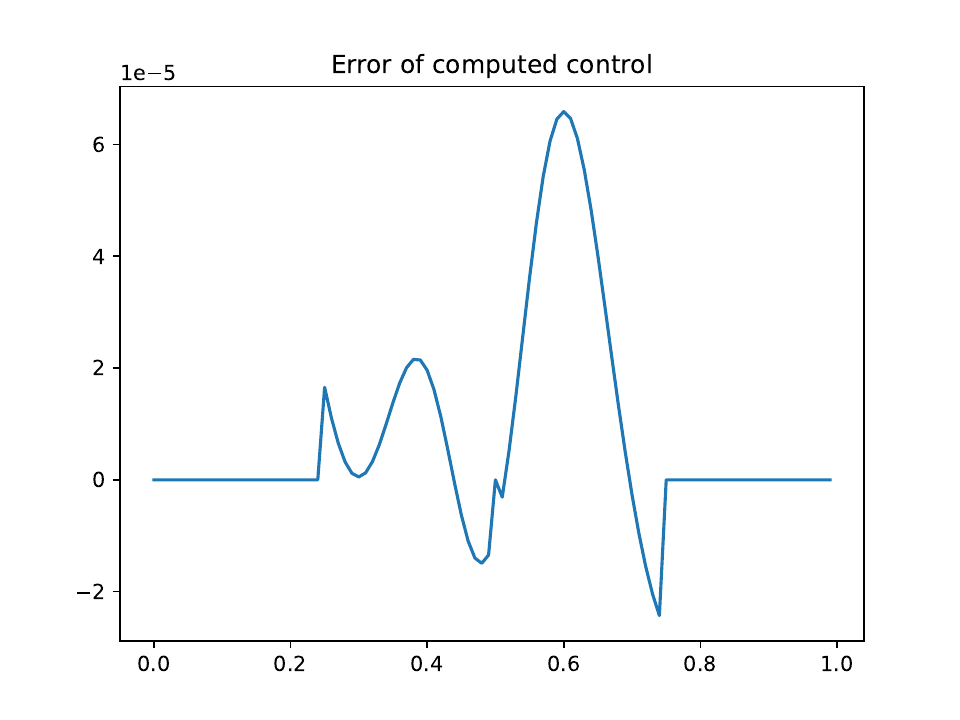}}\\
		\subfloat[Exact state $y$.]	{\includegraphics[width=0.32\textwidth]{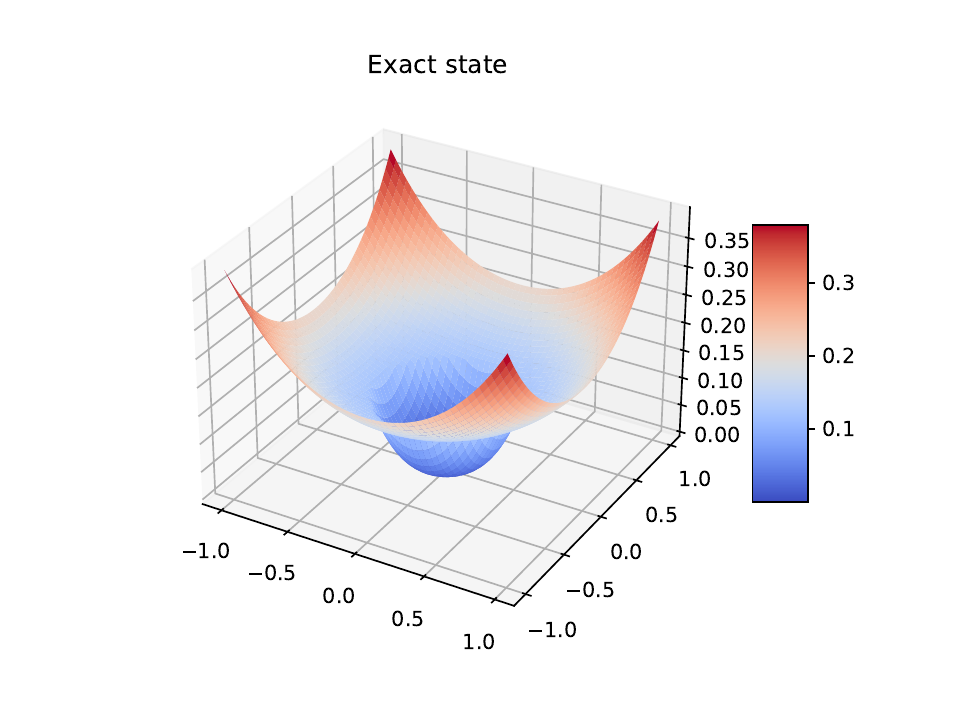}}
		\subfloat[Computed state $y$.]{\includegraphics[width=0.32\textwidth]{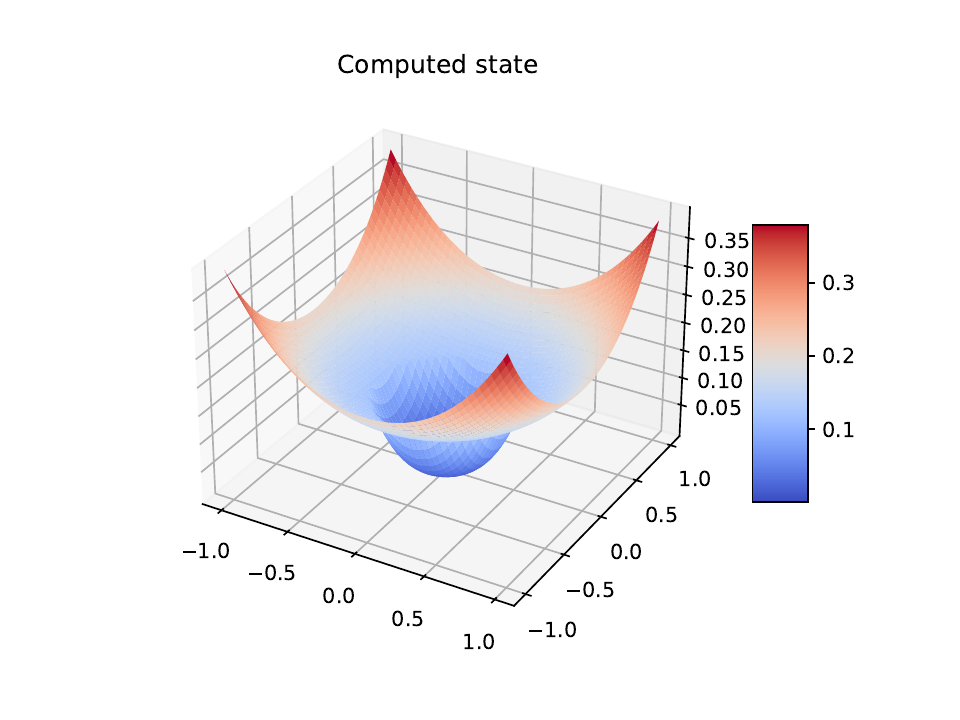}}
		\subfloat[Error of state $y$.]{\includegraphics[width=0.32\textwidth]{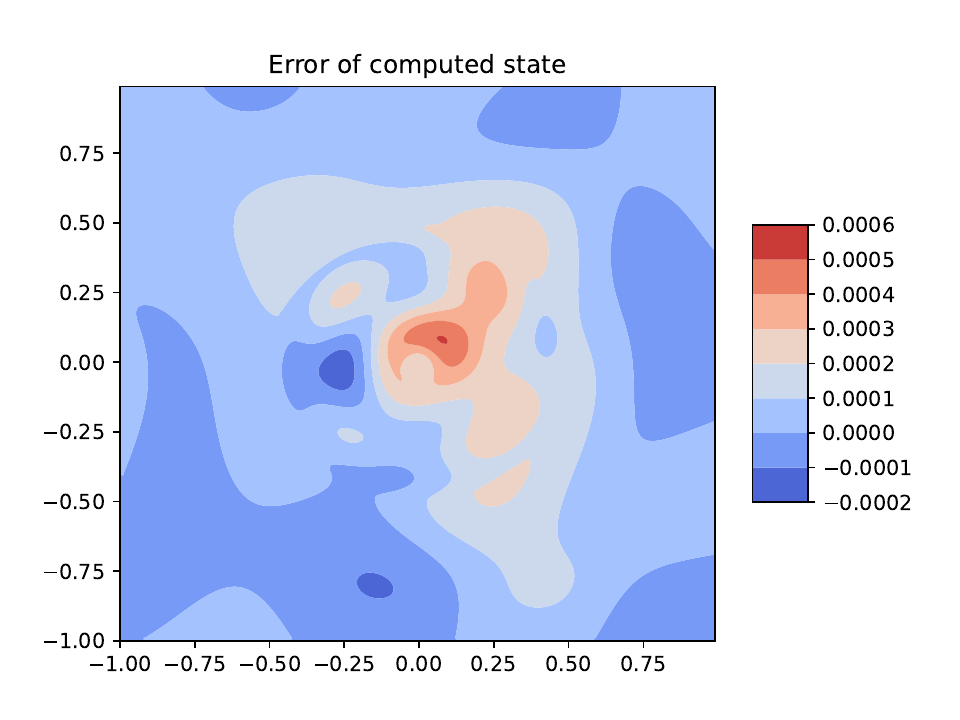}}
		\caption{Numerical results of the hard-constraint PINN method for \Cref{ex:elliptic-reg-ic}.}
         \label{fig:ex4-pinnhc-training-results}
	\end{figure}

\subsection{A parabolic interface optimal control problem}\label{subsec:parab-ocip}

In this subsection, we discuss the extension of  \Cref{alg:pinn-hc} to time-dependent problems.
To this end,  let $\Omega \subset \mathbb{R}^d~(d = 2, 3)$ be a bounded domain and $\Gamma \subset \Omega$ be the interface as the one defined in Section \ref{sec:preliminaries}. Consider the following optimal control problem:
\begin{equation}\label{eq:ocip-parabolic}
	\begin{aligned}
		&\min_{y\in L^2 (\Omega\times(0,T)), u\in L^2 (\Omega\times(0,T))}  J(y, u) := \dfrac{1}{2} \int_0^T \int_\Omega (y - y_d)^2 dx + \dfrac{\alpha}{2} \int_0^T \int_\Omega u^2 dx, \\
		& \mathrm{s.t.} \left\{
		\begin{aligned}
			& \frac{\partial y}{\partial t} -\nabla\cdot(\beta\nabla y)= u + f \quad &&\text{in}~(\Omega \backslash \Gamma) \times (0, T),\quad &&y= h_0 ~ &&\text{on}~\partial\Omega \times (0, T),\\
			& [y]_{\Gamma}=g_0,~ [\beta\partial_{\bm{n}}y]_{\Gamma}= g_1 \quad&&\text{on}~\Gamma \times (0, T),\quad&& y(0) = y_0~&&\text{in}~\Omega.\\
		\end{aligned}
		\right.
	\end{aligned}
\end{equation}
Above, the final time $T>0$ is a fixed constant, the function $y_d \in L^2(\Omega \times (0, T))$ is the target, and the constant $\alpha>0$ is a regularization parameter.
The functions $f, g_0, g_1, h_0$, and $y_0$ are given, and $\beta$ is a postive piecewise-constant function as the one defined  in Section \ref{sec:preliminaries}.
The interface $\Gamma$ is assumed to be time-invariant.
We also set the following constraint for the control variable:
\begin{equation}\label{eq:control-constraint-parab}
	u \in U_{ad} := \{ u \in L^2(\Omega \times (0, T)) : u_a \leq u \leq u_b \text{~a.e.~in~}\Omega\times(0,T) \},
\end{equation}
where $u_a, u_b \in L^2(\Omega \times (0, T))$.  By \cite{zhang2020immersed}, we have the following results.
\begin{theorem}[cf. \cite{zhang2020immersed}]
	\label{thm:opt-condition-parab}
	Problem  \eqref{eq:ocip-parabolic}-\eqref{eq:control-constraint-parab} admits a unique optimal control $(u^*,y^*)^\top \in U_{ad}\times L^2(\Omega\times(0,T))$, and the following first-order optimality system holds
	\begin{equation}\label{eq:opt-cond-parab}
		u^*=\mathcal{P}_{U_{ad}}\Big(-\frac{1}{\alpha}p^*\Big),
	\end{equation}
	where $\mathcal{P}_{U_{ad}}(\cdot)$ denotes the projection  onto $U_{ad}$, and  $p^*$ is the adjoint variable associated with $u^*$ which is obtained from the successive solution of the following two equations
	\small
	\begin{equation}\label{eq:state-parab}
		\hspace{3em}\left\{
		\begin{aligned}
			& \frac{\partial y^*}{\partial t} -\nabla\cdot(\beta\nabla y^*)= u^* + f ~ &&\text{in}~(\Omega \backslash \Gamma) \times (0, T),\quad&& y^*= h_0~&& \text{on}~\partial\Omega \times (0, T), \\
			& [y^*]_{\Gamma}=g_0, \quad [\beta\partial_{\bm{n}}y^*]_{\Gamma}= g_1 ~&&\text{on}~\Gamma \times (0, T),\quad&& y^*(0) = y_0~&& \text{in}~\Omega,\\
		\end{aligned}
		\right.
	\end{equation}
	\begin{equation}\label{eq:adjoint-parab}
		\left\{
		\begin{aligned}
			& -\frac{\partial p^*}{\partial t} - \nabla\cdot(\beta\nabla p^*)= y^*-y_d 
			~&& \text{in}~ (\Omega \backslash \Gamma) \times (0, T), \quad &&p^*= 0 ~&& \text{on}~\partial\Omega \times (0, T), \\
			& [p^*] _{\Gamma}=0, \quad [\beta \partial_{\bn}p^*]_{\Gamma}= 0,~ &&\text{on~} \Gamma \times (0, T),\quad && p^*(T) = 0 ~&& \text{in}~\Omega.\\
		\end{aligned}
		\right.
	\end{equation}
	\normalsize
\end{theorem}

Problem \eqref{eq:ocip-parabolic}-\eqref{eq:control-constraint-parab} is convex, and hence the solution to \eqref{eq:ocip-parabolic}--\eqref{eq:control-constraint-parab} can be obtained by simultaneously solving \eqref{eq:state-parab} and \eqref{eq:adjoint-parab}.
Next, we delineate the extension of  \Cref{alg:pinn-hc} to problem  \eqref{eq:ocip-parabolic}-\eqref{eq:control-constraint-parab}.  For this purpose,
let $\cN_y(x, t, \phi(x); \theta_y)$ and $\cN_p(x, t, \phi(x); \theta_p)$ be two neural networks with smooth activation functions, we then approximate the solutions of  \eqref{eq:state-parab} and \eqref{eq:adjoint-parab} by
\begin{equation}\label{eq:neural-form-hard-bcijic}
	\begin{aligned}
		& \hat{y}(x, t; \theta_y) = g(x, t) + t h(x) \cN_y(x, t, \phi(x); \theta_y), \\
		& \hat{p}(x, t; \theta_p) = (T - t) h(x) \cN_p(x, t, \phi(x); \theta_p).
	\end{aligned}
\end{equation}
Here, $\phi(x)$ is an auxiliary function satisfying \eqref{eq:aux-func-prop}, $h(x)$ is a function satisfying \eqref{eq:h-func-prop}, and both of them are independent of the variable $t$ since the interface $\Gamma$ and the boundary $\partial \Omega$ are time-invariant.
The function $g$ satisfies
\begin{equation}\label{eq:g-func-prop-with-t}
	\begin{aligned}
		& g= h_0, \text{~on~} \partial \Omega \times (0, T), \quad g(0) = y_0, \text{~in~} \Omega, \quad [g]_\Gamma = g_0 \text{~on~} \Gamma \times (0, T).
	\end{aligned}
\end{equation}
Then, using the same arguments as those in  \Cref{subsec:hc-bc}, it is easy to check that $\hat{y}(x, t; \theta_y)$ and $\hat{p}(x, t; \theta_p)$ strictly satisfy the interface, boundary, and initial conditions in  \eqref{eq:state-parab}-\eqref{eq:adjoint-parab} if functions $g$, $h$, and $\phi$ are given in analytic expressions. 
Moreover, we reiterate that,  following the discussions in  \Cref{subsec:hc-bc}, the functions $g, h$, and $\phi$ can be constructed in analytic forms or by neural networks.

To train the neural networks $\hat{y}(x, t; \theta_y)$ and $\hat{p}(x, t; \theta_p)$, we sample the training sets $\cT = \{(x^i, t^i)\}_{i=1}^M \subset \Omega \times (0, T)$ and $\cT_\Gamma = \{x_\Gamma^i, t_\Gamma^i\}_{i=1}^{M_\Gamma} \subset \Gamma \times (0, T)$, and consider the following loss function:
	\small
	\begin{equation}\label{eq:loss-parabolic-hc}
		{\begin{aligned}
				&\mathcal{L}_{HC}(\theta_y, \theta_p) \\
				= & \frac{w_{y, r}}{M} \sum_{i=1}^{M} \left| \frac{\partial \hat{y}(x^i, t^i; \theta_y)}{\partial t} -\Delta_x \hat{y}(x^i, t^i; \theta_{y})-\frac{\mathcal{P}_{[u_a(x^i, t^i), u_b(x^i, t^i)]}(-\frac{\hat{p}}{\alpha}(x^i, t^i; \theta_p) + f(x^i, t^i))}{\beta^\pm} \right| ^2 \\
				& + \frac{w_{y, \Gamma_n}}{M_{\Gamma}} \sum_{i=1}^{M_\Gamma} \left| [\beta\partial_{\bm{n}}\hat{y}]_\Gamma (x_{\Gamma}^i, t_\Gamma^i; \theta_{y}) - g_1(x_{\Gamma}^i, t_\Gamma^i) \right|^2 + \frac{w_{p, \Gamma_n}}{M_{\Gamma}}\sum_{i=1}^{M_\Gamma} \left|[\beta\partial_{\bm{n}}\hat{p}]_{\Gamma} (x^i_{\Gamma}, t_\Gamma^i; \theta_{p}) \right|^2\\
				& + \frac{w_{p, r}}{M} \sum_{i=1}^{M} \left| -\frac{\partial \hat{p}(x^i, t^i; \theta_p)}{\partial t} -\Delta_x \hat{p}(x_i, t^i; \theta_{p})-\frac{ \hat{y}(x^i, t^i; \theta_{y}) - y_d(x_i, t^i)}{\beta^\pm} \right|^2.
		\end{aligned}}
\end{equation}
\normalsize
 Then, we can easily obtain the hard-constraint PINN method for solving  \eqref{eq:state-parab}-\eqref{eq:adjoint-parab} and hence problem  \eqref{eq:ocip-parabolic}-\eqref{eq:control-constraint-parab}. 

\medskip

\noindent\textbf{Example 4.}\exmplabel{ex:parabolic}
	We test the hard-constraint PINN method for solving \eqref{eq:ocip-parabolic}-\eqref{eq:control-constraint-parab} with $\Omega = (-1, 1) \times (-1, 1) \subset \bR^2$, $\Gamma = \{x \in \Omega: \lVert x \rVert_2 \leq r_0\}$, $r_0 = 0.5$ and $T = \pi / 2$.
	The admissible set $U_{ad} = \{u \in L^2(\Omega \times (0, T)): -1 \leq u \leq 1 \text{~a.e. in}~\Omega\times (0,T)\}$.
	We further set $\alpha = 1$, $\beta^- = 1$, $\beta^+ = 3$, $g_0 = g_1 = 0$, $h_0 = 0$ and $y_0 = 0$.
	
	Following \cite{zhang2020immersed}, we let
	\begin{equation*}
		\left\{
		\begin{aligned}
			& p^*(x_1, x_2, t) = {5 \sin(T - t) (x_1^2 + x_2^2 - r_0^2)(x_1^2 - 1)(x_2^2-1)}/{\beta^\pm} \quad \text{in} ~ \Omega^\pm, \\
			& y^*(x_1, x_2, t) ={ 5 \cos(t - T) (x_1^2 + x_2^2 - r_0^2)(x_1^2 - 1)(x_2^2-1)}/{\beta^\pm} \quad \text{in} ~ \Omega^\pm, \\
			& u^*(x_1, x_2, t) = \max\{-1, \min\{1, -\dfrac{1}{\alpha} p^*(x_1, x_2, t)\}\},\\
			& f(x_1, x_2, t) = \frac{\partial y^*}{\partial t}(x_1, x_2, t) -u^*(x_1, x_2, t) - \nabla \cdot (\beta^\pm \nabla y^*(x_1, x_2, t)) ~\text{in}~\Omega^\pm, \\
			& y_d(x_1, x_2, t) = \frac{\partial p^*}{\partial t}(x_1, x_2, t) + y^*(x_1, x_2, t) + \nabla \cdot (\beta^\pm \nabla p^*(x_1, x_2, t)) ~\text{in}~\Omega^\pm.
		\end{aligned}
	\right.
	\end{equation*}
	Then it is easy to verify that $(u^*, y^*, p^*)^\top$ satisfies the optimality system  \eqref{eq:opt-cond-parab}-\eqref{eq:adjoint-parab}, and hence $(u^*, y^*)^\top$ is the solution of  \eqref{eq:ocip-parabolic}-\eqref{eq:control-constraint-parab}.
	
	We construct two neural networks $\cN_y(x, t, \phi(x); \theta_y)$ and $\cN_p(x, t, \phi(x); \theta_p)$ consisting of three hidden layers with $100$ neurons.
	The state and adjoint variables are approximated by \eqref{eq:neural-form-hard-bcijic} with $g(x, t) = 0$ and $h(x) = (x_1^2 - 1)(x_2^2 - 1)$.
	The auxiliary function $\phi$ is chosen as that in \eqref{eq:def_phi_ex}
	
	To evaluate the loss function \eqref{eq:loss-parabolic-hc}, we first select $\{t_i\}_{i=1}^{16}$ by the Chebyshev sampling over $(0, T)$. Then we randomly sample $\{x^i\}_{i=1}^{256} \subset \Omega$ and  $\{x_\Gamma^i\}_{i=1}^{64} \subset \Gamma$.  After that, we take the Cartesian product of $\{t_i\}_{i=1}^{16}$ and $\{x^i\}_{i=1}^{256}$ and $\{x_\Gamma^i\}_{i=1}^{64}$ respectively to generate the training sets $\mathcal{T}=\{(x^i, t^i)\}_{i=1}^{4096} \subset \Omega \times (0, T)$ and $\mathcal{T}_{\Gamma}=\{(x_\Gamma^i, t_\Gamma^i)\}_{i=1}^{1024} \subset \Gamma \times (0, T)$. We initialize the neural network parameters $\theta_y$ and $\theta_p$ following the default settings in PyTorch. The weights in \eqref{eq:loss-parabolic-hc} are all taken to be $1$.
	We implement 40,000 iterations of the ADAM to train the neural networks and the parameters $\theta_y$ and $\theta_p$ are optimized simultaneously in each iteration.
	The learning rate is set to be $3\times 10^{-3}$ in the first $3,000$ iterations, then $1\times 10^{-3}$ in the $3,001$ to $10,000$ iterations, and $5\times 10^{-4}$ in the $10,001$ to $20,000$ iterations, finally $1\times 10^{-4}$ in the $20,001$ to $40,000$ iterations.
	
	The computed results at $t=0.3T$ are presented in \Cref{fig:ex5-0.3-pinnhc-training-results}.
	We can see that the hard-constraint PINN method is capable to deal with time-dependent problems and a high-accurate numerical solution can be pursued.

    \begin{figure}[htbp]
		\centering
		\subfloat[Exact control $u$.]{\includegraphics[width=0.32\textwidth]{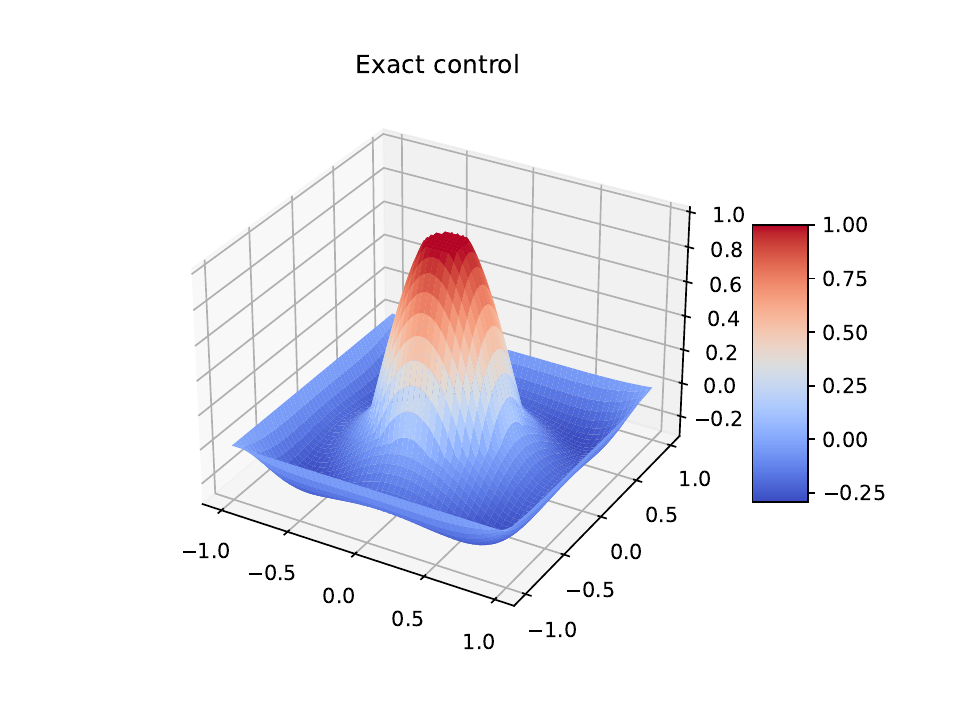}}
		\subfloat[Computed control $u$.]{\includegraphics[width=0.32\textwidth]{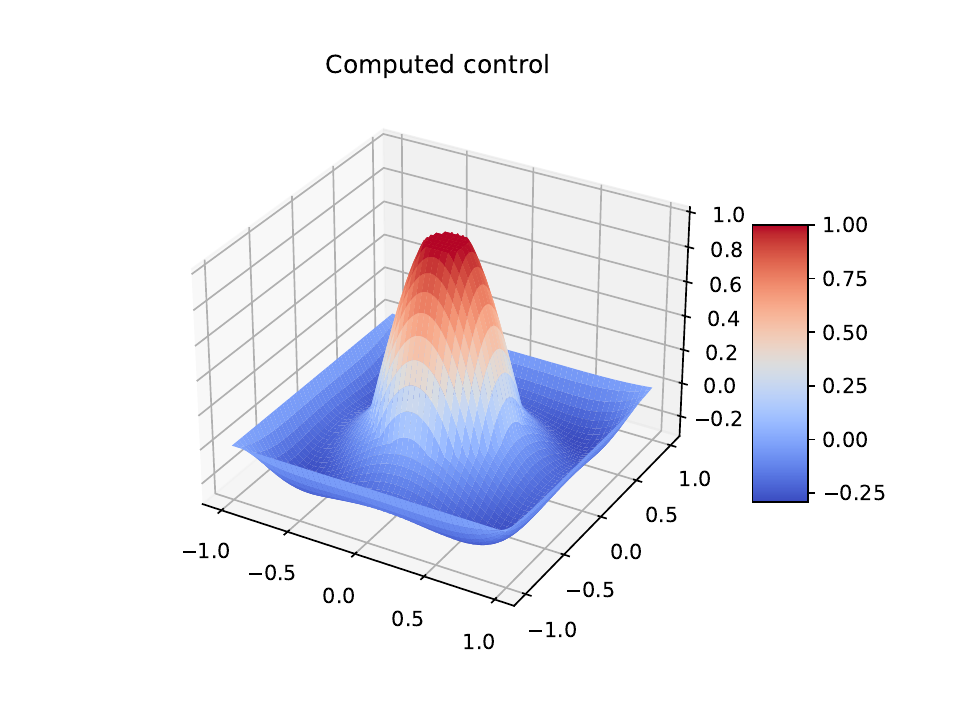}}
		\subfloat[Error of control $u$.]{\includegraphics[width=0.32\textwidth]{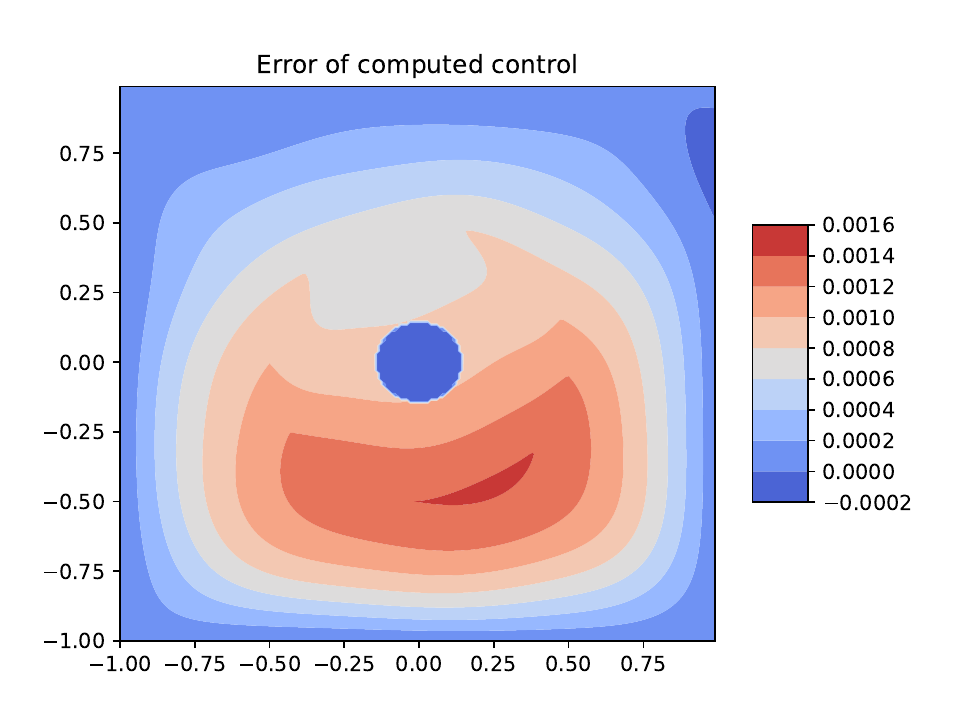}}\\
		\subfloat[Exact state $y$.]	{\includegraphics[width=0.32\textwidth]{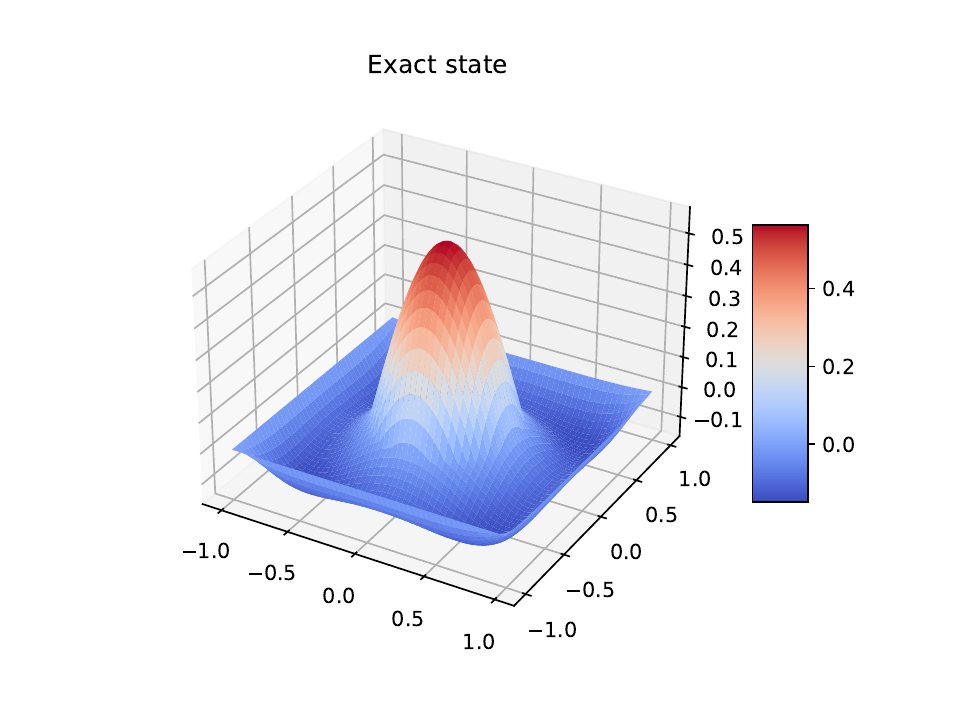}}
		\subfloat[Computed state $y$.]	{\includegraphics[width=0.32\textwidth]{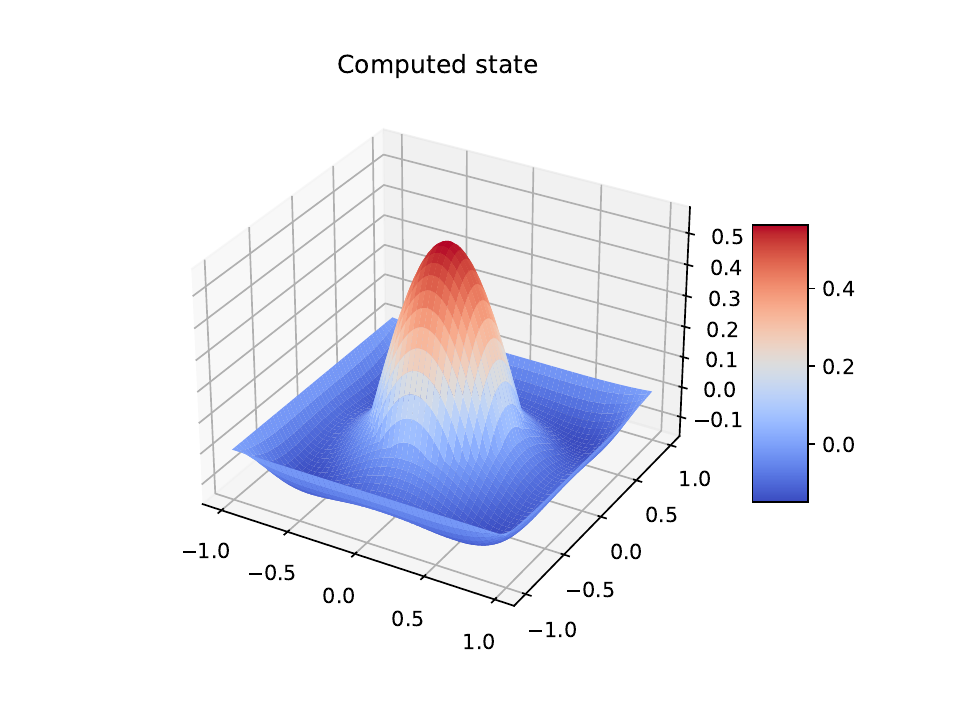}}
		\subfloat[Error of state $y$.]	{\includegraphics[width=0.32\textwidth]{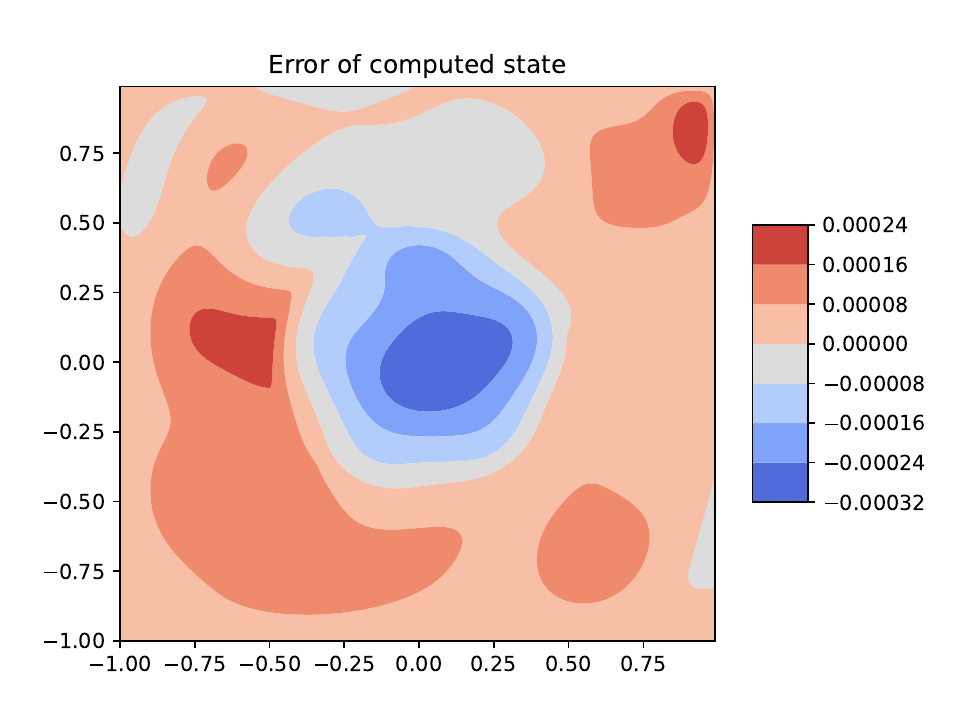}}
		\caption{Numerical results of the hard-constraint PINN method for \Cref{ex:parabolic}.}
	\label{fig:ex5-0.3-pinnhc-training-results}
	\end{figure}

\section{Conclusions and Perspectives}\label{sec:conclusion}

This paper explores the application of the physics-informed neural networks (PINNs) to optimal control problems subject to PDEs with interfaces and control constraints.  We first demonstrate that leveraged by the discontinuity capturing neural networks \cite{hu2022discontinuity}, PINNs can effectively solve such problems.  However, the boundary and interface conditions, along with the PDE, are treated as soft constraints by incorporating them into a weighted loss function. Hence, the boundary and interface conditions cannot be satisfied exactly and must be simultaneously learned with the PDE. This makes it difficult to fine-tune the weights and to train the neural networks, resulting in a loss of numerical accuracy.
To overcome these issues, we propose a novel neural network architecture designed to impose the boundary and interface conditions as hard constraints. The resulting hard-constraint PINNs guarantee both the boundary and interface conditions are satisfied exactly or with a high degree of accuracy, while being independent of the learning process for the PDEs. This hard-constraint approach significantly simplifies the training process and enhances the numerical accuracy. Moreover, the hard-constraint PINNs are mesh-free, easy to implement, scalable to different PDEs, and ensure rigorous satisfaction of the control constraints.
To validate the effectiveness of the proposed hard-constraint PINNs, we conduct extensive tests on various elliptic and parabolic interface optimal control problems.

Our work leaves some important questions for future research. For instance,
the high efficiency of the hard-constraint PINNs for interface optimal control problems emphasizes the necessity for convergence analysis and error estimate.
In the numerical experiments, we adopted fixed weights in the loss function and non-adaptive sampling methods for the training points, which may not be optimal. 
It is worth investigating adaptive weighting and sampling strategies, see e.g. \cite{hao2022physics,lu2021deepxde,mcclenny2020self,wang2021understanding,wu2023comprehensive}, to further improve the numerical accuracy of the hard-constraint PINNs.
In \Cref{subsec:parab-ocip}, we discuss the extension of the hard-constraint PINNs to parabolic interface optimal control problems, where the interface is assumed to be time-invariant.
A natural question is extending our discussions to the interfaces whose shape changes over time. Recall (\ref{eq:ij-gradient-comp}) that the interface-gradient condition $[\beta\partial_{\bm{n}}y]_{\Gamma}$ is still treated as a soft constraint.  
It is thus worth designing some more sophisticated neural networks such that this condition can also be imposed as a hard constraint and the numerical efficiency of PINNs can be further improved.  Finally, note that we focus on Dirichlet boundary conditions and it would be interesting to design some novel neural networks such that other types of boundary conditions (e.g., periodic conditions and Neumann  conditions) can be treated as hard constraints; and the ideas in \cite{lu2021physics} and \cite{sukumar2022exact} could be useful.
\newline

\bibliographystyle{siamplain}
\bibliography{references}

\end{document}